\documentclass[11pt,english,letterpaper]{article}
\usepackage{amsmath}
\usepackage{amsthm}
\usepackage{mathptmx}
\usepackage{helvet}
\usepackage{courier}
\usepackage{newtxmath}
\usepackage[T1]{fontenc}
\usepackage[utf8]{inputenc}
\usepackage{geometry}
\usepackage{color}
\usepackage{babel}
\usepackage{verbatim}
\usepackage{refstyle}
\usepackage{mathtools}
\usepackage{microtype}
\usepackage[unicode=true,pdfusetitle,
 bookmarks=true,bookmarksnumbered=false,bookmarksopen=false,
 breaklinks=true,pdfborder={0 0 0},pdfborderstyle={},backref=false,colorlinks=true]
 {hyperref}
\hypersetup{
 urlcolor=blue,linkcolor=blue,citecolor=blue}

\makeatletter

\AtBeginDocument{\providecommand\eqref[1]{\ref{eq:#1}}}
\AtBeginDocument{\providecommand\thmref[1]{\ref{thm:#1}}}
\AtBeginDocument{\providecommand\secref[1]{\ref{sec:#1}}}
\AtBeginDocument{\providecommand\subsecref[1]{\ref{subsec:#1}}}
\AtBeginDocument{\providecommand\lemref[1]{\ref{lem:#1}}}
\AtBeginDocument{\providecommand\defref[1]{\ref{def:#1}}}
\AtBeginDocument{\providecommand\propref[1]{\ref{prop:#1}}}
\AtBeginDocument{\providecommand\corref[1]{\ref{cor:#1}}}
\AtBeginDocument{\providecommand\enuref[1]{\ref{enu:#1}}}
\AtBeginDocument{\providecommand\appendixref[1]{\ref{appendix:#1}}}
\RS@ifundefined{subsecref}
  {\newref{subsec}{name = \RSsectxt}}
  {}
\RS@ifundefined{thmref}
  {\def\RSthmtxt{theorem~}\newref{thm}{name = \RSthmtxt}}
  {}
\RS@ifundefined{lemref}
  {\def\RSlemtxt{lemma~}\newref{lem}{name = \RSlemtxt}}
  {}

\numberwithin{equation}{section}
\numberwithin{figure}{section}
\numberwithin{table}{section}
\theoremstyle{plain}
\newtheorem{thm}{\protect\theoremname}[section]
\theoremstyle{plain}
\newtheorem{prop}[thm]{\protect\propositionname}
\theoremstyle{remark}
\newtheorem{rem}[thm]{\protect\remarkname}
\theoremstyle{plain}
\newtheorem{lem}[thm]{\protect\lemmaname}
\theoremstyle{definition}
\newtheorem{defn}[thm]{\protect\definitionname}
\theoremstyle{plain}
\newtheorem{cor}[thm]{\protect\corollaryname}

\usepackage{refstyle}
\newref{rem}{name=Remark~,names=Remarks~}
\newref{lem}{name=Lemma~,names=Lemmas~}
\newref{def}{name=Definition~,names=Definitions~}
\newref{thm}{name=Theorem~,names=Theorems~}
\newref{prop}{name=Proposition~,names=Propositions~}
\newref{cond}{name=Condition~,names=Conditions~}
\newref{cor}{name=Corollary~,names=Corollaries~}
\newref{sec}{name=Section~,names=Sections~}
\newref{fig}{name=Figure~,names=Figures~}
\newref{subsec}{name=Section~,names=Sections~}
\newref{appendix}{name=Appendix~,names=Appendices~}
\newref{eq}{name=,refcmd=(\ref{#1})}
\hypersetup{final}

\makeatother

\providecommand{\corollaryname}{Corollary}
\providecommand{\definitionname}{Definition}
\providecommand{\lemmaname}{Lemma}
\providecommand{\propositionname}{Proposition}
\providecommand{\remarkname}{Remark}
\providecommand{\theoremname}{Theorem}

\begin{document}
\global\long\def\Cov{\operatorname{Cov}}%

\global\long\def\Var{\operatorname{Var}}%

\global\long\def\Lip{\operatorname{Lip}}%

\global\long\def\e{\mathrm{e}}%

\global\long\def\R{\mathbf{R}}%

\global\long\def\Law{\operatorname{Law}}%

\global\long\def\dif{\mathrm{d}}%

\global\long\def\eps{\varepsilon}%

\newcommand{\email}[1]{{\href{mailto:#1}{\nolinkurl{#1}}}}
\title{A forward-backward SDE from the 2D nonlinear stochastic heat equation} 
\author{Alexander Dunlap\thanks{Department of Mathematics, Courant Institute of Mathematical Sciences, New York University, New York, NY 10012 USA. \email{alexander.dunlap@cims.nyu.edu}. Partially supported by the NSF Mathematical Sciences Postdoctoral Fellowship program under grant no.\ DMS-2002118, and by NSF grant no.\ DMS-1910023 and BSF grant no.\ 2014302 (both awarded to Lenya Ryzhik).} \and Yu
	Gu\thanks{Department of Mathematics, University of Maryland, College Park, MD 20742 USA. \email{ygu7@umd.edu}. Partially supported by the NSF through grants no.\ DMS-1907928 and CAREER-2042384.}}
\maketitle
\begin{abstract}
	We consider a nonlinear stochastic heat equation in spatial dimension
	$d=2$, forced by a white-in-time multiplicative Gaussian noise with spatial correlation
	length $\eps>0$ but divided by a factor of $\sqrt{\log\eps^{-1}}$.
	We impose a condition on the Lipschitz constant of the nonlinearity
	so that the problem is in the ``weak noise'' regime. We show that,
	as $\eps\downarrow0$, the one-point distribution of the solution
	converges, with the limit characterized
	in terms of the solution to a forward-backward stochastic differential equation (FBSDE). We also characterize the limiting multipoint statistics of
	the solution, when the points are chosen on appropriate scales, in
	similar terms. Our approach is new even for the linear case, in which the FBSDE can be solved explicitly and
	we recover results of Caravenna, Sun, and Zygouras (\emph{Ann.\ Appl.\ Probab.}\ 27(5):3050--3112, 2017).
\end{abstract}

\section{Introduction}

Fix a Lipschitz function $\sigma:[0,\infty)\to[0,\infty)$ with $\sigma(0)=0$.
Define $\beta=\Lip(\sigma)$. We are interested in the following two-dimensional
stochastic heat equation with colored noise of spatial correlation length
$\eps>0$, started at constant initial condition $a\in\mathbf{R}_{\ge 0}$:
\begin{align}
	\dif u_{\eps,a}(t,x) & =\frac{1}{2}\Delta u_{\eps,a}(t,x)\dif t+(\log\eps^{-1})^{-\frac{1}{2}}\sigma(u_{\eps,a}(t,x))\dif W^{\eps}(t,x),\quad t>0,x\in\mathbf{R}^{2};\label{eq:duxi} \\
	u_{\eps,a}(0,x)      & =a.\label{eq:uxiic}
\end{align}
Here we define $W^{\eps}=G_{\eps^{2}/2}*W$, where $G_{t}(x)=\frac{1}{2\pi t}\e^{-|x|^{2}/(2t)}$
is the two-dimensional heat kernel, $\dif W$ is a spacetime white
noise, and $*$ denotes convolution in space. The choice of mollifier is not essential, and we restrict to
this choice only to simplify some of the  computations. The covariance operator
of $\dif W^{\eps}$ is formally given by
\begin{equation}
	\mathbf{E}\dif W^{\eps}(t,x)\dif W^{\eps}(t',x')=\delta(t-t')G_{\eps^{2}}(x-x')=\delta(t-t')\tfrac{1}{\eps^2}G_1(\tfrac{x-x'}{\eps}).\label{eq:covarianceoperator}
\end{equation}
For $\eps>0$, the well-posedness of the initial value problem \eqref{duxi}--\eqref{uxiic}
is well-known (see e.g.~\cite{PZ97}), and we consider the mild formulation
\begin{equation}
	u_{\eps,a}(t,x)=a+\frac{1}{\sqrt{\log\eps^{-1}}}\int_{0}^{t}\int G_{t-s}(x-y)\sigma(u_{\eps,a}(s,y))\,\dif W^{\eps}(s,y).\label{eq:umild}
\end{equation}
General properties of solutions to the nonlinear stochastic heat equation
have previously been studied in general spatial dimensions by many
authors. We mention the non-exhaustive list of works \cite{Dal99,DQ11,CK19,CH19,CK20}.

We are interested in taking $\eps\downarrow0$ and identifying nontrivial
limiting behavior for the solutions of \eqref{duxi}--\eqref{uxiic}.
The linear problem, in which $\sigma(x)=\beta x$, is a particularly
important special case. Here it is known that the attenuating factor
$(\log\eps^{-1})^{-\frac{1}{2}}$ in \eqref{duxi} is required, and
that there is phase transition at $\beta=\sqrt{2\pi}$. The subcritical
linear problem ($\beta<\sqrt{2\pi}$) was previously studied in \cite{CSZ17}
(which we will discuss in more detail shortly), while the critical
linear problem ($\beta\approx\sqrt{2\pi}$) has been studied in \cite{BC98,CSZ19,GQT21,CSZ21}. It is worth mentioning that the notion of ``criticality'' here is different from the one in \cite[Section 8]{Hai14}.
In the linear case, the equation is related by the Cole--Hopf transform
to the two-dimensional KPZ equation, as considered in \cite{CD20,CSZ20,Gu20}.
The linear problem also admits a Feynman--Kac formula \cite{BC95}
and thus a connection to directed polymers, with the solution to the SPDE  interpreted as the partition function of directed polymers in random environment. The Feynman--Kac representation
has proved to be very useful in analyzing properties of the solutions, but is
not available in the nonlinear case.
In \cite{CSZ17}, Caravenna, Sun, and Zygouras showed  that if $\sigma(x)=\beta x$, $\beta\in(0,\sqrt{2\pi})$, then
for any fixed $T>0$ and $X\in\mathbf{R}^{2}$, $u_{\eps,a}(T,X)$
converges in distribution as $\eps\downarrow0$ to a log-normal random
variable. Their proof used the Feynman--Kac formula to connect the
problem to directed polymers, and then worked to understand a polynomial
chaos expansion in great detail.

The goal of the present paper is
to study the nonlinear case in which many previously-used tools
are not available.  We will show in \thmref{convergence} below that
if $\sigma$ is $\beta$-Lipschitz, $\beta\in(0,\sqrt{2\pi})$, then
$u_{\eps,a}(T,X)$ converges in distribution as $\eps\downarrow0$.
The limit depends on $\sigma$ and is obtained through the solution of a forward-backward stochastic differential equation.
Our method is also new in the linear case. In the nonlinear case, the limit does not seem to be log-normal in general.

Part of the reason we are interested in such a problem comes from the recent progress in proving the Edwards-Wilkinson limit of the  KPZ equation \cite{CD20,CSZ20,Gu20,MU18,DGRZ20,LZ20,CNN20}  in $d\geq2$. Most of these results rely on the Cole--Hopf transformation which, in some sense, linearizes the problem so that one can focus on studying the linear stochastic heat equation (as in \cite{DS80,TZ98,MSZ16,GRZ18,DGRZ18a,LZ20,CNN20}) and how its solution behaves after the logarithmic transformation. For general Hamilton--Jacobi type equations, this linearization does not exist and there are no results of this type. (See a conjecture in \cite[p.~5]{HQ18} and some related directions for the anisotropic KPZ equation in \cite{CES19,CET20,CET21}.) We hope that working on the nonlinear stochastic heat equation can help bridge the difficulty and shed light on other nonlinear problems such as the Hamilton--Jacobi equation. A similar effort in $d\ge 3$ was carried out in \cite{GL20}. The convergence to Edwards-Wilkinson equation in $d\geq2$ is as random Schwartz distributions, which, in our case, corresponds to the convergence in distribution of the random variable
\[
	\sqrt{\log\eps^{-1}}\int [u_{\eps,a}(T,x) -a]g(x) \dif x
\] for Schwartz test function $g$. The limiting marginal distributions of $u_{\eps,a}$ play an important role in passing to the limit of the above random variable, which we will discuss in more detail below in \remref{ew}.

In order to state our main result (\thmref{convergence} below) precisely,
we first have to define the limit object.
Let
$\{B(q)\}_{q\ge0}$ be a 1D standard Brownian motion with the natural filtration
$\{\mathcal{G}_{q}\}_{q\ge0}$. We consider the following system of
equations, satisfied by $\{\Xi_{a,Q}(\cdot)\}_{a,Q}$, with the parameters $a\ge0$ and $Q\in[0,2]$:
\begin{align}
	\dif\Xi_{a,Q}(q) & =J(Q-q,\Xi_{a,Q}(q))\dif B(q),\qquad q\in(0,Q];\label{eq:dXi-intro}                  \\
	\Xi_{a,Q}(0)     & =a;\label{eq:Xiic-intro}                                                             \\
	J(q,b)           & =\frac{1}{2\sqrt{\pi}}[\mathbf{E}\sigma^2(\Xi_{b,q}(q))]^{1/2}.\label{eq:Jdef-intro}
\end{align}
The parameter $a$ plays the role of initial data, $Q$ is the terminal time, and the above equation can be interpreted as follows: for the process started at $a$ with the terminal time $Q$, to determine the diffusion coefficient at any time $q\in [0,Q]$, we run an independent process, starting from the current position $b=\Xi_{a,Q}(q)$ and with terminal time $Q-q$. The new process at time $Q-q$ is distributed like $\Xi_{b,Q-q}(Q-q)$. Then the square of the diffusion coefficient for the original process, at time $q$, is given by the expectation of $\frac{1}{4\pi}\sigma^2(\Xi_{b,Q-q}(Q-q))$.
We emphasize that a solution to \eqref{dXi-intro}--\eqref{Jdef-intro}
consists of both a family of random processes
$\{\Xi_{a,Q}(\cdot)\}_{a\ge0,Q\in[0,2]}$
and also a deterministic function $J:[0,2]\times\mathbf{R}_{\ge0}\to\mathbf{R}_{\ge0}$.
That is, $J$ is not given as part of the data of the problem but
is rather found as part of the solution. Probabilistically, the processes
$\Xi_{a,Q}$ are not coupled in any particular way across various
choices of $a$ and $Q$: each $\Xi_{a,Q}$ could be taken to live
on a different probability space. However, their \emph{laws} are related
through the deterministic function $J$.

We note that another, equivalent, way to write the system \eqref{dXi-intro}--\eqref{Jdef-intro}
is as
\begin{align}
	\dif\Xi_{a,Q}(q) & =\frac{1}{2\sqrt{\pi}}\big(\mathbf{E}[\sigma^2(\Xi_{a,Q}(Q))\mid\mathcal{G}_{q}]\big)^{1/2}\dif B(q),\qquad q\in(0,Q];\label{eq:dXi-intro-noJ} \\
	\Xi_{a,Q}(0)     & =a.\label{eq:Xiic-intro-noJ}
\end{align}
The formulation \eqref{dXi-intro-noJ}--\eqref{Xiic-intro-noJ} is essentially a forward-backward stochastic differential equation (FBSDE). Fixing $a\geq 0$ and $Q\in[0,2]$, we consider the process $\{(X(q),Y(q),Z(q))\}_{q\in[0,Q]}$, with all components adapted to the filtration $\{\mathcal{G}_q\}_{q\geq0}$, satisfying the coupled forward-backward stochastic differential equation
\begin{align}
	 & \dif X(q)=\sqrt{Y(q)} \dif B(q),  \quad \quad X(0)=a, \label{eq:y1}                   \\
	 & \dif Y(q)=Z(q)\dif B(q),  \quad\quad Y(Q)=\frac{1}{4\pi}\sigma^2(X(Q)). \label{eq:y2}
\end{align}
Here the equation for $X(\cdot)$ is forward since the initial condition is given, and the equation for $Y(\cdot)$ is backward since the terminal condition is given. Because $Y$ is supposed to be a martingale with terminal value $\frac{1}{4\pi}\sigma^2(X(Q))$, we actually have $Y(q)=\frac{1}{4\pi}\mathbf{E}[\sigma^2(X(Q))\mid\mathcal{G}_q]$. As a result, $X(\cdot)$ solves the same equation as $\Xi_{a,Q}(\cdot)$.

In the FBSDE formulation, the auxiliary function $J$ (called a ``decoupling function'' in the FBSDE literature \cite{MPY94,MWZZ15,Fro14}) is not required, although
it can be recovered from \eqref{dXi-intro-noJ} by \eqref{Jdef-intro}.
The formulations \eqref{dXi-intro-noJ}--\eqref{Xiic-intro-noJ} and \eqref{dXi-intro}--\eqref{Jdef-intro} are equivalent because the law of $\Xi_{a,Q}(Q)$
conditional on $\Xi_{a,Q}(q)=b$ is the same as the law of $\Xi_{b,Q-q}(Q-q)$. We similarly note that a solution to \eqref{y1}--\eqref{y2} will satisfy $Y(q)=J^2(Q-q,X(q))$.
The formulation \eqref{dXi-intro}--\eqref{Jdef-intro} turns out to be easier
to work with, since one can first solve for the deterministic decoupling function
$J$, and once $J$ is known the problem \eqref{dXi-intro}--\eqref{Xiic-intro}
becomes a standard stochastic differential equation. We refer the reader to, for example, \cite{MY99} for background on FBSDEs. We also point out that the function $J^2(q,b)$ is a viscosity solution to the quasilinear heat equation
\begin{align}
	\partial_q J^2 & =\frac{1}{2}J^2\partial_{bb}J^2; \\
	J^2(0,b)       & =\frac{1}{4\pi}\sigma^2(b),
\end{align}
as can be seen by an argument similar to that of \cite[Section 8.2]{MY99}, using the moment bound in \remref{SDEmoments} below.

The non-Lipschitz dependence of \eqref{y1} on $Y$, as well as the potentially quadratic growth of $\sigma^2$ at infinity, exclude the system \eqref{y1}--\eqref{y2} from the established well-posedness theories for FBSDEs, discussed in \cite{MY99,MWZZ15}. Nonetheless, we can prove the following well-posedness result.
\begin{thm}
	\label{thm:wellposedness}If $\beta<\sqrt{2\pi}$, then there is a
	unique continuous function $J:[0,2]\times\mathbf{R}_{\ge0}\to\mathbf{R}_{\ge0}$
	satisfying the following conditions:
	\begin{enumerate}
		\item \label{enu:Jconds}For each $q\in[0,2]$, $J(q,\cdot)$ is Lipschitz,
		      \begin{equation}
			      J(q,0)=0,\label{eq:Jq0}
		      \end{equation}
		      and
		      \begin{equation}
			      \Lip J(q,\cdot)\le(4\pi/\beta^{2}-q)^{-1/2}.\label{eq:hLipschitz-1}
		      \end{equation}
		\item \label{enu:Jsolves}For each $a\ge0$ and $Q\in[0,2]$, the solution
		      $\{\Xi_{a,Q}(q)\}_{0\le q\le Q}$ to the problem \eqref{dXi-intro}--\eqref{Xiic-intro}
		      (with this choice of $J$) satisfies $\frac{1}{2\sqrt{\pi}}(\mathbf{E}\sigma(\Xi_{a,Q}(Q))^{2})^{1/2}=J(Q,a)$.
		      In other words, \eqref{Jdef-intro} is satisfied with $q=Q$ and $b=a$.
	\end{enumerate}
\end{thm}

The proof of \thmref{wellposedness} is given in \secref{well-posedness}.
Now that we have established existence and uniqueness of solutions
to \eqref{dXi-intro}--\eqref{Jdef-intro}, in the sense of \thmref{wellposedness}, we can state our main
theorem.
\begin{thm}
	\label{thm:convergence}If $\beta<\sqrt{2\pi}$, then for any $Q\in[0,2]$ and $X\in\mathbf{R}^2$, we have
	\begin{equation}
		u_{\eps,a}(\eps^{2-Q},X)\xrightarrow[\eps\downarrow0]{\mathrm{law}}\Xi_{a,Q}(Q),\label{eq:uepsaQ}
	\end{equation}
	where $\Xi_{a,Q}$ comes from the solution to \eqref{dXi-intro}--\eqref{Jdef-intro}. For any fixed
	$T>0$ and $X\in\mathbf{R}^{2}$ we have
	\begin{equation}
		u_{\eps,a}(T,X)\xrightarrow[\eps\downarrow0]{\mathrm{law}}\Xi_{a,2}(2).\label{eq:maintheorem-convergeinlaw}
	\end{equation}
\end{thm}

The constant $2$ appearing (twice) in \eqref{maintheorem-convergeinlaw}
comes from the fact that, for fixed $T>0$, the time variables $q$ and $t$, corresponding to the ODE \eqref{dXi-intro} and the PDE \eqref{duxi} respectively, are (informally) related by
\[
	t=T-\eps^{q}.
\]
This is related to the fact that the noise contributes to the solution
on this $\eps$-dependent exponential scale, as we discuss more in \subsecref[s]{expscale} and~\ref{subsec:proofsketch} below. The terminal time $2$ corresponds
to the $G_{\eps^{2}}$ in the correlation function \eqref{covarianceoperator}
for the noise: the mollification cuts off the dynamics below this scale. 

Of course, even deterministic ODEs are not generally integrable in
elementary terms, so we do not expect to be able to solve the system
\eqref{dXi-intro}--\eqref{Jdef-intro} explicitly for general $\sigma$.
However, in the linear case $\sigma(u)=\beta u$, the system can indeed
be solved explicitly. In that case, we recover the log-normal fluctuations
proved in \cite{CSZ17}.
We show how to do this in \subsecref{linearcase} below.

The work \cite{CSZ17} also dealt with limiting multipoint statistics
of solutions to \eqref{duxi}--\eqref{uxiic} with $\sigma(x)=\beta x$.
It turns out that $u_{\eps,a}(t_{1},x_{1})$ and $u_{\eps,a}(t_{2},x_{2})$
are asymptotically independent if
\begin{equation}
	d((\tau_{1},x_{1}),(\tau_{2},x_{2}))\coloneqq\max\{|t_{1}-t_{2}|^{1/2},|x_{1}-x_{2}|\}\label{eq:ddef}
\end{equation}
is of order $1$. To see a nontrivial correlation structure, we must
put $t_{2}=t_{1}+\eps^{\alpha}$ and $x_{2}=x_{1}+\eps^{\beta}$ for
some $\alpha,\beta>0$. This situation persists in the nonlinear case,
and we can express the limiting joint laws of multiple points separated
on these scales by a branching version of the ODE \eqref{dXi-intro}--\eqref{Xiic-intro},
as we state in the following theorem. Note that once $J$ has been
obtained from the single-point problem \eqref{dXi-intro}--\eqref{Jdef-intro},
it is no longer necessary to consider \eqref{Jdef-intro} in the multipoint
problem: $J$ is then simply a fixed deterministic function, depending
only on $\sigma$.
\begin{thm}
	\label{thm:multipoint}Suppose that $\beta<\sqrt{2\pi}$. Let $N\in\mathbf{N}$
	and fix $N$ space-time points $(\tau_{\eps,1},x_{\eps,1}),\ldots,(\tau_{\eps,N},x_{\eps,N})\in\mathbf{R}_{>0}\times\mathbf{R}^{2},$
	 depending on $\eps$. Define the metric
	$d$ as in \eqref{ddef}. Suppose that
	\begin{equation}
		d_{ij}\coloneqq1-\lim_{\eps\downarrow0}\log_{\eps}d((\tau_{\eps,i},x_{\eps,i}),(\tau_{\eps,j},x_{\eps,j}))\label{eq:dijdef}
	\end{equation}
	exists %
	for all $i,j$, and suppose
	that
	\begin{equation}
		Q\coloneqq2-\lim_{\eps\downarrow0}\log_{\eps}\tau_{\eps,j}\label{eq:Qdef-1}
	\end{equation}
	exists, is independent of $j$, and is at most $2$. Define
	\begin{equation}
		i_{q}(j)=\min\{i\in\{1,\ldots,N\}\ :\ d_{ij}<q\}.\label{eq:iqdef}
	\end{equation}
	Let $J$ be as in the solution to \eqref{dXi-intro}--\eqref{Jdef-intro}.
	Let $B_{1},\ldots,B_{N}$ be a family of $N$ independent standard
	Brownian motions. For $a\in\mathbf{R}$, let $(\Gamma_{a,Q,j})_{j=1}^{N}$
	solve the family of SDEs
	\begin{align}
		\dif\Gamma_{a,Q,j}(q) & =J(Q-q,\Gamma_{a,Q,j}(q))\dif B_{i_{(Q-q)/2}(j)}(q),\qquad j\in\{1,\ldots,N\};\label{eq:dGamma-intro} \\
		\Gamma_{a,Q,j}(0)     & =a.\label{eq:Gammaic-intro}
	\end{align}
	Then we have
	\begin{equation}
		(u_{\eps,a}(\tau_{\eps,j},x_{\eps,j}))_{j=1}^{N}\xrightarrow[\eps\downarrow0]{\mathrm{law}}(\Gamma_{a,Q,j}(Q))_{j=1}^{N}.\label{eq:multipointconvergence}
	\end{equation}
\end{thm}

The quantity $d_{ij}$ represents the distance between $(\tau_{\eps,i},x_{\eps,i})$
and $(\tau_{\eps,j},x_{\eps,j})$ on the exponential scale. Of particular
note here is the ultrametricity property
\begin{equation}
	d_{ik}\le\max\{d_{ij},d_{jk}\}\label{eq:strongtriangleinequality}
\end{equation}
for all $i,j,k\in\{1,\ldots,N\}$. If one restricts to a single point ($N=1$)
then it is of course clear that \eqref{dGamma-intro}--\eqref{Gammaic-intro}
agrees with \eqref{dXi-intro}--\eqref{Xiic-intro}. For two points, if we consider $\tau_{\eps,1}=\tau_{\eps,2}=T>0$ independent of $\eps$ and $|x_{\eps,1}-x_{\eps,2}|=\eps^{\alpha}$ with some $\alpha\in[0,1]$, then $Q=2$, $d_{11}=d_{22}=-\infty$ , $d_{12}=1-\alpha$, and it is clear that $\Xi_{a,Q,1}$ is driven by $B_1$ in $[0,2]$, while $\Xi_{a,Q,2}$ is driven by $B_1$ in $[0,2\alpha]$ and by $B_2$ in $[2\alpha,2]$. Two extreme cases are $\alpha=0$ and $\alpha=1$, in which $\Xi_{a,Q,1}$ and $\Xi_{a,Q,2}$ are independent and identical respectively. In the general case,
we note that the set $\{i_{(Q-q)/2}(j): j\in \{1,\ldots,N\}\}$ only grows larger
as $q$ increases. Therefore, the members of the family of SDEs \eqref{dGamma-intro}--\eqref{Gammaic-intro}
will generally start stuck together and then branch apart at times
$q$ such that $1-\tfrac{q}{2}=d_{ij}$ for some $i,j\in\{1,\ldots,N\}$. Thus
we obtain a multiscale correlation structure generalizing the one obtained for the linear case
in \cite[Theorem 2.15 and (2.18)]{CSZ17}. %
In \subsecref{linearcase} we show how to recover \cite[(2.18)]{CSZ17} from \thmref{multipoint} in the linear case.

\subsection{The exponential time scale\label{subsec:expscale}}
A key feature of the SPDE \eqref{duxi}--\eqref{uxiic} is that, in the subcritical regime $\beta<\sqrt{2\pi}$, it evolves on an exponential time scale, with respect to the strength of the random noise. To see this, consider the following equation in microscopic variables:
\[
	\dif u_a(t,x)=\frac12\Delta u_a(t,x)\dif t+\delta \,\sigma(u_a(t,x))\dif W^1(t,x),\quad\quad u_a(0,\cdot)\equiv a,
\]
with $\dif W^1$ the Gaussian noise that is white in time and smooth in space (the spatial covariance function being $G_1$ by \eqref{covarianceoperator}), and $\delta>0$ a fixed small parameter.  We are interested in determining the scales on which nontrivial effects from the random noise can be observed. As expected, it depends on the dimension through the integrability of the heat kernel.

In $d=1$, the correct scale turns out to be  $(t,x)=(\frac{T}{\delta^4},\frac{X}{\delta^2})$, where $(T,X)$ are the corresponding macroscopic variables, as discussed for directed polymers in \cite{AKQ14} and for SPDEs in \cite{BC95,HP15}. In $d\geq3$, if $\delta \beta$ is small enough so that the problem is in the weak disorder regime, one can consider an ``arbitrarily long'' diffusive scale $(t,x)=(\frac{T}{\eps^2},\frac{X}{\eps})$ with $\eps\to0$ independent of $\delta$. The $d=2$ case is different and very special. In $d=2$, the second moment $f_a(t):=\mathbf{E}u_a(t,x)^2$ approximately satisfies the Volterra equation
\[
		f_a(t)\approx a^2+\frac{\delta^2\beta^2}{4\pi}\int_0^t \frac{f_a(s)}{t-s+\frac12}ds.
\]
One can easily analyze the asymptotic behavior of $f_a(t)$ for large $t$ and small $\delta$:
\[f_a(t)\approx \frac{a^2}{1-\frac{\delta^2\beta^2\log t}{4\pi}}, \quad \text{ if } \frac{\delta^2\beta^2}{4\pi}\log t<1.
\] Due to the dependence on $\log t$, to see a nontrivial evolution, one should consider an exponential time scale and let $t=\e^{Q/\delta^2}$ with $Q\leq 2$. (We used $Q$ rather than $T$ as the macroscopic variable here, to emphasize this is on the \emph{exponential} scale.) This exponential time scale was previously observed for the linear case $\sigma(u)=\beta u$ in \cite{CSZ17}. For $\delta=(\log\eps^{-1})^{-\frac12}$, the scale becomes $t=\eps^{-Q}$.  On the other hand, by the scaling property of the white noise, one can easily check that, in $d=2$, we have
\[
	u_{\eps,a}(\cdot,\cdot)\overset{\mathrm{law}}{=} u_a(\tfrac{\cdot}{\eps^2},\tfrac{\cdot}{\eps}), \quad\quad \text{ if } \delta=(\log\eps^{-1})^{-\frac12}.
\]
Thus, $u_{\eps,a}(\eps^{2-Q},0)\overset{\mathrm{law}}{=} u_a(\eps^{-Q},0)$, and from this perspective, it is natural to consider the scaling used in  \eqref{uepsaQ}, which says that for any macroscopic variable $Q\in[0,2]$, we have
\[
	u_a(\eps^{-Q},0)\xrightarrow[\eps\downarrow0]{\mathrm{law}}\Xi_{a,Q}(Q).
\]

\subsection{Sketch of the proof\label{subsec:proofsketch}}

The proof of \thmref{convergence} begins with a series of approximations
of the SPDE \eqref{duxi}--\eqref{uxiic}. Fix $T>0,X\in\mathbf{R}^{2}$.  The underlying phenomenology
behind these approximations is that the contribution  of the noise $\dif W^{\eps}$ on an interval $[T-\eps^{q},T-\eps^{q+\gamma}]$ to the $L^2$ norm of the solution
$u_{\eps,a}(T,X)$ can be bounded from above by $\gamma^{1/2}$. Therefore, we can ``turn off'' the noise
on   intervals $[T-\eps^{q_{i}},T-\eps^{q_{i}+\gamma}]$,
$i=1,\ldots,M$, and as long as $M\gamma^{1/2}\ll1$, this will not change $u_{\eps,a}(T,X)$ in the limit. (We describe precisely
how we choose these increments at the beginning of \secref{timediscretization}.)
For any $A\subset[0,\infty)$, we define $u_{\eps,a}^{A}$
as the solution to
\begin{align}
	\dif u_{\eps,a}^{A}(t,x) & =\frac{1}{2}\Delta u_{\eps,a}^{A}(t,x)\dif t+\frac{\mathbf{1}_{\mathbf{R}\setminus A}(t)}{\sqrt{\log\eps^{-1}}}\sigma(u_{\eps,a}^{A}(t,x))\dif W^{\eps}(t,x);\label{eq:uA} \\
	u_{\eps,a}^{A}(0,x)      & =a.\label{eq:uAic}
\end{align}
This comes from the problem \eqref{duxi}--\eqref{uxiic} by ``turning
off'' the noise on the set  $A$. \secref{shutoffnoise-oneinterval}
is devoted to bounding the error incurred by turning off the noise
on an interval.

Let $\tilde{u}_{\eps,a}=u_{\eps,a}^{A}$, with $A=\bigcup_{i=1}^{M}[T-\eps^{q_{i}},T-\eps^{q_{i}+\gamma}]$,
denote the solution with the noise turned off in this way. Fix any $i=1,\ldots,M$. Since we
expect the problem to have a diffusive scaling, $\tilde{u}_{\eps,a}(T-\eps^{q_{i}+\gamma},x)$
should contribute to $u_{\eps,a}(T,X)$ only for those $x$ such that $|x-X|\lesssim\eps^{(q_{i}+\gamma)/2}$. We further choose $\gamma$ so that $\eps^\gamma \ll 1$.
The noise is turned off on the interval $[T-\eps^{q_{i}},T-\eps^{q_{i}+\gamma}]$,
so $\tilde{u}_{\eps,a}(T-\eps^{q_{i}+\gamma},\cdot)$ has been subject
to the deterministic heat equation (with no noise) for the last $T-\eps^{q_{i}+\gamma}-(T-\eps^{q_{i}})=\eps^{q_i}(1-\eps^\gamma)\approx\eps^{q_{i}}$
amount of time, and thus is essentially constant on spatial scales much smaller
than $\eps^{q_{i}/2}$. Thus, since $\eps^{\gamma}\ll 1$ and thus $\eps^{(q_{i}+\gamma)/2}\ll\eps^{q_{i}/2}$, the main contribution of noise up until time $T-\eps^{q_{i}+\gamma}$ on $u_{\eps,a}(T,X)$ is via the constant $\tilde{u}_{\eps,a}(T-\eps^{q_{i}+\gamma},X)$.
\secref{smoothedfield} is devoted to bounding the error incurred
by replacing the field by a (random) constant after the solution has been subject
to the deterministic heat equation for some time. In \secref{timediscretization},
we define the time discretization that we use, and then
iterate the results of \secref[s]{shutoffnoise-oneinterval} and~\ref{sec:smoothedfield}
to bound the total error incurred by this approximation scheme.

Our approximation scheme approximates the solution $u_{\eps,a}(T,X)$
in terms of a scalar-valued Markov chain whose $i$th value is $\tilde{u}_{\eps,a}(T-\eps^{q_{i}+\gamma},X)$. (Since the equation starts from constant initial data and we are interested in the marginal distribution, by space-stationarity, the choice of $X$ is arbitrary and plays no role.) This Markov chain, which
is also a discrete martingale, will approximate the solution to \eqref{dXi-intro}--\eqref{Jdef-intro}.
To see why, we note that step $(i+1)$ of the Markov chain is given by
solving the original equation \eqref{duxi}--\eqref{uxiic} with the
initial condition $a$ equaling to the current value of the Markov chain, which is $\tilde{u}_{\eps,a}(T-\eps^{q_{i}+\gamma},X)$,
on an interval of length $\eps^{q_{i}+\gamma}-\eps^{q_{i+1}}\approx{\eps^{q_{i}+\gamma}}$,
and then letting the solution evolve according to the heat equation
for time $\eps^{q_{i+1}}-\eps^{q_{i+1}+\gamma}\approx\eps^{q_{i+1}}$.
Although it only represents one step of the Markov chain, approximating
the solution on these time scales require running another instance
of the Markov chain for $M-i$ steps. This is a consequence of the mild solution formula; see \lemref{Vtildegoodapprox} below. This corresponds to the $Q-q$
in the argument of $J$ in \eqref{dXi-intro}. On the other hand,
since this only represents one step of the Markov chain, one only
needs to understand the variance rather than the complete law in order to
compute the diffusivity of the limiting diffusion. Accounting for
the averaging from the heat equation (which gives us a factor of $q_{i}-q_{i-1}$), it turns out that this variance
is approximated by the expression on the right side of \eqref{Jdef-intro}
in the limit. In particular, the fact that only the variance is important
is reflected in the fact that an expectation is taken on the right
side of \eqref{Jdef-intro}. Making these ideas precise is the main
task of \secref{discretemartingale}.

The fact that the diffusion coefficient of the limiting SDE can be
represented in terms of statistics of the chain itself is of course
critical to proving the existence of the limit. The fact that the
self-similar structure characterizes the limit is reflected in the
fact that the problem \eqref{dXi-intro}--\eqref{Jdef-intro} is
well-posed, as stated in \thmref{wellposedness}. This well-posedness
allows us to construct the limiting diffusion coefficient and then
show that the Markov chain converges to the diffusion using standard
techniques. This is the content of \secref{convergence}.

We address multipoint statistics, and prove \thmref{multipoint},
in \secref{multipoint}. At this stage, since the problem \eqref{dXi-intro}--\eqref{Jdef-intro}
has been solved, the function $J$ has been identified. The Markov
chains corresponding to multiple points stay together at earlier times,
but then eventually branch apart from each other as the remaining
time scale approaches the spatial separation of the points. It turns
out that once they branch apart, they are completely independent in
the limit. This yields the branching diffusion structure \eqref{dGamma-intro}--\eqref{Gammaic-intro}.

\subsection{The linear case\label{subsec:linearcase}}

In this subsection, we consider the linear case $\sigma(u)=\beta u$
and show that solutions to \eqref{dXi-intro}--\eqref{Jdef-intro}
have log-normal one-point statistics, and moreover that we recover
the limiting variance \cite[(2.18)]{CSZ17} obtained in \cite[Theorem 2.15]{CSZ17}.
In this case, the linearity of the problem \eqref{dXi-intro}--\eqref{Jdef-intro}
allows us to make the \emph{ansatz} $J(q,b)=b\underline{J}(q)$, with
$\underline{J}(q)=J(q,1)$. Then the problem becomes
\begin{align}
	\dif\Xi_{a,Q}(q) & =\underline{J}(Q-q)\Xi_{a,Q}(q)\dif B(q),\qquad q\in[0,Q];\label{eq:dXi-intro-linear}     \\
	\Xi_{a,Q}(0)     & =a;\label{eq:Xiic-intro-linear}                                                           \\
	\underline{J}(q) & =\frac{\beta}{2\sqrt{\pi}}(\mathbf{E}\Xi_{1,q}(q)^{2})^{1/2}.\label{eq:Jdef-intro-linear}
\end{align}
We can already see that (up to a time-change determined by $\underline{J}$)
the problem \eqref{dXi-intro-linear}--\eqref{Xiic-intro-linear}
is solved by a geometric Brownian motion. It turns out that we can
compute $\underline{J}$ explicitly. By Itô's formula applied to \eqref{dXi-intro-linear}
we have
\begin{equation}
	\dif(\log\Xi_{a,Q})(q)=\underline{J}(Q-q)\dif B(q)-\frac{1}{2}\underline{J}(Q-q)^{2}\dif q,\label{eq:logXi}
\end{equation}
and hence
\begin{equation}
	\Xi_{a,Q}(Q)=a\exp\left\{ \int_{0}^{Q}\underline{J}(Q-q)\,\dif B(q)-\frac{1}{2}\int_{0}^{Q}\underline{J}(Q-q)^{2}\,\dif q\right\} .\label{eq:lognormalsolution}
\end{equation}
Taking $a=1$, substituting \eqref{lognormalsolution} into \eqref{Jdef-intro-linear},
and computing the expectation, we obtain
\begin{align*}
	\underline{J}(Q)^{2} & =\frac{\beta^{2}}{4\pi}\exp\left\{ \int_{0}^{Q}\underline{J}(q)^2\,\dif q\right\} .
\end{align*}
Differentiating this expression gives us the differential equation
$\frac{\dif}{\dif Q}\underline{J}(Q)^{2}=\underline{J}(Q)^{{4}}.$
Combining this with the initial condition $\underline{J}(0)=\frac{\beta}{2\sqrt{\pi}}$,
which is evident from \eqref{Xiic-intro} and \eqref{Jdef-intro},
we obtain
\begin{equation}
	\underline{J}(Q)=(4\pi/\beta^{2}-Q)^{-1/2}.\label{eq:underlineJ}
\end{equation}
Note that the resulting $J$, given by
\begin{equation}
 J(q,b) = \frac{b}{\sqrt{4\pi/\beta^2-q}},
\end{equation}
saturates the bound \eqref{hLipschitz-1}.
Substituting \eqref{underlineJ} into \eqref{lognormalsolution},
we have
\begin{equation}
\begin{aligned}
	\Xi_{a,Q}(Q)&=a\exp\left\{ \int_{0}^{Q}\frac{1}{\sqrt{4\pi/\beta^{2}-(Q-q)}}\,\dif B(q)-\frac{1}{2}\int_{0}^{Q}\frac{\dif q}{4\pi/\beta^{2}-(Q-q)}\right\}\\& \overset{\mathrm{law}}{=}a\exp\left\{ S-\frac{1}{2}\mathbf{E}S^{2}\right\} ,\end{aligned}\label{eq:lognormalsolution-1}
\end{equation}
where $S\sim N(0,  {\log}\frac{4\pi/\beta^{2}}{4\pi/\beta^{2}-Q})$. In the
case $Q=2$ and $a=1$, this agrees with the expression \cite[(2.12)]{CSZ17}.

Now we address the multipoint statistics, i.e. the problem \eqref{dGamma-intro}--\eqref{Gammaic-intro}.
As in \eqref{logXi}, but now knowing \eqref{underlineJ}, we have
\[
	\dif(\log\Gamma_{a,Q,j})(q)=\frac{\dif B_{i_{(Q-q)/2}(j)}(q)}{\sqrt{4\pi/\beta^{2}-(Q-q)}}-\frac{\dif q}{8\pi/\beta^{2}-2(Q-q)}.
\]
From this linear SDE we see that the family $(\log\Gamma_{a,Q,j}(Q))_{j=1}^{N}$
is jointly Gaussian. All of the means are equal as
\[
	\mathbf{E}[\log\Gamma_{a,Q,j}(Q)]=\log a-\frac{1}{2}\int_{0}^{Q}\frac{\dif q}{4\pi/\beta^{2}-(Q-q)}=\log a-\frac{1}{2}\log\frac{4\pi/\beta^{2}}{4\pi/\beta^{2}-Q}
\]
as in \eqref{lognormalsolution-1}. The covariance structure is given
by
\begin{equation}
	\begin{aligned}\Cov&(\log\Gamma_{a,Q,i}(Q),\log\Gamma_{a,Q,j}(Q)) \\& =\int_{\{q\in[0,Q]\ :\ i_{(Q-q)/2}(i)=i_{(Q-q)/2}(j)\}}\frac{\dif q}{4\pi/\beta^{2}-(Q-q)}                            \\
		                                                  & =\int_{[0,Q-2d_{ij}\vee 0]}\frac{\dif q}{4\pi/\beta^{2}-(Q-q)}=\log\frac{4\pi/\beta^{2}-(2d_{ij}\vee 0)\wedge Q}{4\pi/\beta^{2}-Q}.
	\end{aligned}
	\label{eq:CovGamma}
\end{equation}
The second equality is by the ultrametricity property \eqref{strongtriangleinequality}
of the $d_{ij}$s. For $Q=2$, \eqref{CovGamma} is the same as the
covariance structure \cite[(2.18)]{CSZ17} obtained in \cite[Theorem 2.15]{CSZ17}.

\section{Proof of Theorem~\ref{thm:wellposedness}\label{sec:well-posedness}}

In this section we prove \thmref{wellposedness}, establishing the
well-posedness of the limiting problem. The analysis here is essentially
independent of the rest of the paper.
\begin{proof}[Proof of \thmref{wellposedness}.]
	If $g:[0,Q]\times\mathbf{R}_{\ge0}\to\mathbf{R}_{\ge0}$ is continuous, is
	Lipschitz in the second variable, and satisfies $g(\cdot,0)\equiv0$,
	then for each $a\ge0$ and $Q\in[0,2]$ we let $\Xi_{a,Q}^{g}$ solve
	the problem
	\begin{align}
		\dif\Xi_{a,Q}^{g}(q) & =g(Q-q,\Xi_{a,Q}^{g}(q))\dif B(q);\label{eq:dXig} \\
		\Xi_{a,Q}^{g}(0)     & =a.\label{eq:Xigic}
	\end{align}
	It is standard that \eqref{dXig}--\eqref{Xigic} has a unique strong
	solution with continuous sample paths almost surely, and that this
	solution is positive with probability $1$. (For the last property
	see e.g.~\cite[Lemma 2.1]{Mao94}.) We write \eqref{dXig}--\eqref{Xigic}
	in the mild formulation
	\[
		\Xi_{a,Q}^{g}(q)=a+\int_{0}^{q}g(Q-s,\Xi_{a,Q}^{g}(s))\,\dif B_{s}.
	\]
	Define
	\[
		\mathcal{Q}g(Q,a)=\frac{1}{2\sqrt{\pi}}(\mathbf{E}\sigma(\Xi_{a,Q}^{g}(Q))^{2})^{1/2}.
	\]
	We note that $J$ satisfies the condition~\ref{enu:Jsolves} in the
	statement of the theorem if and only if $\mathcal{Q}J=J$. We will
	show that there is a unique such fixed point $J$ under the additional
	assumption that condition~\ref{enu:Jconds} in the statement of the
	theorem is satisfied.

	To this end, let $\mathcal{X}$ be the Banach space of continuous
	functions $f:\mathbf{R}_{\ge0}\to\mathbf{R}$ such that $f(0)=0$
	and the norm
	\[
		\|f\|_{\mathcal{X}}=\sup_{a>0}\frac{|f(a)|}{a}
	\]
	is finite. Let $\mathcal{Y}$ be the Banach space of continuous functions
	$g:[0,2]\times\mathbf{R}_{\ge0}\to\mathbf{R}$ such that $g(q,0)=0$
	for all $q\in[0,2]$ and the norm
	\begin{equation}
		\|g\|_{\mathcal{Y}}=\sup_{\substack{q\in[0,2]\\
				a>0
			}
		}\e^{-R(\beta)q}\frac{|g(q,a)|}{a}\label{eq:Ynorm}
	\end{equation}
	is finite, where we have defined
	\begin{equation}
		R(\beta)=2\beta^{2}\left(\frac{4\pi/\beta^{2}}{4\pi/\beta^{2}-2}\right)^{3}.\label{eq:Rdef}
	\end{equation}
	Finally, let $\mathcal{Z}\subset\mathcal{Y}$ be the closed subset defined by
	\[
		\mathcal{Z}=\left\{ g\in\mathcal{Y}\ :\ \inf_{a\ge0}g(q,a)\ge0 \text{ and }\Lip g(q,\cdot)\le(4\pi/\beta^{2}-q)^{-1/2}\text{ for all }q\in[0,2]\right\}.
	\]
	Thus, we are done if we can show that the map $\mathcal{Q}$ has a
	unique fixed point in $\mathcal{Z}$, and we will do this by showing
	that $\mathcal{Q}$ maps $\mathcal{Z}$ into itself and moreover is a contraction
	on $\mathcal{Z}$.

	\emph{Step 1: $L^{2}$ bound.} If $g\in\mathcal{Z}$, by the fact that $g(q,0)=0$ we have $g(q,x)\leq \Lip g(q,\cdot)x$ for any $x>0$, so
	\begin{align*}
		\mathbf{E}\Xi_{a,Q}^{g}(q)^{2} & =a^{2}+\int_{0}^{q}\mathbf{E}g(Q-p,\Xi_{a,Q}^{g}(p))^{2}\,\dif p\le a^{2}+\int_{0}^{q}\frac{\mathbf{E}\Xi_{a,Q}^{g}(p)^{2}}{4\pi/\beta^{2}-Q+p}\,\dif p.
	\end{align*}
	By Grönwall's inequality, this means that
	\begin{align}
		\mathbf{E}\Xi_{a,Q}^{g}(q)^{2} & \le a^{2}\exp\left\{ \int_{0}^{q}\frac{1}{4\pi/\beta^{2}-Q+p}\,\dif p\right\} =a^{2}\cdot\frac{4\pi/\beta^{2}-Q+q}{4\pi/\beta^{2}-Q}.\label{eq:L2norm}
	\end{align}

	\emph{Step 2: $\mathcal{Q}$ maps $\mathcal{Z}$ to itself. }Let $g\in\mathcal{Z}$.
	It is clear that %
	$\mathcal{Q}g(q,0)=0$
	for all $q\in[0,2]$. It remains to check that $\mathcal{Q}g$ is continuous and $\Lip(\mathcal{Q}g(q,\cdot))\le(4\pi/\beta^{2}-q)^{-1/2}$
	for all $q\in[0,2]$. For the Lipschitz property, we have
	\begin{align}
		|\mathcal{Q}g(Q,a)-\mathcal{Q}g(Q,b)| & =\frac{1}{2\sqrt{\pi}}\left|(\mathbf{E}\sigma(\Xi_{a,Q}^{g}(Q))^{2})^{1/2}-(\mathbf{E}\sigma(\Xi_{b,Q}^{g}(Q))^{2})^{1/2}\right|\nonumber \\
		                                      & \le\frac{\beta}{2\sqrt{\pi}}\left(\mathbf{E}[\Xi_{a,Q}^{g}(Q)-\Xi_{b,Q}^{g}(Q)]^{2}\right)^{1/2}.\label{eq:difQs}
	\end{align}
	Now we note that, for any $q\le Q$, we have
	\begin{align*}
		\mathbf{E}[\Xi_{a,Q}^{g}(q)-\Xi_{b,Q}^{g}(q)]^{2} & =(a-b)^{2}+\int_{0}^{q}\mathbf{E}[g(Q-p,\Xi_{a,Q}^{g}(p))-g(Q-p,\Xi_{b,Q}^{g}(p))]^{2}\,\dif p            \\
		                                                  & \le(a-b)^{2}+\int_{0}^{q}\Lip(g(Q-p,\cdot))^{2}\mathbf{E}[\Xi_{a,Q}^{g}(p)-\Xi_{b,Q}^{g}(p)]^{2}\,\dif p.
	\end{align*}
	By Grönwall's inequality, this means that
	\[
		\mathbf{E}[\Xi_{a,Q}^{g}(q)-\Xi_{b,Q}^{g}(q)]^{2}\le(a-b)^{2}\exp\left\{ \int_{0}^{q}\Lip(g(Q-p,\cdot))^{2}\,\dif s\right\} .
	\]
	Using this in \eqref{difQs}, we have
	\begin{align*}
		|\mathcal{Q}g(Q,a)-\mathcal{Q}g(Q,b) | & \le\frac{\beta}{2\sqrt{\pi}}|a-b|\exp\left\{ \frac{1}{2}\int_{0}^{Q}\Lip(g(Q-p,\cdot))^{2}\,\dif p\right\} \\
		                                       & =\frac{\beta}{2\sqrt{\pi}}|a-b|\exp\left\{ \frac{1}{2}\int_{0}^{Q}\Lip(g(p,\cdot))^{2}\,\dif p\right\} ,
	\end{align*}
	so
	\[
		\Lip(\mathcal{Q}g(Q,\cdot))\le\frac{\beta}{2\sqrt{\pi}}\exp\left\{ \frac{1}{2}\int_{0}^{Q}\Lip(g(p,\cdot))^{2}\,\dif p\right\} .
	\]
	Therefore, since
	\[
		\Lip(g(p,\cdot))\le(4\pi/\beta^{2}-p)^{-1/2},
	\]
	we also have
	\[
		\Lip(\mathcal{Q}g(Q,\cdot))\le\frac{\beta}{2\sqrt{\pi}}\exp\left\{ \frac{1}{2}\int_{0}^{Q}\frac{1}{4\pi/\beta^{2}-p}\,\dif p\right\} =(4\pi/\beta^{2}-Q)^{-1/2}.
	\]

	Next we show that for each $a>0$, $\mathcal{Q}g(\cdot,a)$ is continuous on $[0,2]$. The argument is rather standard and similar to the above discussion, so we do not provide all details. Taking $0\leq Q_1<Q_2\leq2$,  we have	\[
		|\mathcal{Q}g(Q_1,a)-\mathcal{Q}g(Q_2,a)| \leq  \frac{\beta}{2\sqrt{\pi}}\left(\mathbf{E}|\Xi_{a,Q_1}^{g}(Q_1)-\Xi_{a,Q_2}^{g}(Q_2)|^{2}\right)^{1/2}.
	\]
	For any $q\leq Q_1$, we write the difference as
	\[
		\Xi_{a,Q_1}^{g}(Q_1)-\Xi_{a,Q_2}^{g}(Q_2)=\Xi_{a,Q_1}^{g}(Q_1)-\Xi_{a,Q_2}^{g}(Q_1)+\Xi_{a,Q_2}^{g}(Q_1)-\Xi_{a,Q_2}^{g}(Q_2),
	\]
	and the first term can be estimated as follows: for any $q\leq Q_1$,
	\[
		\Xi_{a,Q_1}^{g}(q)-\Xi_{a,Q_2}^{g}(q)=\int_0^q g(Q_1-s,\Xi_{a,Q_1}^g(s))\,\dif B_s-\int_0^qg(Q_2-s,\Xi_{a,Q_2}^g(s))\,\dif B_s,
	\]
	which yields
	\[
		\begin{aligned}
			\mathbf{E}|\Xi_{a,Q_1}^{g}(q)-\Xi_{a,Q_2}^{g}(q)|^2 &\leq  2\int_0^{q} \mathbf{E}|g(Q_1-s,\Xi_{a,Q_1}^g(s))-g(Q_2-s,\Xi_{a,Q_1}^g(s))|^2 \,\dif s         \\
			                                                         & \quad+2\int_0^q \mathbf{E}|g(Q_2-s,\Xi_{a,Q_1}^g(s))-g(Q_2-s,\Xi_{a,Q_2}^g(s))|^2\,\dif s\\&=:I_1+I_2.
		\end{aligned}
	\]
	The term $I_2$ can be bounded from above by
	\[
		2\int_0^q\Lip(g(Q_2-s,\cdot)^2\mathbf{E}|\Xi_{a,Q_1}^{g}(s)-\Xi_{a,Q_2}^{g}(s)|^2\,\dif s.
	\]
	For $I_1$,  the integrand
	\[
		\mathbf{E}[|g(Q_1-s,\Xi_{a,Q_1}^g(s))-g(Q_2-s,\Xi_{a,Q_1}^g(s))|^2]
	\] is bounded, and converges to zero as $Q_2\to Q_1$ for each $s$, by the dominated convergence theorem, \eqref{L2norm} and the fact that $g$ is continuous in the first variable and $g(q,x)\leq Cx$ for all $x\geq0,q\in[0,2]$. Therefore, invoking Grönwall's inequality again, we obtain
	\[
		\mathbf{E}|\Xi_{a,Q_1}^{g}(Q_1)-\Xi_{a,Q_2}^{g}(Q_1)|^2 \to 0, \quad\quad \text{ as } Q_2\to Q_1.
	\]
	A simpler argument shows that
	\[
		\mathbf{E}|\Xi_{a,Q_2}^{g}(Q_2)-\Xi_{a,Q_2}^{g}(Q_1)|^2 \to 0, \quad\quad \text{ as } Q_2\to Q_1.
	\]
	Therefore,   $\mathcal{Q}g(\cdot,a)$ is continuous, so $\mathcal{Q}$ maps $\mathcal{Z}$ to itself.

	\emph{Step 3: contraction.} Let $g_{1},g_{2}\in\mathcal{Z}$. Then
	we have
	\[
		\Xi_{a,Q}^{g_{1}}(q)-\Xi_{a,Q}^{g_{2}}(q)=\int_{0}^{q}[g_{1}(Q-p,\Xi_{a,Q}^{g_{1}}(p))-g_{2}(Q-p,\Xi_{a,Q}^{g_{2}}(p))]\,\dif B(p),
	\]
	so
	\begin{align*}
		&\mathbf{E}  [\Xi_{a,Q}^{g_{1}}-\Xi_{a,Q}^{g_{2}}](q)^{2}=\int_{0}^{q}\mathbf{E}[g_{1}(Q-p,\Xi_{a,Q}^{g_{1}}(p))-g_{2}(Q-p,\Xi_{a,Q}^{g_{2}}(p))]^{2}\,\dif p                                                                              \\
		           & \ \le2\int_{0}^{q}\left(\|(g_{1}-g_{2})(Q-p,\cdot)\|_{{\mathcal{X}}}^{2}\mathbf{E}\Xi_{a,Q}^{g_{1}}(p)^{2}+\Lip(g_{2}(Q-p,\cdot))^{2}\mathbf{E}[\Xi_{a,Q}^{g_{1}}-\Xi_{a,Q}^{g_{2}}](p)^{2}\right)\,\dif p                      \\
		           & \ \le2\int_{0}^{q}\left(\|(g_{1}-g_{2})(Q-p,\cdot)\|_{{\mathcal{X}}}^{2}a^{2}\cdot\frac{4\pi/\beta^{2}-Q+p}{4\pi/\beta^{2}-Q}+\frac{\mathbf{E}[\Xi_{a,Q}^{g_{1}}-\Xi_{a,Q}^{g_{2}}](p)^{2}}{4\pi/\beta^{2}-Q+p}\right)\,\dif p,
	\end{align*}
	with the last inequality by \eqref{L2norm}. By Grönwall's inequality,
	this means that
	\begin{align*}
		\mathbf{E}&[\Xi_{a,Q}^{g_{1}}-\Xi_{a,Q}^{g_{2}}](q)^{2}\\ & \le2a^{2}\left(\int_{0}^{q}\|(g_{1}-g_{2})(Q-p,\cdot)\|_{{\mathcal{X}}}^{2}\frac{4\pi/\beta^{2}-Q+p}{4\pi/\beta^{2}-Q}\,\dif p\right)\exp\left\{ \int_{0}^{q}\frac{{2}}{4\pi/\beta^{2}-Q+p}\,\dif p\right\} \\
		                                                       & =2a^{2}\cdot\left(\frac{4\pi/\beta^{2}-Q+q}{4\pi/\beta^{2}-Q}\right)^{{2}}\int_{0}^{q}\|(g_{1}-g_{2})(Q-p,\cdot)\|_{{\mathcal{X}}}^{2}\frac{4\pi/\beta^{2}-Q+p}{4\pi/\beta^{2}-Q}\,\dif p.
	\end{align*}
	In particular, we have
	\[
		\mathbf{E}[\Xi_{a,Q}^{g_{1}}-\Xi_{a,Q}^{g_{2}}](Q)^{2}\le2a^{2}\cdot\left(\frac{4\pi/\beta^{2}}{4\pi/\beta^{2}-Q}\right)^{{3}}\int_{0}^{Q}\|(g_{1}-g_{2})(p,\cdot)\|_{{\mathcal{X}}}^{2}\,\dif p.
	\]
	Then we have
	\begin{align*}
		(\mathcal{Q}g_{1}-\mathcal{Q}g_{2})(q,a)^{2} & =\left|(\mathbf{E}\sigma(\Xi_{a,q}^{g_{1}}(q))^{2})^{1/2}-(\mathbf{E}\sigma(\Xi_{a,q}^{g_{2}}(q))^{2})^{1/2}\right|^{2}\\&\le\beta^{2}\mathbf{E}[\Xi_{a,q}^{g_{1}}(q)-\Xi_{a,q}^{g_{2}}(q)]^{2} \\
		                                             & \le2a^{2}\beta^{2}\left(\frac{4\pi/\beta^{2}}{4\pi/\beta^{2}-q}\right)^{{3}}\int_{0}^{q}\|(g_{1}-g_{2})(p,\cdot)\|_{\mathcal{X}}^{2}\,\dif p.
	\end{align*}
	This implies that, as long as $\beta<\sqrt{2\pi}$, for all $q\in[0,2]$
	we have
	\[
		\|(\mathcal{Q}g_{1}-\mathcal{Q}g_{2})(q,\cdot)\|_{\mathcal{X}}^{2}\le2\beta^{2}\left(\frac{4\pi/\beta^{2}}{4\pi/\beta^{2}-q}\right)^{{3}}\int_{0}^{q}\|(g_{1}-g_{2})(p,\cdot)\|_{\mathcal{X}}^{2}\,\dif p.
	\]
	Therefore,
	\begin{align*}
		\|\mathcal{Q}g_{1}-\mathcal{Q}g_{2}\|_{\mathcal{Y}}^{2} & =\sup_{q\in[0,2]}\e^{-2R(\beta)q}\|(\mathcal{Q}g_{1}-\mathcal{Q}g_{2})(q,\cdot)\|_{\mathcal{X}}^{2}                                                                        \\
		                                                        & \le\sup_{q\in [0,2]}2\beta^{2}\left(\frac{4\pi/\beta^{2}}{4\pi/\beta^{2}-2}\right)^{{3}}e^{-2R(\beta)q}\int_{0}^{q}\|(g_{1}-g_{2})(p,\cdot)\|_{\mathcal{X}}^{2}\,\dif p    \\
		                                                        & \leq\beta^{2}\left(\frac{4\pi/\beta^{2}}{4\pi/\beta^{2}-2}\right)^{{3}}\frac{1}{R(\beta)}\|g_{1}-g_{2}\|_{\mathcal{Y}}^{2}=\frac{1}{2}\|g_{1}-g_{2}\|_{\mathcal{Y}}^{2}.
	\end{align*}
	Recall that $R(\beta)$ was defined in \eqref{Rdef}. Therefore, $\mathcal{Q}$ is a contraction
	on $\mathcal{Z}$ (equipped with the norm inherited from $\mathcal{Y}$)
	and so $\mathcal{Q}$ admits a unique fixed point in $\mathcal{Z}$,
	which is what we needed to show.
\end{proof}

\begin{rem}\label{rem:SDEmoments}
	By the stochastic comparison principle for SDEs \cite{CFG96} and the fact that the geometric Brownian motion (i.e.\ a log-normal random variable) has finite positive moments of all orders, we see that $\mathbf{E} \Xi_{a,Q}(q)^k<\infty$ for all $k\in[0,\infty)$ as well.
\end{rem}

\section{Moment bounds}

The next several sections will work towards a proof of \thmref{convergence}.
In order to carry out our analysis, we will need some bounds on the
moments of the solutions to \eqref{duxi}--\eqref{uxiic}. We establish
these in this section. Moment bounds depend crucially on the \emph{subcriticality}
of the problem, which for us means $\beta<\sqrt{2\pi}$. We will assume
throughout the paper that this is true without further comment. We
also now fix a time horizon $T_{0}\in[1,\infty)$ which will also
			remain fixed throughout the paper. Furthermore, fix $\eps_{0}\in(0,1]$
so that
\begin{equation}
	\frac{\beta^{2}}{4\pi}\cdot\frac{\log(1+2\eps^{-2}T_{0})}{\log\eps^{-1}}<1\label{eq:eps0cond}
\end{equation}
for all $\eps\in(0,\eps_{0}]$. The condition that $\beta<\sqrt{2\pi}$
means that such an $\eps_{0}$ exists. As we are ultimately interested
in the limit $\eps\downarrow0$, the condition \eqref{eps0cond} is
simply a convenience so that various quantities are finite. In \defref{K0def}
below, we fix a constant $K_{0}<\infty$, which depends on $\beta$,
$\eps_{0}$, and $T_{0}$, and will appear in upper bounds throughout
the paper.
\begin{prop}
	\label{prop:momentbound}There exist constants $p>2$ and  $K<\infty$
	(depending on $T_{0}$ and $\beta$) so that, for all $\eps\in(0,\eps_{0}]$,
	all $a\ge0$, and all $t\in[0,T_{0}],x\in\mathbf{R}^2$, we have
	\begin{equation}
		\mathbf{E}u_{\eps,a}(t,x)^{p}\le K^{p}a^{p}.\label{eq:momentbound}
	\end{equation}
\end{prop}

\begin{proof}
	Let $v_{\eps,a}$ solve the linear problem given by \eqref{duxi}--\eqref{uxiic}
	with $\sigma(u)=\beta u$. By \cite[(5.11)]{CSZ20}, for any $p\in[1,2\pi/\beta^{2}+1)$
	we have a constant $K$ so that $\mathbf{E}v_{\eps,a}(t,x)^{p}\le K^{p}a^{p}$.
	Using the stochastic comparison principle proved in \cite[(E-4)]{CK20},
	since $\sigma(u)\le\beta u$ for all $u\in[0,\infty)$ we have $\mathbf{E}u_{\eps,a}(t,x)^{p}\le\mathbf{E}v_{\eps,a}(t,x)^{p}\le K^{p}a^{p}$.
	By the assumption that $\beta<\sqrt{2\pi}$, we have $2\pi/\beta^{2}+1>2$,
	so we can choose $p>2$ as required.
\end{proof}
\begin{rem}
	The case $p=2$ in \eqref{momentbound} is much simpler than the case
	$p>2$. Indeed, the $p=2$ case is a special case of \propref{L2bound}
	below. On the other hand, the proof of the moment bound for $p>2$
	in \cite{CSZ20} for the linear case uses hypercontractivity, and
	the stochastic comparison principle \cite{CK20} takes a substantial
	amount of analysis to prove. Most of the analysis in this paper will
	be in the $L^{2}$ setting, so we will mostly use the $p=2$ case.
	However, we will rely on some tightness statements that require a
	higher moment bound.
\end{rem}

The following proposition gives an $L^{2}$ bound on the difference
of two solutions started at different initial conditions. Recall that
$u_{\eps,a}^{A}$ solves the problem \eqref{uA}--\eqref{uAic},
with the noise turned off on the set of times $A$. The problem \eqref{uA}--\eqref{uAic}
has the mild formulation
\begin{equation}
	u_{\eps,a}^{A}(t,x)=a+\frac{1}{\sqrt{\log\eps^{-1}}}\int_{[0,t]\setminus A}\int G_{t-s}(x-y)\sigma(u_{\eps,a}^{A}(s,y))\,\dif W^{\eps}(s,y).\label{eq:uAmild}
\end{equation}
Here and henceforth, when we do not specify the domain of integration for an integral we mean that the integral is taken over all of $\R^2$.
\begin{prop}
	\label{prop:L2bound}There exists a constant $K<\infty$ (depending on
	$T_{0}$ and $\beta$) so that, for all $\eps\in(0,\eps_{0}]$,
		$a_{1},a_{2}\ge0$,  $T\in[0,T_{0}]$, $x\in\mathbf{R}^2$, and measurable $A\subset[0,\infty)$,
	we have
	\begin{equation}
		\left(\mathbf{E}[u_{\eps,a_{2}}^{A}(t,x)-u_{\eps,a_{1}}^{A}(t,x)]^{2}\right)^{1/2}\le K|a_{2}-a_{1}|.\label{eq:diffudifferentxis}
	\end{equation}
	In particular, for any $a>0$,
	\begin{equation}
		\left(\mathbf{E}u_{\eps,a}^{A}(t,x)^{2}\right)^{1/2}\le Ka.\label{eq:L2bound}
	\end{equation}
	In fact, \eqref{diffudifferentxis} and \eqref{L2bound} hold with
	\begin{equation}
		K=\left(1-\frac{\beta^{2}}{4\pi}\cdot\frac{\log(1+2\eps^{-2}t)}{\log\eps^{-1}}\right)^{-1/2}.\label{eq:explicitK}
	\end{equation}
\end{prop}

Of course, a very important special case is when $A=\emptyset$. Then
the bounds \eqref{diffudifferentxis} and \eqref{L2bound} just involve
$u_{\eps,a}$. (In the latter case this of course is a special case
of \propref{momentbound}.)
\begin{proof}
	Since \eqref{L2bound} is just \eqref{diffudifferentxis} with $a_{2}=a$
	and $a_{1}=0$, it suffices to prove \eqref{diffudifferentxis}. Subtracting
	two copies of \eqref{uAmild} (with $a=a_{1}$ and $a=a_{2}$) and
	taking second moments, we obtain
	\begin{align*}
		\mathbf{E} & (u_{\eps,a_{2}}^{A}(t,x)-u_{\eps,a_{1}}^{A}(t,x))^{2}                                                                                                                                                                                             \\
		           & =(a_{2}-a_{1})^{2}                                                                                                                                                                                                                                \\
		           & \quad+\frac{1}{\log\eps^{-1}}\int_{[0,t]\setminus A}\iint\mathbf{E}\prod_{i=1}^{2}\left([\sigma(u_{\eps,a_{2}}^{A}(s,y_{i}))-\sigma(u_{\eps,a_{1}}^{A}(s,y_{i}))]G_{t-s}(x-y_{i})\right)\\&\qquad \qquad \qquad\qquad\qquad\qquad\cdot G_{\eps^{2}}(y_{1}-y_{2})\,\dif y_{1}\,\dif y_{2}\,\dif s \\
		           & \le(a_{2}-a_{1})^{2}+\frac{\beta^{2}}{{2\pi}\log\eps^{-1}}\int_{0}^{t}\frac{\mathbf{E}|u_{\eps,a_{2}}^{A}(s,x)-u_{\eps,a_{1}}^{A}(s,x)|^{2}}{2(t-s)+\eps^{2}}\,\dif s.
	\end{align*}
	Then \eqref{diffudifferentxis} follows from \lemref{secondmomentbound-general}
	below.
\end{proof}
It remains to prove the lemma used above, which will also be useful
in the future.
\begin{lem}
	\label{lem:secondmomentbound-general}For all $\eps\in(0,\eps_{0}]$,
		all $a\ge0$, and all $T\in[0,T_{0}]$, the following holds. Let $f:[0,T]\to[0,\infty)$
	be such that
	\[
		f(t)\le a^{2}+\frac{\beta^{2}}{2\pi\log\eps^{-1}}\int_{0}^{t}\frac{f(s)}{2(t-s)+\eps^{2}}\,\dif s
	\]
	for all $t\in[0,T]$. Then, for all $t\in[0,T]$, we have
	\[
		f(t)\le\frac{a^{2}}{1-\frac{\beta^{2}}{4\pi}\cdot\frac{\log(1+2\eps^{-2}t)}{\log\eps^{-1}}}.
	\]
\end{lem}

\begin{proof}
	Define $[0,t]_<^j=\{(s_{1},\ldots,s_{j})\in[0,t]^{j}\mid s_{1}\le\cdots\le s_{j}\}$.
	Then we have
	\begin{align}
		f(t) & \le a^{2}\sum_{j=0}^{\infty}\frac{\beta^{2j}}{(4\pi\log\eps^{-1})^{j}}\int_{[0,t]_<^j}\prod_{k=1}^{j}\frac{1}{s_{k+1}-s_{k}+\eps^{2}/2}\,\dif s_{1}\cdots\dif s_{j}\nonumber                                                                                                                       \\
		     & \le a^{2}\sum_{j=0}^{\infty}\frac{\beta^{2j}}{(4\pi\log\eps^{-1})^{j}}\int_{[0,t]^{j}}\prod_{k=1}^{j}\frac{1}{r_{j}+\eps^{2}/2}\,\dif r_{1}\cdots\dif r_{j}\notag\\&=a^{2}\sum_{j=0}^{\infty}\frac{\beta^{2j}}{(4\pi\log\eps^{-1})^{j}}\left(\int_{0}^{t}\frac{1}{r+\eps^{2}/2}\,\dif r\right)^{j}\nonumber \\
		     & =a^{2}\sum_{j=0}^{\infty}\left(\frac{\beta^{2}}{4\pi\log\eps^{-1}}\log(1+2\eps^{-2}t)\right)^{j}=\frac{a^{2}}{1-\frac{\beta^{2}}{4\pi}\cdot\frac{\log(1+2\eps^{-2}t)}{\log\eps^{-1}}},\label{eq:fbd-final}
	\end{align}
	where we used \eqref{eps0cond} for the last identity.
\end{proof}
To avoid having to constantly quantify constants, we now fix our essential
constant once and for all.

\begin{defn}
	\label{def:K0def}Fix
	\begin{equation}
		K_{0}\ge\sup_{\eps\in(0,\eps_{0}]}\left(1-\frac{\beta^{2}}{4\pi}\cdot\frac{2+\log(1+2\eps^{-2}t)}{\log\eps^{-1}}\right)^{-1/2}\label{eq:K0lb}
	\end{equation}
	large enough so that \propref[s]{momentbound} and~\ref{prop:L2bound}
	hold with $K=K_{0}$.
\end{defn}

By \eqref{explicitK} and the proof of \cite[(5.11)]{CSZ20}, we see
that we could take
\[
	K_{0}=\sup\limits _{\eps\in(0,\eps_{0}]}\left(1-\frac{\tilde{\beta}^{2}}{4\pi}\cdot\frac{2+\log(1+2\eps^{-2}t)}{\log\eps^{-1}}\right)^{-1/2}
\]
for some $\tilde{\beta}\in(\beta,\sqrt{2\pi})$. The precise form of
$K_{0}$ will not be important for us (although at one point we will
directly use the explicit expression \eqref{explicitK}). The extra
summand of $2$ in the lower limit condition \eqref{K0lb} for $K_{0}$ (compared to \eqref{explicitK}) is to allow
$K_{0}$ to also suffice for bounds in later sections. (See the proofs
of \lemref[s]{killthenoise-integralineq} and~\ref{lem:spatialrecursive}
below.)

Now we can bootstrap \propref{L2bound} to obtain a stronger bound
on the variance of the solution.
\begin{prop}
	\label{prop:varbound}If $a>0$, $\eps\in[0,\eps_{0})$, and $A\subset[0,\infty)$
	is measurable, then
	\begin{equation}
		\left(\mathbf{E}[u_{\eps,a}^{A}(t,x)-a]^{2}\right)^{1/2}\le\frac{\beta aK_{0}}{2\sqrt{\pi}}\sqrt{\frac{\log(1+2\eps^{-2}t)}{\log \eps^{-1}}}.\label{eq:varbound}
	\end{equation}
\end{prop}

Of course, for $t$ of order $1$, the bound \eqref{varbound} is
redundant to \eqref{L2bound}. It will be used
when $t$ is chosen small so that $\log(1+\eps^{-2}t)\ll\log \eps^{-1}$.
\begin{proof}
	Similar to   the computation in \propref{L2bound}, we have
	\begin{align*}
		\mathbf{E}[u_{\eps,a}^{A}(t,x)-a]^{2} & \le\frac{\beta^{2}}{2\pi\log\eps^{-1}}\int_{0}^{t}\frac{\mathbf{E}u_{\eps,a}^{A}(s,x)^{2}}{2(t-s)+\eps^{2}}\,\dif s,
	\end{align*}
	and then \eqref{varbound} follows from \eqref{L2bound}.
\end{proof}

\section{Turning off the noise on an interval\label{sec:shutoffnoise-oneinterval}}

As discussed in the introduction, an important part of our argument
will be turning off the noise in the equation \eqref{duxi}--\eqref{uxiic}
for a certain set of times, and comparing the resulting solution to
the original solution. In this section we bound the error incurred
by this noise shutoff procedure when the noise is shut off on a single
interval. In \secref{timediscretization}, we will iterate this procedure
to turn off the noise on multiple intervals. For now our goal
is to prove the following proposition.
\begin{prop}
	\label{prop:exciseoneinterval}Let $A\subset[0,\infty)$
	and suppose that $\sup A\le\tau_{1}\le\tau_{2}\le T_{0}$. Then for
	any $t\in[\tau_{2},T_{0}]$ and any $x\in\mathbf{R}^{2}$ we have
	\[
		\mathbf{E}\left(u_{\eps,a}^{A}(t,x)-u_{\eps,a}^{A\cup[\tau_{1},\tau_{2}]}(t,x)\right)^{2}\le\frac{K_{0}^{{4}}\beta^{2}a^{2}}{4\pi\log\eps^{-1}}\left(\log\frac{t-\tau_{1}+\eps^{2}/2}{t-\tau_{2}+\eps^{2}/2}+K_{0}^{2}\right).
	\]
\end{prop}

\begin{proof}
	Subtracting two copies of the mild formulation \eqref{uAmild} (with
	the sets $A$ and $A\cup[\tau_{1},\tau_{2}]$ respectively), we have
	\begin{align*}
		u_{\eps,a}^{A} & (t,x)-u_{\eps,a}^{A\cup[\tau_{1},\tau_{2}]}(t,x)                                                                                                                               \\
		               & =\frac{1}{\sqrt{\log\eps^{-1}}}\int_{[0,t]\setminus A}\int G_{t-s}(x-y)\sigma(u_{\eps,a}^{A}(s,y))\,\dif W^{\eps}(s,y)                                                         \\
		               & \qquad-\frac{1}{\sqrt{\log\eps^{-1}}}\int_{[0,t]\setminus(A\cup[\tau_{1},\tau_{2}])}\int G_{t-s}(x-y)\sigma(u_{\eps,a}^{A\cup[\tau_{1},\tau_{2}]}(s,y))\,\dif W^{\eps}(s,y)    \\
		               & =\frac{1}{\sqrt{\log\eps^{-1}}}\int_{\tau_{1}}^{\tau_{2}}\int G_{t-s}(x-y)\sigma(u_{\eps,a}^{A}(s,y))\,\dif W^{\eps}(s,y)                                                      \\
		               & \qquad+\frac{1}{\sqrt{\log\eps^{-1}}}\int_{\tau_{2}}^{t}\int G_{t-s}(x-y)[\sigma(u_{\eps,a}^{A}(s,y))-\sigma(u_{\eps,a}^{A\cup[\tau_{1},\tau_{2}]}(s,y))]\,\dif W^{\eps}(s,y).
	\end{align*}
	In the second ``='' we used that $u_{\eps,a}^{A}(t,x)=u_{\eps,a}^{A\cup[\tau_{1},\tau_{2}]}(t,x)$
	whenever $t\le\tau_{1}$. Taking the second moment, we have for all
	$t\ge\tau_{2}$ that
	\begin{align}
		\mathbf{E} & \left(u_{\eps,a}^{A}(t,x)-u_{\eps,a}^{A\cup[\tau_{1},\tau_{2}]}(t,x)\right)^{2}\nonumber                                                                                                                                                                                                                        \\
		           & =\frac{1}{\log\eps^{-1}}\int_{\tau_{1}}^{\tau_{2}}\iint G_{\eps^{2}}(y_{1}-y_{2})\mathbf{E}\prod_{i=1}^{2}\left(G_{t-s}(x-y_{i})\sigma(u_{\eps,a}^{A}(s,y_i))\right)\,\dif y_{1}\,\dif y_{2}\,\dif s\nonumber                                                                                                   \\
		           & \quad+\frac{1}{\log\eps^{-1}}\int_{\tau_{2}}^{t}\iint \mathbf{E}\prod_{i=1}^{2}\left(G_{t-s}(x-y_{i})[\sigma(u_{\eps,a}^{A}(s,y_{i}))-\sigma(u_{\eps,a}^{A\cup[\tau_{1},\tau_{2}]}(s,y_{i}))]\right)\notag\\&\qquad\qquad\qquad\qquad\cdot G_{\eps^{2}}(y_{1}-y_{2})\,\dif y_{1}\,\dif y_{2}\,\dif s\nonumber                                          \\
		           & \le\frac{\beta^{2}}{2\pi\log\eps^{-1}}\int_{\tau_{1}}^{\tau_{2}}\frac{\mathbf{E}u_{\eps,a}^{A}(s,y)^{2}}{2(t-s)+\eps^{2}}\,\dif s+\frac{\beta^{2}}{2\pi\log\eps^{-1}}\int_{\tau_{2}}^{t}\frac{\mathbf{E}[u_{\eps,a}^{A}(s,y)-u_{\eps,a}^{A\cup[\tau_{1},\tau_{2}]}(s,y)]^{2}}{2(t-s)+\eps^{2}}\,\dif s\nonumber \\
		           & \le\frac{\beta^{2}a^{2}K_{0}^{2}}{{4\pi}\log\eps^{-1}}\log\frac{t-\tau_{1}+\eps^{2}/2}{t-\tau_{2}+\eps^{2}/2}+\frac{\beta^{2}}{4\pi\log\eps^{-1}}\int_{\tau_{2}}^{t}\frac{\mathbf{E}[u_{\eps,a}^{A}(s,y)-u_{\eps,a}^{A\cup[\tau_{1},\tau_{2}]}(s,y)]^{2}}{t-s+\eps^{2}/2}\,\dif s.\label{eq:Adiffbd}
	\end{align}
	In the last inequality we used \eqref{L2bound}. Now if we put
	\begin{equation}
		f(t)=\mathbf{E}\left(u_{\eps,a}^{A}(\tau_{2}+t,x)-u_{\eps,a}^{A\cup[\tau_{1},\tau_{2}]}(\tau_{2}+t,x)\right)^{2}, \quad\quad t\geq0,\label{eq:fdef}
	\end{equation}
	then \eqref{Adiffbd} can be rewritten as
	\[
		f(t)\le\frac{\beta^{2}a^{2}K_{0}^{2}}{4\pi\log\eps^{-1}}\log\frac{t+\tau_{2}-\tau_{1}+\eps^{2}/2}{t+\eps^{2}/2}+\frac{\beta^{2}}{4\pi\log\eps^{-1}}\int_{0}^{t}\frac{f(s)}{t-s+\eps^{2}/2}\,\dif s.
	\]
	Now we apply \lemref{killthenoise-integralineq} below with $M=(4\pi)^{-1}\beta^{2}a^{2}K_{0}^{2}$
	and $r=\tau_{2}-\tau_{1}$. (The requirement that $f$ has a bounded supremum on compact intervals is satisfied by applying \propref{momentbound}.) This gives us
	\begin{align*}
		f(t) & \le\frac{K_{0}^{{4}}\beta^{2}a^{2}}{4\pi\log\eps^{-1}}\left(\log\frac{t+\tau_{2}-\tau_{1}+\eps^{2}/2}{t+\eps^{2}/2}+K_{0}^{2}\right).
	\end{align*}
	Recalling the definition \eqref{fdef} completes the proof.
\end{proof}
We will prove \lemref{killthenoise-integralineq}, which we used in
the above proof, shortly. First we need a preliminary lemma.
\begin{lem}
	\label{lem:logovertminuss}For any $t,r,\eps>0$ we have
	\[
		\int_{0}^{t}\frac{\log\frac{t+r+\eps^{2}/2}{s+\eps^{2}/2}}{t-s+\eps^{2}/2}\,\dif s\le\left(2+\log(1+2\eps^{-2}t)\right)\left(1+\log\frac{t+r+\eps^{2}/2}{t+\eps^{2}/2}\right).
	\]
\end{lem}

\begin{proof}
	We write
	\begin{align}
		\int_{0}^{t}\frac{\log\frac{t+r+\eps^{2}/2}{s+\eps^{2}/2}}{t-s+\eps^{2}/2}\,\dif s & =\left(\int_{0}^{t/2}+\int_{t/2}^t\right)\frac{\log\frac{t+r+\eps^{2}/2}{s+\eps^{2}/2}}{t-s+\eps^{2}/2}\,\dif s\nonumber\\
& \le\frac{2}{t}\int_{0}^{t}\log\frac{t+r+\eps^{2}/2}{s+\eps^{2}/2}\,\dif s+\left(\log\frac{t+r+\eps^{2}/2}{t/2+\eps^{2}/2}\right)\int_{0}^{t}\frac{1}{t-s+\eps^{2}/2}\,\dif s.\label{eq:splittheintegral-1}
	\end{align}
	Now we have
	\begin{align}
		\int_{0}^{t}\log\frac{t+r+\eps^{2}/2}{s+\eps^{2}/2}\,\dif s & =t-{\frac{\eps^{2}}{2}}\log\frac{t+\eps^{2}/2}{\eps^{2}/2}+t\log\frac{t+r+\eps^{2}/2}{t+\eps^{2}/2}\notag\\&\le t\left(1+\log\frac{t+r+\eps^{2}/2}{t+\eps^{2}/2}\right).\label{eq:boundintegral-1}
	\end{align}
	Also, we have
	\begin{equation}
		\log\frac{t+r+\eps^{2}/2}{t/2+\eps^{2}/2}=\log\frac{2t+2r+\eps^{2}}{t+\eps^{2}}\le\log2+\log\frac{t+r+\eps^{2}/2}{t+\eps^{2}/2}.\label{eq:boundlog-1}
	\end{equation}
	Using \eqref{boundintegral-1} and \eqref{boundlog-1} in \eqref{splittheintegral-1},
	we have
	\begin{align}
		\int_{0}^{t}\frac{\log\frac{t+r+\eps^{2}/2}{s+\eps^{2}/2}}{t-s+\eps^{2}/2}\,\dif s & \le2+2\log\frac{t+r+\eps^{2}/2}{t+\eps^{2}/2}+\left(\log2+\log\frac{t+r+\eps^{2}/2}{t+\eps^{2}/2}\right)\log(1+2\eps^{-2}t)\nonumber \\
		                                                                                   & \le\left(2+\log(1+2\eps^{-2}t)\right)\left(1+\log\frac{t+r+\eps^{2}/2}{t+\eps^{2}/2}\right),\label{eq:logovert-1}
	\end{align}
	which was the claim.
\end{proof}
\begin{lem}
	\label{lem:killthenoise-integralineq}Let $\eps\in(0,\eps_{0}]$ and $M,r>0$,
	suppose that $f$ satisfies	the bound
	\begin{equation}
		f(t)\le\frac{M}{\log\eps^{-1}}\log\frac{t+r+\eps^{2}/2}{t+\eps^{2}/2}+\frac{\beta^{2}}{4\pi\log\eps^{-1}}\int_{0}^{t}\frac{f(s)}{t-s+\eps^{2}/2}\,\dif s\label{eq:ubound}
	\end{equation}
	for all $t\in[0,T_{0}]$, and $\sup_{t\in[0,T_0]}|f(t)|<\infty$. Then we have
	\begin{equation}
		f(t)\le\frac{K_{0}^{2}M}{\log\eps^{-1}}\left(\log\frac{t+r+\eps^{2}/2}{t+\eps^{2}/2}+K_{0}^{2}\right)\label{eq:boundconcl}
	\end{equation}
	for all $t\in[0,T_{0}]$.
\end{lem}

\begin{proof}
	Suppose that
	\begin{equation}
		f(t)\le B_{1}\log\frac{t+r+\eps^{2}/2}{t+\eps^{2}/2}+B_{2}.\label{eq:uB1B2ineq}
	\end{equation}
	By assumption, this inequality holds with $B_{1}=0$ and $B_{2}=\sup_{t\in[0,T_0]}|f(t)|$.
	Substituting \eqref{uB1B2ineq} into the r.h.s. of \eqref{ubound}, we have
	\begin{align}
		f(t) & \le\frac{M}{\log\eps^{-1}}\log\frac{t+r+\eps^{2}/2}{t+\eps^{2}/2}+\frac{\beta^{2}}{4\pi\log\eps^{-1}}\int_{0}^{t}\frac{B_{1}\log\frac{s+r+\eps^{2}/2}{s+\eps^{2}/2}+B_{2}}{t-s+\eps^{2}/2}\,\dif s\nonumber                                                                \\
		     & \le\frac{M}{\log\eps^{-1}}\log\frac{t+r+\eps^{2}/2}{t+\eps^{2}/2}+\frac{\beta^{2}B_{1}}{4\pi\log\eps^{-1}}\int_{0}^{t}\frac{\log\frac{s+r+\eps^{2}/2}{s+\eps^{2}/2}}{t-s+\eps^{2}/2}\,\dif s+\frac{\beta^2B_2\log(1+2\eps^{-2}T_0)}{4\pi\log\eps^{-1}}.\label{eq:u2expand}
	\end{align}
	For the middle term of the above inequality, we have
	\begin{align}
		\int_{0}^{t}\frac{\log\frac{s+r+\eps^{2}/2}{s+\eps^{2}/2}}{t-s+\eps^{2}/2}\,\dif s\le & \int_{0}^{t}\frac{\log\frac{{t}+r+\eps^{2}/2}{s+\eps^{2}/2}}{t-s+\eps^{2}/2}\,\dif s\le
		\left(2+\log(1+2\eps^{-2}t)\right)\left(1+\log\frac{t+r+\eps^{2}/2}{t+\eps^{2}/2}\right)\label{eq:applylogovertminus}
	\end{align}
	by \lemref{logovertminuss}. Substituting \eqref{applylogovertminus}
	into \eqref{u2expand}, we have
	\begin{align*}
		f(t) & \le\frac{M}{\log\eps^{-1}}\log\frac{t+r+\eps^{2}/2}{t+\eps^{2}/2}+\frac{\beta^{2}B_{1}}{4\pi\log\eps^{-1}}\left(2+\log(1+2\eps^{-2}t)\right)\left(1+\log\frac{t+r+\eps^{2}/2}{t+\eps^{2}/2}\right)\\&\qquad+\frac{\beta^2B_2\log(1+2\eps^{-2}T_0)}{4\pi\log\eps^{-1}} \\
		     & =\frac{1}{\log\eps^{-1}}\left(\frac{\beta^{2}B_1}{4\pi}\left(2+\log(1+2\eps^{-2}t)\right)+M\right)\log\frac{t+r+\eps^{2}/2}{t+\eps^{2}/2}                                        \\
		     & \qquad+\frac{\beta^{2}B_{1}}{4\pi\log\eps^{-1}}\left(2+\log(1+2\eps^{-2}t)\right) +\frac{\beta^2B_2\log(1+2\eps^{-2}T_0)}{4\pi \log\eps^{-1}}                                                                                                                                                                                            \\
		     & \le\left((1-K_{0}^{-2})B_{1}+\frac{M}{\log\eps^{-1}}\right)\log\frac{t+r+\eps^{2}/2}{t+\eps^{2}/2}+                                                                                                                                                          
		B_1+(1-K_{0}^{-2})B_{2},
	\end{align*}
	where in the last inequality we used \eqref{K0lb}.
	Define $B_{1}^{(0)}=0$
	and $B_{2}^{(0)}=\sup_{t\in [0,T_0]}|f(t)|$, so for each $n\ge0$,
	\eqref{uB1B2ineq} holds with
	\begin{align}
		B_{1}=B_{1}^{(n)} & =(1-K_{0}^{-2})B_{1}^{(n-1)}+\frac{M}{\log\eps^{-1}},\label{eq:B1n} \\
		B_{2}=B_{2}^{(n)} & =B_{1}^{(n-1)}+(1-K_{0}^{-2})B_{2}^{(n-1)}.\label{eq:B2n}
	\end{align}
	From \eqref{B1n} we conclude that
	\begin{equation}
		B_{1}^{(n)}\le\frac{K_{0}^{2}M}{\log\eps^{-1}}\label{eq:B1nbounded}
	\end{equation}
	for all $n$. Then we have from \eqref{B2n} that
	\[
		B_{2}^{(n)}\le\frac{K_{0}^{2}M}{\log\eps^{-1}}+(1-K_{0}^{-2})B_{2}^{(n-1)},
	\]
	so
	\begin{equation}
		\limsup_{n\to\infty}B_{2}^{(n)}\le\frac{K_{0}^{4}M}{\log\eps^{-1}}.\label{eq:B2nbounded}
	\end{equation}
	Using \eqref{B1nbounded} and \eqref{B2nbounded} in
	\eqref{uB1B2ineq}, we obtain \eqref{boundconcl}.
\end{proof}

\section{Replacing a smoothed field with a constant\label{sec:smoothedfield}}

In \secref{shutoffnoise-oneinterval}, we estimated the effect on the
solution of turning off the noise on a given time interval. In this
section we seek a further simplification. After an interval of time
in which the noise has been turned off, the resulting solution will
have been undergoing nothing more than the deterministic heat equation on
that interval. Therefore, it will have been smoothed, with a strength depending on the length of the interval. By restricting our attention to a comparatively small spatial region,
we would expect that the solution may be replaced by a constant at
the end of this interval. The following proposition is to quantify the induced error   when we replace the solution by a (random) constant at the end of each ``quiet'' interval.
\begin{prop}
	\label{prop:replacebyconstant}Let $A\subset[0,\infty)$ be measurable
	and let $\tau_{1}<\tau_{2}<T$ be such that $\tau_{2}=\sup A$ and
	$[\tau_{1},\tau_{2}]\subset A$. Fix $X\in\mathbf{R}^{2}$ and let
	$v$ solve the problem
	\begin{align}
		\dif v(t,x)   & =\frac{1}{2}\Delta v(t,x)\dif t+(\log\eps^{-1})^{-\frac{1}{2}}\sigma(v(t,x))\dif W^{\eps}(t,x),\qquad t>\tau_{2},\,x\in\mathbf{R}^{2};\label{eq:duxi-1} \\
		v(\tau_{2},x) & =u_{\eps,a}^{A}(\tau_{2},X).\label{eq:uxiic-1}
	\end{align}
	Then we have, for all $t\in[\tau_{2},T]$ and {$\eps\leq e^{-K_0^2}$}, that
	\begin{equation}
		\mathbf{E}(v-u_{\eps,a}^{A})(t,x)^{2}\le
		K_{0}^{4}a^{2}\frac{3(t-\tau_{2})+|x-X|^{2}}{\tau_{2}-\tau_{1}}.\label{eq:vuAbd}
	\end{equation}
\end{prop}

\begin{proof}
	We first note that  $u_{\eps,a}^A(\tau_2,X)=\int G_{\tau_2-\tau_1}(X-y)u_{\eps,a}^A(\tau_1,y)\dif y$, since $u_{\eps,a}^A$ solves the deterministic heat equation in the time interval $[\tau_1,\tau_2]$. Then,  we have for any $t>\tau_2$ that
	\begin{align*}
		(v-u_{\eps,a}^{A})(t,x) & =\int[G_{\tau_{2}-\tau_{1}}(X-y)-G_{t-\tau_{1}}(x-y)]u_{\eps,a}^{{A}}(\tau_{1},y)\,\dif y                                                    \\
		                        & \qquad+\frac{1}{\sqrt{\log\eps^{-1}}}\int_{\tau_{2}}^{t}\int G_{t-s}(x-y)[\sigma(v(s,y))-\sigma(u_{\eps,a}^{{A}}(s,y))]\,\dif W^{\eps}(s,y).
	\end{align*}
	Taking the second moment, 
	we obtain
		\begin{align} & \mathbf{E}(v-u_{\eps,a}^{A})(t,x)^{2}\notag  \\&\quad\leq\iint\mathbf{E}\prod_{i=1}^{2}\left([G_{\tau_{2}-\tau_{1}}(X-y_{i})-G_{t-\tau_{1}}(x-y_{i})]u_{\eps,a}^{A}(\tau_{1},y_{i})\right)\,\dif y_{1}\,\dif y_{2}                \notag       \\
			 & \qquad+\frac{\beta^{2}}{\log\eps^{-1}}\int_{\tau_{2}}^{t}\iint G_{\eps^{2}}(y_1-y_2)\mathbf{E}\prod_{i=1}^2\left(G_{t-s}(x-y_i)|v(s,y_i)-u_{\eps,a}^{A}(s,y_i)|\right)\,\dif y_{1}\,\dif y_{2}\,\dif s\notag\\&\quad\eqqcolon I_{1}+I_{2}.		\label{eq:I1I2}
		\end{align}
	For the first term, we estimate by the Cauchy--Schwarz inequality
	(on the probability space) that
	\begin{align}
		I_{1} & \le\left(\int|G_{\tau_{2}-\tau_{1}}(X-y)-G_{t-\tau_{1}}(x-y)|\left(\mathbf{E}u_{\eps,a}^{A}(\tau_{1},y)^{2}\right)^{1/2}\,\dif y\right)^{2}\nonumber \\
		      & \le K_{0}^{2}a^{2}\|G_{\tau_{2}-\tau_{1}}(X-\cdot)-G_{t-\tau_{1}}(x-\cdot)\|_{L^{1}(\mathbf{R}^{2})}^{2},\label{eq:I1initial}
	\end{align}
	where the second inequality is by \eqref{L2bound}. By Pinsker's inequality
	(see e.g. \cite[Lemma 1.5.3 and Theorem 1.5.4]{Iha93}), we have
	\begin{equation}
		\|G_{\tau_{2}-\tau_{1}}(X-\cdot)-G_{t-\tau_{1}}(x-\cdot)\|_{L^{1}(\mathbf{R}^{2})}^{2}\le2D_{\mathrm{KL}}(G_{t-\tau_{1}}(x-\cdot)\parallel G_{\tau_{2}-\tau_{1}}(X-\cdot)),\label{eq:Pinsker}
	\end{equation}
	where $D_{\mathrm{KL}}$ denotes the Kullback--Leibler divergence (also known
	as the relative entropy). We recall that for two continuous probability distributions
	$F_1$ and $F_2$ on $\R^2$, the Kullback--Leibler divergence is defined as
	\[D_{\mathrm{KL}}(F_1\parallel F_2) = \int F_1(x)\log\frac{F_1(x)}{F_2(x)}\,\dif x.\]
	Then we can compute explicitly (see e.g. \cite[Theorem 1.8.2]{Iha93})
	that
	\begin{align}
		D_{\mathrm{KL}}(G_{t-\tau_{1}}(x-\cdot)\parallel G_{\tau_{2}-\tau_{1}}(X-\cdot)) & =\log\frac{\tau_{2}-\tau_{1}}{t-\tau_{1}}-1+\frac{t-\tau_{1}}{\tau_{2}-\tau_{1}}+\frac{|X-x|^{2}}{2(\tau_{2}-\tau_{1})}\notag\\&\le\frac{t-\tau_{2}+\frac{1}{2}|X-x|^{2}}{\tau_{2}-\tau_{1}}.\label{eq:explicitKL}
	\end{align}
	Substituting \eqref{explicitKL} into \eqref{Pinsker} and then into
	\eqref{I1initial}, we have
	\begin{equation}
		I_{1}\le\frac{K_{0}^{2}a^{2}}{\tau_{2}-\tau_{1}}[2(t-\tau_{2})+|X-x|^{2}].\label{eq:I1final}
	\end{equation}

	Considering the second term of \eqref{I1I2}, we apply the inequality $|ab|\leq \tfrac12 (a^2+b^2)$ and use the symmetry in $y_1,y_2$ to derive
		\begin{align*}
	 I_2 &\le \frac{\beta^{2}}{2\log\eps^{-1}}\sum_{j=1}^2\int_{\tau_{2}}^{t}\iint G_{\eps^{2}}(y_1-y_2)\mathbf{E}|v(s,y_j)-u^A_{\eps,a}(s,y_j)|^2\prod_{i=1}^2 G_{t-s}(x-y_i)\,\dif y_{1}\,\dif y_{2}\,\dif s\\
	 &= \frac{\beta^{2}}{\log\eps^{-1}}\int_{\tau_{2}}^{t}\int G_{t-s+\eps^{2}}(x-y) G_{t-s}(x-y)\mathbf{E}|v(s,y)-u^A_{\eps,a}(s,y)|^2\,\dif y\,\dif s.
	\end{align*}
	Recalling the simple
	fact that in $d=2$,
	\begin{equation}
		G_{t_{1}}(\cdot)G_{t_{2}}(\cdot)=\frac{1}{2\pi(t_{1}+t_{2})}G_{\frac{t_{1}t_{2}}{t_{1}+t_{2}}}(\cdot),\label{eq:Gaussianproduct}
	\end{equation}
	for any $t_{1},t_{2}>0$, we further obtain
	\begin{equation}
		I_{2}\leq\frac{\beta^{2}}{4\pi\log\eps^{-1}}\int_{\tau_{2}}^{t}\int \frac{1}{t-s+\eps^{2}/2} G_{\frac{(t-s)(t-s+\eps^{2})}{2(t-s)+\eps^{2}}}(x-y)\mathbf{E}[v(s,y)-u_{\eps,a}^{A}(s,y)]^{2}\,\dif y\,\dif s.\label{eq:I2identity}
	\end{equation}
	Using \eqref{I1final} and \eqref{I2identity} in \eqref{I1I2}, we
	obtain
	\begin{align*}
		&\mathbf{E}(v-u_{\eps,a}^{A})(t,x)^{2} \\& \quad\le\frac{K_{0}^{2}a^{2}}{\tau_{2}-\tau_{1}}[2(t-\tau_{2})+|X-x|^{2}]                                                                                                                                      \\
		                                      & \qquad+\frac{\beta^{2}}{4\pi\log\eps^{-1}}\int_{\tau_{2}}^{t}\int \frac{1}{t-s+\eps^{2}/2}G_{\frac{(t-s)(t-s+\eps^{2})}{2(t-s)+\eps^{2}}}(x-y)\mathbf{E}[v(s,y)-u_{\eps,a}^{A}(s,y)]^{2}\,\dif y\,\dif s.
	\end{align*}
	Thus the hypotheses of \lemref{spatialrecursive} below are satisfied
	with
	\begin{align*}
		f(t,x) & =\mathbf{E}(v-u_{\eps,a}^{A})(t,x)^{2}, & A_{1} & =2K_{0}^{2}a^{2}\frac{t-\tau_{2}}{\tau_{2}-\tau_{1}}, & A_{2} & =\frac{K_{0}^{2}a^{2}}{\tau_{2}-\tau_{1}},
	\end{align*}
	from which we obtain
	\[
		\mathbf{E}(v-u_{\eps,a}^{A})(t,x)^{2}\le K_{0}^{4}a^{2}\left(\frac{2(t-\tau_{2})}{\tau_{2}-\tau_{1}}+\frac{\beta^{2}{K_0^2(t-\tau_2)}}{2\pi(\tau_{2}-\tau_{1})\log\eps^{-1}}+\frac{|x-X|^{2}}{\tau_{2}-\tau_{1}}\right),
	\]
	hence \eqref{vuAbd}, since we have $K_0^2<\log \eps^{-1}$ by assumption.
\end{proof}
It remains to prove the lemma we used above.
\begin{lem}
	\label{lem:spatialrecursive}Suppose that $0\le\tau\le T\le T_{0}$, $\sup_{t\in [\tau,T],x\in\R^2}|f(t,x)|<\infty$, and there exist constants $A_1,A_2$ such that
	\begin{equation}
		f(t,x)\le A_{1}+A_{2}|x-X|^{2}+\frac{\beta^{2}}{4\pi\log\eps^{-1}}\int_{\tau}^{t}\int\frac{1}{t-s+\eps^{2}/2}G_{\frac{(t-s)(t-s+\eps^{2})}{2(t-s)+\eps^{2}}}(x-y)f(s,y)\,\dif y\,\dif s\label{eq:fspatialrecursive}
	\end{equation}
	for all $t\in[\tau,T]$ and all $x\in\mathbf{R}^{2}$. Then,
	for all $t\in[\tau,T]$ and all $x\in\mathbf{R}^{2}$, we have
	\begin{equation}
		f(t,x)\le K_{0}^{2}\left(A_{1}+\frac{\beta^{2}(t-\tau)}{2\pi\log\eps^{-1}}{K_0^2}A_{2}+A_{2}|x-X|^{2}\right).\label{eq:ftxbd}
	\end{equation}
\end{lem}

\begin{proof}
	Suppose that
	\begin{equation}
		f(t,y)\le B_{1}+B_{2}|y-{X}|^{2}\label{eq:fB1B2bd}
	\end{equation}
	for all $t\in[\tau,T]$ and all $y\in\mathbf{R}^{2}$, where $B_1,B_2\geq0$ are constants. Of course
	this holds for
	\[
		B_{1}=\sup_{t\in [\tau,T],x\in\R^2}|f(t,x)|, \quad\quad B_{2}=0.
	\] Assuming \eqref{fB1B2bd}, we compute from \eqref{fspatialrecursive}
	that
	\begin{equation}
		f(t,x)\le A_{1}+A_{2}|x-X|^{2}+\frac{\beta^{2}}{4\pi\log\eps^{-1}}\int_{\tau}^{t}\int\frac{1}{t-s+\eps^{2}/2}G_{\frac{(t-s)(t-s+\eps^{2})}{2(t-s)+\eps^{2}}}(x-y)[B_{1}+B_{2}|y-X|^{2}]\,\dif y\,\dif s.\label{eq:ftyspatialapplyinductive}
	\end{equation}
	Now we can evaluate the spatial integral by noting that
	\[
		\int G_{\frac{(t-s)(t-s+\eps^{2})}{2(t-s)+\eps^{2}}}(x-y)|y-X|^{2}\,\dif y=\frac{2(t-s)(t-s+\eps^{2})}{2(t-s)+\eps^{2}}+|x-X|^{2}\le t-s+\eps^{2}+|x-X|^{2}.
	\]
	This implies that
	\begin{align*}
		\int_{\tau}^{t}\int & \frac{G_{\frac{(t-s)(t-s+\eps^{2})}{2(t-s)+\eps^{2}}}(x-y)}{t-s+\eps^{2}/2}[B_{1}+B_{2}|y-X|^{2}]\,\dif y\,\dif s\leq\int_{\tau}^{t}\frac{B_{1}+B_{2}(t-s+\eps^{2})+B_{2}|x-X|^{2}}{t-s+\eps^{2}/2}\,\dif s \\
		                    & \le\left(B_{1}+B_{2}|x-X|^{2}\right)\log\frac{t-\tau+\eps^{2}/2}{\eps^{2}/2}+2B_{2}(t-\tau)\le\left(B_{1}+B_{2}|x-X|^{2}\right)\log(1+2\eps^{-2}T)+2B_{2}(t-\tau).
	\end{align*}
	Substituting this back into \eqref{ftyspatialapplyinductive} and
	rearranging (also recalling \eqref{K0lb}), we obtain
	\begin{align}
		f(t,x) & \le A_{1}+A_{2}|x-X|^{2}+\frac{\beta^{2}}{4\pi\log\eps^{-1}}\left[\left(B_{1}+B_{2}|x-X|^{2}\right)\log(1+2\eps^{-2}T)+2B_{2}(t-\tau)\right]\nonumber         \\
		       & \le\left(A_{1}+(1-K_{0}^{-2})B_{1}+\frac{\beta^{2}(t-\tau)}{2\pi\log\eps^{-1}}B_{2}\right)+\left(A_{2}+(1-K_{0}^{-2})B_{2}\right)|x-X|^{2}.\label{eq:ftxineq}
	\end{align}
	Let $B_{1}^{(0)}=\sup_{t\in [\tau,T],x\in\R^2}|f(t,x)|$,
	$B_{2}^{(0)}=0$, and
	\begin{align}
		B_{1}^{(n)} & =A_{1}+(1-K_{0}^{-2})B_{1}^{(n-1)}+\frac{\beta^{2}(t-\tau)}{2\pi\log\eps^{-1}}B_{2}^{(n-1)},\label{eq:B1n-spatial} \\
		B_{2}^{(n)} & =A_{2}+(1-K_{0}^{-2})B_{2}^{(n-1)}\label{eq:B2n-spatial}
	\end{align}
	for each $n\ge1$. By \eqref{ftxineq} and induction, \eqref{fB1B2bd}
	holds with $B_{1}=B_{1}^{(n)}$ and $B_{2}=B_{2}^{(n)}$ for all $n$.
	From \eqref{B2n-spatial} we see that
	\[
		B_{2}^{(n)}\le K_{0}^{2}A_{2}
	\]
	for all $n$, and thus from \eqref{B1n-spatial} we obtain
	\[
		\limsup_{n\to\infty}B_{1}^{(n)}\le K_{0}^{2}\left(A_{1}+\frac{\beta^{2}(t-\tau)}{2\pi\log\eps^{-1}}{K_0^2}A_{2}\right).
	\]
	Using the last two displays in \eqref{fB1B2bd}, we obtain \eqref{ftxbd}.
\end{proof}

\section{The time discretization and the approximating functions\label{sec:timediscretization}}

In this section, we will iterate \propref[s]{exciseoneinterval} and~\propref{replacebyconstant}
on many subintervals of time to construct a discrete Markov chain which approximates the marginal distribution of the solution to the SPDE.  First we  construct these intervals, which will correspond to our time-discretization scheme.

\subsection{The time discretization}

Our approximation scheme will ultimately be focused on approximating the distribution of $u_{\eps,a}$ at
a single space-time point $(T,X)$. The time intervals of interest
thus depend on the terminal time $T$.

For $\eps\in(0,\eps_{0}]$, define $\delta_{\eps}$, $\gamma_{\eps}$, $\zeta_{\eps}$, and $\lambda_\eps$ such that
\begin{align}
	(\log\eps^{-1})^{-1}                              & \ll\gamma_{\eps}\ll\delta_{\eps}^{2}\ll\lambda_\eps\ll1,\label{eq:deltagammabd} \\
	\delta_{\eps}^{-1}\eps^{\frac{1}{2}\gamma_{\eps}} & \ll1,\label{eq:deltastupidbound}                                                \\
	(\log\eps^{-1})^{-1}                              & \ll\zeta_{\eps}\ll1,\label{eq:zetabound}
\end{align}
where the notation $f(\eps)\ll g(\eps)$ means that $f(\eps)\le g(\eps)$
for all $\eps$ and $\lim\limits _{\eps\downarrow0}\frac{f(\eps)}{g(\eps)}=0$.
To avoid introducing further constants later on, we
further assume that
\begin{equation}
	\max\{\eps^{\gamma_{\eps}},\eps^{\delta_{\eps}/2}\}\le1/2\label{eq:epsgammaepslt12}
\end{equation}
for all $\eps>0$. The choices of the parameters will become more clear later; see the discussion at the end of this subsection.

Now we define, for $T>\eps^{2-\lambda_\eps}$,
\begin{equation}
	s_{m}=\eps^{m\delta_{\eps}}\qquad\text{and}\qquad s'_{m}=\eps^{m\delta_{\eps}+\gamma_{\eps}}\label{eq:sdef}
\end{equation}
and
\begin{equation}
	t_{m}=T-s_{m}\qquad\text{and}\qquad t_{m}'=T-s_{m}'.\label{eq:tdef}
\end{equation}
Note that these quantities all depend on $\eps$, and $t_{m}$ and
$t'_{m}$ also depend on $T$, but we suppress this to simplify notations. We note that the time of interest $T$, unlike the
time horizon $T_{0}$, is \emph{not} fixed throughout the paper. However,
whenever we use $t_{m}$ and $t_{m}'$, the $T$ of current interest
will be clear from the context.

Define
\begin{align}
	M_{1}(\eps,T) & =\lceil\delta_{\eps}^{-1}\log_{\eps}T\rceil-1,\label{eq:M1def}     \\
	M_{2}(\eps)   & =\lfloor\delta_{\eps}^{-1}(2-\zeta_{\eps})\rfloor.\label{eq:M2def}
\end{align}
Thus $M_{1}(\eps,T)+1$ is the least integer $m$ so that $t_{m}\ge0$,
and $M_{2}(\eps)$ is the greatest integer $m$ so that $s_{m}\ge\eps^{2-\zeta_{\eps}}$.
For example, for fixed $T>0$ independent of $\eps$, we have  for sufficiently small $\eps$ that
\[
	M_1(\eps,T)=\left\{\begin{array}{ll}
		-1, & \text{ if }T\geq 1,     \\
		0,  & \text{ if } T\in (0,1).
	\end{array}
	\right.
\]
For the discrete time Markov chain to be constructed, the starting point in time will be given by $M_{1}(\eps,T)$,
and the ending point will be given by $M_{2}(\eps)$. We note for
future use that
\begin{equation}
	M_{2}(\eps)-M_{1}(\eps,T)+1\le\delta_{\eps}^{-1}(2-\log_{\eps}T).\label{eq:M2minusM1}
\end{equation}

Note that by the assumption of $\delta_\eps>\gamma_\eps$ and $\eps^{\gamma_\eps}<1$, we have
\[
	\begin{aligned}
		 & t_{m+1}=T-\eps^{m\delta_\eps+\delta_\eps}>T-\eps^{m\delta_\eps+\gamma_\eps}=t_m', \\
		 & t_{m}'=T-\eps^{m\delta_\eps+\gamma_\eps}>T-\eps^{m\delta_\eps}=t_m.
	\end{aligned}
\]
Thus we can write %
\[
	[t_{M_1(\eps,T)+1},t_{M_2(\eps)}']=I_{1}\cup I_{2},\quad  \text{ with } I_{1}=\bigcup_{m=M_1(\eps,T)+1}^{M_2(\eps)} [t_m,t_m'], \quad I_{2}=\bigcup_{m=M_1(\eps,T)+1}^{M_2(\eps)-1} [t_m',t_{m+1}].
\]

To approximate $u_{\eps,a}(T,X)$, we will turn off the noise in  $I_{1}$, which consists of the ``quiet'' intervals. For each $m$, we first solve the deterministic heat equation in the interval $[t_m,t_m']$. Then we replace the solution at $(t_m',\cdot)$ by its value at $(t_m',X)$. In the next ``noisy'' interval $[t_m',t_{m+1}]$, we solve the stochastic heat equation with the corresponding ``constant'' initial data. The error incurred in those ``quiet'' intervals will be quantified by \propref{exciseoneinterval}, and is negligible as $\eps\to0$ by the assumption $\gamma_\eps\ll \delta_\eps^2$. The error incurred by modifying the initial data for those ``noisy'' intervals will be quantified by \propref{replacebyconstant}, and goes to zero by the assumption of $\delta_{\eps}^{-1}\eps^{\frac{1}{2}\gamma_{\eps}}\ll1$. The role of $\zeta_\eps$ is in \eqref{M2def} to provide a small amount of extra separation between the final $t_m$ and the time $T$, which will be needed for the last step of the approximation; see the proof of \propref{approxubymarkovchain} below.

In the inequality \eqref{M2minusM1}, we need $\log_\eps T<2$ for all $\eps\ll1$ so that the above construction makes sense with   $M_2(\eps)\geq M_1(\eps,T)$, and this prevents us from considering those $T$ of order $O(\eps^2)$. From  \propref{varbound}, we already know that, if $T$ is chosen so that $\log (1+2\eps^{-2}T)\ll \log \eps^{-1}$, the random noise plays no role in the short interval $[0,T]$, and we have $u_{\eps,a}(T,x)\to a$ as $\eps\to0$. Therefore, those small $T$ can be treated separately without constructing the   Markov chain. To unify the notations, we use the following conventions: 

\begin{enumerate}
  \item  If $T>\eps^{2-\lambda_\eps}$, we have $2-\log_\eps T\gg \delta_\eps$, and $M_1(\eps,T),M_2(\eps)$ are defined as above.

  \item If $T\in [0,\eps^{2-\lambda_\eps}]$, we have $\log (1+2\eps^{-2}T)\ll \log \eps^{-1}$ and hence $u_{\eps,a}(T)\to a$ as $\eps\to0$, and we simply define $M_1(\eps,T)=M_2(\eps)=1$.
\end{enumerate}

\subsection{The approximating functions}

As we have mentioned, our approximation will be focused on a particular
terminal space-time point $(T,X)$. So in this section we fix $T\ge0$,
$X\in\mathbf{R}^{2}$. To define our approximation, we introduce a
sequence of functions 
$\{w^{(m)}\}_{m=M_{1}(\eps,T),\ldots,M_{2}(\eps)}$
as follows. Define
\begin{equation}
	w^{(M_{1}(\eps,T))}(t,x)=u_{\eps,a}(t,x),\qquad t\ge0,x\in\R^2.\label{eq:M1ic}
\end{equation}
For $m\in\{M_{1}(\eps,T)+1,\ldots,M_{2}(\eps)\}$, we then inductively define $\{w^{(m)}(t,x):t\geq t_m',x\in \mathbf{R}^2\}$ %
to be the solution to
\begin{align}
	\dif w^{(m)}(t,x) & =\frac{1}{2}\Delta w^{(m)}(t,x)\dif t+(\log\eps^{-1})^{-\frac{1}{2}}\sigma(w^{(m)}(t,x))\dif W^{\eps}(t,x),\qquad t> t_{m}',x\in\mathbf{R}^{2},\label{eq:dwm} \\
	w^{(m)}(t_{m}',x) & =\int G_{t_{m}'-t_{m}}(X-y)w^{(m-1)}(t_{m},y)\,\dif y,\qquad x\in\mathbf{R}^{2}.\label{eq:wmic}
\end{align}
Therefore, %
$w^{(m)}$ solves \eqref{duxi}--\eqref{uxiic}
but with  constant initial condition at time $t_{m}'$. Recall that $X$ is fixed which is our reference spatial point. We note (recalling \eqref{tdef} and \eqref{M1def}) that (whenever $m\ge t_{M_1(\eps,T)}+1$) we have $t_m'\ge t_m\ge \cdots \ge t_{M_1(\eps,T)+1}'\ge t_{M_1(\eps,T)+1}\ge 0$ and so the 
initial conditions \eqref{wmic} are inductively well-defined. We also emphasize that the function $w^{(m)}$  depends on the parameters $\eps,a,T,X$, and  the simplified notation $w^{(m)}=w^{(m)}_{\eps,a,T,X}$ will be used when there is no confusion. We will make the dependence explicit when needed. It is worth mentioning that for those $T\leq \eps^{2-\lambda_\eps}$, we only have one element in the chain which is $w^{(1)}=u_{\eps,a}$.

To compare $u_{\eps,a}$ with $w^{(m)}$, it turns out to be convenient to introduce another sequence of functions $\{\tilde{w}^{(m)}\}_{m=M_1(\eps,T),\ldots,M_2(\eps)}$. Define $\{\tilde{w}^{(m)}(t,x):t\geq t_m',x\in \mathbf{R}^2\}$ %
as the solution to
\begin{align}
	 & \dif\tilde{w}^{(m)}(t,y)          =\frac{1}{2}\Delta\tilde{w}^{(m)}(t,x)\dif t+\frac{\mathbf{1}_{\mathbf{R}\setminus[t_{m+1},t_{m+1}']}(t)}{\sqrt{\log\eps^{-1}}}\sigma(\tilde{w}^{(m)}(t,x))\dif W^{\eps}(t,x),\quad t>t'_{m},x\in\mathbf{R}^{2},\label{eq:dwmtilde} \\
	 & \tilde{w}^{(m)}(t_{m}',x)         =\int G_{t_{m}'-t_{m}}(X-y)\tilde{w}^{(m-1)}(t_{m},y)\,\dif y,\qquad x\in\mathbf{R}^{2},m\ge M_{1}(\eps,T)+1,\label{eq:wmtildeic}                                                                                                   \\
	 & \tilde{w}^{(M_{1}(\eps,T))}(0,x) =a.\label{eq:wtildeM1ic}
\end{align}
For each $m\ge M_1(\eps,T)+1$, we note that since $\tilde{w}^{(m-1)}$ satisfies the unforced heat equation on the time interval $[t_m,t_m']$, the initial condition \eqref{wmtildeic} can be rewritten as
\begin{equation}
	\tilde{w}^{(m)}(t_{m}',x)=\tilde{w}^{(m-1)}(t_{m}',{X}).\label{eq:wmtildeicrephrased}
\end{equation}
We also have the following lemma relating $\tilde{w}^{(m)}$ to $w^{(m)}$.
\begin{lem}\label{lem:wmwmtilde}
 For all $m\in \{M_1(\eps,T),\ldots,M_2(\eps)\}$, we have
 $
  w^{(m)}(t,x)=\tilde{w}^{(m)}(t,x)
 $
 for all $t\in [t_m'\vee 0,t_{m+1}]$ and all $x\in\R^2$.
\end{lem}
\begin{proof}
 The proof is by induction on $m$. For $m=M_1(\eps,T)$, by \eqref{dwmtilde},  \eqref{wtildeM1ic}, and \eqref{M1ic}, we see that $\tilde{w}^{(M_1(\eps,T))}=u_{\eps,a} = w^{(M_1(\eps,T))}$ on $[0,t_{M_1(\eps,T)+1}]\times\R^2$. For the inductive step, if $m\ge M_1(\eps,T)-1$ and we assume that $w^{(m-1)}(t,x)=\tilde{w}^{(m-1)}(t,x)$ for all $(t,x)\in [t'_{m-1}\vee 0,t_{m}]\times\R^2$, then this in particular means that $w^{(m-1)}(t_{m},\cdot)=\tilde{w}^{(m-1)}(t_{m},\cdot)$. This means that the initial conditions \eqref{wmic} for $w^{(m)}$ and \eqref{wmtildeic} for $\tilde{w}^{(m)}$ (both imposed at time $t_m'$) agree. Since the evolution equations \eqref{dwm} and \eqref{dwmtilde} also agree on the ``noisy'' time interval $[t_m',t_{m+1}]$, this implies that $w^{(m)}=\tilde{w}^{(m)}$ on the time interval $[t_m',t_{m+1}]$ as well.
\end{proof}
By \lemref{wmwmtilde} and \eqref{wmtildeicrephrased}, we see that the initial condition \eqref{wmic} is equivalent to
\begin{equation}
	w^{(m)}(t_{m}',x)=\tilde{w}^{(m-1)}(t_{m}',X).\label{eq:wmicintermsofwmtilde}
\end{equation}
Thus, for each $m\geq M_1(\eps,T)$, we initiate $w^{(m)}$ and $\tilde{w}^{(m)}$ with the same data at $t=t_m'$, with $w^{(m)}$ solving the original stochastic heat equation for $t>t_m'$ and $\tilde{w}^{(m)}$ solving the equation with the noise turned off in $[t_{m+1},t_{m+1}']$.

Our goal in this section is to estimate the approximation error $|w^{(m)}(t,x)-u_{\eps,a}(t,x)|$ for $t\geq t_m'$, $x\in\mathbf{R}^2$, and $m\in\{M_1(\eps,T),\ldots,M_2(\eps)\}$. By definition, we have $w^{(M_1(\eps,T))}=u_{\eps,a}$, thus by applying triangle inequality it suffices  to estimate $w^{(m)}-w^{(m-1)}$for each $m$. We briefly explain below how it will be achieved, by applying the results from \secref[s]{shutoffnoise-oneinterval} and~\ref{sec:smoothedfield}. First, through $\tilde{w}^{(m-1)}$ we can write the difference as
\[
	w^{(m)}(t,x)-w^{(m-1)}(t,x)=[w^{(m)}(t,x)-\tilde{w}^{(m-1)}(t,x)]+[\tilde{w}^{(m-1)}(t,x)-w^{(m-1)}(t,x)], \quad\quad t\geq t_m',x\in\mathbf{R}^2.
\]
We bound the two terms separately:
\begin{enumerate}
	\item For the first error term $w^{(m)}(t,x)-\tilde{w}^{(m-1)}(t,x)$, we recall three facts (i) $w^{(m)}(t_m',\cdot)=\tilde{w}^{(m-1)}(t_m',X)$; (ii) $\tilde{w}^{(m-1)}$ solves the deterministic heat equation in the  interval $[t_m,t_m']$; (iii) for $t>t_m'$, $w^{(m)}$ and $\tilde{w}^{(m-1)}$ solve  the same stochastic heat equation. Therefore, the difference of $w^{(m)}(t,x)$ from $\tilde{w}^{(m-1)}(t,x)$ only comes from replacing the initial data $\tilde{w}^{(m-1)}(t_m',\cdot)$ by its value at $X$, which can be quantified by  \propref{replacebyconstant}.

	\item For the second error term $\tilde{w}^{(m-1)}(t,x)-w^{(m-1)}(t,x)$, we have  %
	      \[
		      w^{(m-1)}(t_{m-1}',\cdot)=\tilde{w}^{(m-1)}(t_{m-1}',\cdot)=\tilde{w}^{(m-2)}(t_{m-1}',X).
	      \] The equations satisfied by $w^{(m-1)}$ and $\tilde{w}^{(m-1)}$ in $t>t_{m-1}'$ are the same except that the noise is turned off in $[t_m,t_m']$ for $\tilde{w}^{(m-1)}$. Therefore, the error only comes from turning off the noise in $[t_m,t_m']$. This can be quantified by \propref{exciseoneinterval}.
\end{enumerate}

The following proposition is the main result of the section.
\begin{prop}
	\label{prop:approxubyw}
	Suppose that $(C_\eps)_{\eps>0}$ is an arbitrary  sequence of numbers such that $C_{\eps}\to\infty$ as $\eps\downarrow 0$,
	and that $c\in[0,1)$ is a fixed constant. Define the set
	\[
		S_{\eps,T_0,c}:=\left\{(T,a,k,t,x): T\in[0,T_0],\,a>0,\,M_1(\eps,T){+}1\leq k\leq M_2(\eps),\,t \in [T-c\eps^{k\delta_\eps+\gamma_\eps},T],\,x\in\mathbf{R}^2\right\}.
	\]Then we have
	\begin{equation}
		\lim_{\eps\downarrow0}\quad\sup_{(T,a,k,t,x)\in S_{\eps,T_0,c} }\frac{\left(\mathbf{E}(u_{\eps,a}-w_{\eps,a,T,X}^{(k)})(t,x)^{2}\right)^{1/2}}{a(1+C_{\eps}\eps^{-k\delta_{\eps}/2}|x-X|)}=0.\label{eq:uapproxbyw}
	\end{equation}
\end{prop}

In order to prove \propref{approxubyw}, we need the following second
moment bound.
\begin{lem}
	\label{lem:wL2bound}There is a constant $K_{1}<\infty$ so that if
	$T\in[0,T_{0}]$, $\eps\in(0,\eps_{0}]$, $m\in\{M_{1}(\eps,T),\ldots,M_{2}(\eps)\}$,
	$a>0$, %
	then we have for all $x\in \mathbf{R}^2$ that
	\begin{equation}
		\mathbf{E}w_{\eps,a,T,X}^{(m)}(t_{m}',x)^{2}\le K_{1}^{2}a^{2}.\label{eq:Ewmbd}
	\end{equation}
\end{lem}

\begin{proof}
	Throughout the proof, we will again use the simplified notation $w^{(m)},\tilde{w}^{(m)}$. %
	Consider a fixed $m$. For all $t\ge t_{m-1}'$, by the mild formulation of the equation satisfied by $\tilde{w}^{(m-1)}$ and \eqref{wmtildeicrephrased}, we have
	\begin{align}
		\mathbf{E}\tilde{w}^{(m-1)}(t,X)^{2} & =\mathbf{E}\tilde{w}^{(m-2)}(t_{m-1}',X)^{2}+\frac{1}{\log \eps^{-1}}\int_{[t_{m-1}',t]\setminus[t_{m},t_{m}']}\iint G_{t-s}(X-y_{1})G_{t-s}(X-y_{2})G_{\eps^2}(y_1-y_2)\nonumber                                             \\
		                                     & \qquad\qquad\qquad\qquad\qquad\qquad\qquad\qquad\cdot\mathbf{E}\left[\sigma(\tilde{w}^{(m-1)}(s,y_{1}))\sigma(\tilde{w}^{(m-1)}(s,y_{2}))\right]\,\dif y_{1}\,\dif y_{2}\,\dif s\nonumber                                     \\
		                                     & \le\mathbf{E}\tilde{w}^{(m-2)}(t_{m-1}',X)^{2}+\frac{\beta^{2}}{2\pi\log\eps^{-1}}\int_{[t_{m-1}',t]\setminus[t_{m},t_{m}']}\frac{\mathbf{E}[\tilde{w}^{(m-1)}(s,X)^{2}]}{2(t-s)+\eps^{2}}\,\dif s.\label{eq:wtildeinductive}
	\end{align}
	Here we used the fact that $\tilde{w}^{(m-1)}$ is stationary in the spatial variable. In particular, we have
	\[\mathbf{E}\tilde{w}^{(m-1)}(t,X)^{2}\le \mathbf{E}\tilde{w}^{(m-2)}(t_{m-1}',X)^{2}+\frac{\beta^{2}}{2\pi\log\eps^{-1}}\int_{t'_{m-1}}^t\frac{\mathbf{E}[\tilde{w}^{(m-1)}(s,X)^{2}]}{2(t-s)+\eps^{2}}\,\dif s,\]
	which by \lemref{secondmomentbound-general} (taking there $f(s) = \mathbf{E}\tilde{w}^{(m-1)}(t_{m-1}'+s,X)^{2}$, and also using \defref{K0def}), implies that
	\[
		\mathbf{E}\tilde{w}^{(m-1)}(t,X)^{2}\le K_{0}^{2}\mathbf{E}\tilde{w}^{(m-2)}(t_{m-1}',X)^{2}.
	\]
	Substituting this back into \eqref{wtildeinductive}, taking $t=t_{m}'$,
	and recalling \eqref{wmicintermsofwmtilde}, we have %
	\begin{align}
		\mathbf{E}w^{(m)}(t_{m}',X)^{2} & =\mathbf{E}\tilde{w}^{(m-1)}(t_m',X)^{2}\notag\\&\le\mathbf{E}\tilde{w}^{(m-2)}(t_{m-1}',X)^{2}\left(1+\frac{K_{0}^{2}\beta^{2}}{4\pi\log\eps^{-1}}\int_{t_{m-1}'}^{t_{m}}\frac{1}{t_{m}'-s+\eps^{2}/2}\,\dif s\right)\nonumber \\
		                                & \le\mathbf{E}\tilde{w}^{(m-2)}(t_{m-1}',X)^{2}\left(1+\frac{K_{0}^{2}\beta^{2}}{4\pi\log\eps^{-1}}\log\frac{t_{m}'-t_{m-1}'}{t_{m}'-t_{m}}\right).\label{eq:Ewm}
	\end{align}
	The logarithm   can be estimated as
	\begin{align*}
		\log\frac{t_{m}'-t_{m-1}'}{t_{m}'-t_{m}} & =\log\frac{\eps^{(m-1)\delta_{\eps}+\gamma_{\eps}}-\eps^{m\delta_{\eps}+\gamma_{\eps}}}{\eps^{m\delta_{\eps}}-\eps^{m\delta_{\eps}+\gamma_{\eps}}}=\log\frac{\eps^{\gamma_\eps-\delta_{\eps}}-\eps^{\gamma_{\eps}}}{1-\eps^{\gamma_{\eps}}} \\
		                                         & \le\delta_{\eps}\log\eps^{-1}+\log\frac{\eps^{\gamma_{\eps}}}{1-\eps^{\gamma_{\eps}}}\le\delta_{\eps}\log\eps^{-1},
	\end{align*}
	where the last inequality is by \eqref{epsgammaepslt12}. Substituting
	this back into \eqref{Ewm}, we have
	\[
		\mathbf{E}w^{(m)}(t_{m}',X)^{2}=\mathbf{E}\tilde{w}^{(m-1)}(t_m',X)^{2}\le\mathbf{E}\tilde{w}^{(m-2)}(t_{m-1}',X)^{2}\left(1+\frac{K_{0}^{2}\beta^{2}\delta_{\eps}}{4\pi}\right).
	\]
	Iterating this and recalling \eqref{M2minusM1}, we have for all $x\in \mathbf{R}^2$,
	\begin{align*}
		\mathbf{E}w^{(m)}(t_{m}',x)^{2}=\mathbf{E}w^m(t_m',X)^2&\le {K_0^2}a^{2}\left(1+\frac{K_{0}^{2}\beta^{2}\delta_{\eps}}{4\pi}\right)^{m-M_{1}(\eps,T)} \\& \le {K_0^2}a^{2}\exp\left\{ \frac{\beta^{2}}{4\pi}K_{0}^{2}\delta_{\eps}(m-M_{1}(\eps,T))\right\} \\
		                                                                                                                                                     & \le {K_0^2}a^{2}\exp\left\{ \frac{2-\log_{\eps}T}{4\pi}\beta^{2}K_{0}^{2}\right\}
	\end{align*}
	for all $m\le M_{2}(\eps)$. The exponential on the right-hand side
	is uniformly bounded over all $T\le T_{0}$ and all $\eps\in(0,\eps_{0}]$,
	so we obtain \eqref{Ewmbd}.
\end{proof}
Now we can prove \propref{approxubyw}.
\begin{proof}[Proof of \propref{approxubyw}.]
	For any $t\in[T-c\eps^{m\delta_{\eps}+\gamma_{\eps}},T]$, we clearly have $t\geq t_m'$. By \propref{replacebyconstant} and \lemref{wL2bound}, we have
	\begin{align}
		\mathbf{E}(w^{(m)}-\tilde{w}^{(m-1)})(t,x)^{2} & \le K_{0}^{4}\left(\frac{3(t-t_{m}')+|x-X|^{2}}{t_{m}'-t_{m}}\right)\mathbf{E}\tilde{w}^{(m-1)}(t_{m-1}',X)^{2}\nonumber \\
		                                               & \le K_{0}^{4}K_{1}^{2}a^{2}\left(\frac{3(t-t_{m}')+|x-X|^{2}}{t_{m}'-t_{m}}\right).\label{eq:applyreplacebyconstant}
	\end{align}
	We have $t_{m}'-t_{m}=\eps^{m\delta_{\eps}}(1-\eps^{\gamma_{\eps}})\in\left[\frac{1}{2}\eps^{m\delta_{\eps}},\eps^{m\delta_{\eps}}\right]$
	by \eqref{epsgammaepslt12}, and $t-t_{m}'\le T-t_{m}'=\eps^{m\delta_{\eps}+\gamma_{\eps}}$
	by \eqref{tdef}, so \eqref{applyreplacebyconstant} yields
	\begin{align}
		\mathbf{E}(w^{(m)}-\tilde{w}^{(m-1)})(t,x)^{2} & \le K_{0}^{4}K_{1}^{2}a^{2}\left({6}\eps^{\gamma_{\eps}}+{2}\eps^{-m\delta_{\eps}}|x-X|^{2}\right).\label{eq:replacebyconstant-final}
	\end{align}

	On the other hand, by \propref{exciseoneinterval} and \lemref{wL2bound},
	we have for all $t\ge t_{m}'$ that
	\begin{align}
		\mathbf{E}(\tilde{w}^{(m-1)}-w^{(m-1)})(t,x)^{2} & \le\frac{\beta^{2}K_{0}^{{4}}}{4\pi\log\eps^{-1}}\left(\log\frac{t-t_{m}+\eps^{2}/2}{t-t_{m}'+\eps^{2}/2}+K_{0}^{2}\right)\mathbf{E}\tilde{w}^{({m-2})}(t_{m-1}',X)^{2}\nonumber \\
		                                                 & \le\frac{\beta^{2}K_{0}^{{4}}K_{1}^{2}a^{2}}{4\pi\log\eps^{-1}}\left(\log\frac{t-t_{m}+\eps^{2}/2}{t-t_{m}'+\eps^{2}/2}+K_{0}^{2}\right).\label{eq:applyexciseinterval}
	\end{align}
	We have $t-t_{m}\le T-t_{m}=\eps^{m\delta_{\eps}}$ and $t-t_{m}'\ge T-c\eps^{m\delta_{\eps}+\gamma_{\eps}}-t_{m}'=(1-c)\eps^{m\delta_{\eps}+\gamma_{\eps}}$,
	so \eqref{applyexciseinterval} gives us
	\begin{equation}
		\mathbf{E}(\tilde{w}^{(m-1)}-w^{(m-1)})(t,x)^{2}\le\frac{\beta^{2}}{4\pi}K_{0}^{{4}}K_{1}^{2}a^{2}\left(\gamma_{\eps}+\frac{\log\frac{1}{1-c}+K_{0}^{2}}{\log\eps^{-1}}\right).\label{eq:excisebeginningpart-final}
	\end{equation}

	Iterating {\eqref{replacebyconstant-final}} and \eqref{excisebeginningpart-final}
	and using the triangle inequality, we get
	\begin{align*}
		 & \left(\mathbf{E}(w^{(k)}-u_{\eps,a})(t,x)^{2}\right)^{1/2}                                                                                                                                                                                               \\
		 & \quad\le K_{0}^{{2}}K_{1}a\sum_{m=M_{1}(\eps,T)+1}^{k}\left[\sqrt{6}\eps^{\gamma_{\eps}/2}+\sqrt{2}\eps^{-m\delta_{\eps}/2}|x-X|+\frac{\beta}{2\sqrt{\pi}}\left(\gamma_{\eps}^{1/2}+\sqrt{\frac{\log\frac{1}{1-c}+K_{0}^{2}}{\log\eps^{-1}}}\right)\right]             \\
		 & \quad\le K_{0}^{{2}}K_{1}a\left[(2-\log_{\eps}T)\delta_{\eps}^{-1}\left(\sqrt{6}\eps^{\gamma_{\eps}/2}+\frac{\beta\gamma_{\eps}^{1/2}}{2\sqrt{\pi}}+\frac{\sqrt{\log\frac{1}{1-c}}+K_{0}}{\sqrt{\log\eps^{-1}}}\right)+2^{3/2}|x-X|\eps^{-k\delta_{\eps}/2}\right],
	\end{align*}
	where in the last inequality we used \eqref{M2minusM1} and \eqref{epsgammaepslt12}.
	Therefore, we have (with $C_\eps$, as in the statement of the proposition, an arbitrary sequence so that $C_\eps\to\infty$ as $\eps\downarrow 0$)
	\[\frac{\left(\mathbf{E}(u_{\eps,a}-w^{(k)})(t,x)^{2}\right)^{1/2}}{a(1+C_{\eps}\eps^{-k\delta_{\eps}/2}|x-X|)}
	\le K_{0}^{{2}}K_{1}\left[(2-\log_{\eps}T)\delta_{\eps}^{-1}\left(\sqrt{6}\eps^{\gamma_{\eps}/2}+\frac{\beta\gamma_{\eps}^{1/2}}{2\sqrt{\pi}}+\frac{\sqrt{\log\frac{1}{1-c}}+K_{0}}{\sqrt{\log\eps^{-1}}}\right)+2^{3/2}C_\eps^{-1}\right].
	\]
	The first summand in the square brackets goes to $0$ as $\eps\downarrow0$
	by \eqref{deltagammabd} and \eqref{deltastupidbound}, and since we assumed that $C_\eps\to\infty$, we obtain
	\eqref{uapproxbyw}.
\end{proof}

\section{The discrete martingale\label{sec:discretemartingale}}

The key advantage of the approximation carried out in \propref{approxubyw}
is that we now have an essentially one-dimensional problem. Note from
the definitions \eqref{dwm}--\eqref{wmtildeic}, and also \eqref{wmicintermsofwmtilde} and the white-in-time nature of the noise, that if we (fix
once and for all $X\in\mathbf{R}^{2}$ and) define
\begin{align*}
	Y_{\eps,a,T}(M_{1}(\eps,T)) & =a;                                                                         \\
	Y_{\eps,a,T}(m)             & =w_{\eps,a,T,X}^{(m)}(t_{m}',X)=\tilde{w}_{\eps,a,T,X}^{(m-1)}(t_{m}',X),\qquad M_{1}(\eps,T)+1\le m\le M_{2}(\eps),
\end{align*}
then the process $\{Y_{\eps,a,T}(m)\}_{m=M_{1}(\eps,T),\ldots,M_{2}(\eps)}$
is a Markov chain  and
a martingale (both with respect to its own filtration). The key point is that $w_{\eps,a,T,X}^{(m)}$ evolves with spatially-constant initial condition $Y_{\eps,a,T}(m)$, driven by the  noise that is independent of the past. Thus $Y_{\eps,a,T}(m+1)$ depends on the past only via $Y_{\eps,a,T}(m)$. Moreover the expectation of $Y_{\eps,a,T}(m+1)$ conditional on $Y_{\eps,a,T}(m)$ is simply $Y_{\eps,a,T}(m)$ due to the fact that, when starting from constant initial data, the stochastic heat equation (with the noise either on or off) preserves expectations. Recall that in the case of (very small) $T\in [0,\eps^{2-\lambda_\eps}]$, we have defined $M_1(\eps,T)=M_2(\eps)=1$, and in this case we simply let $Y_{\eps,a,T}(M_{1}(\eps,T))=Y_{\eps,a,T}(M_2(\eps))=u_{\eps,a}(T,X)$.

\subsection{Approximating the one-point SPDE solution by the Markov chain}

In this section we show that $Y_{\eps,a,T}(M_{2}(\eps))$, at its
terminal time $M_{2}(\eps)$, is a good approximation for $u_{\eps,a}(T,X)$
(in fact, for $u_{\eps,a}(T,x)$ if $x$ is close to $X$). Most of
the work has already been done in \propref{approxubyw}.
\begin{prop}
	\label{prop:approxubymarkovchain}We have
	\begin{equation}
		\adjustlimits\lim_{\substack{\eps\downarrow0\vphantom{T\in[0,T_{0}]}\\
		\vphantom{a>0,x\in\mathbf{R}^{2}}
		}
		}\sup_{\substack{T\in[0,T_{0}]\\
		a>0,x\in\mathbf{R}^{2}
		}
		}\frac{\left(\mathbf{E}(Y_{\eps,a,T}(M_{2}(\eps))-u_{\eps,a}(T,x))^{2}\right)^{1/2}}{a(1+\eps^{-1}|x-X|)}=0.\label{eq:approxubymarkovchain}
	\end{equation}
\end{prop}

\begin{proof}
By \propref{approxubyw} (choosing $C_{\eps}=\eps^{-\zeta_{\eps}/2}\to\infty$
	by \eqref{zetabound}), we have
	\begin{equation}
		\lim_{\substack{\eps\downarrow0\vphantom{T\in[0,T_{0}]}\\
		\vphantom{a>0,x\in\mathbf{R}^{2}}
		}
		}\sup_{\substack{T\in[0,T_{0}]\\
		a>0,x\in\mathbf{R}^{2}
		}
		}\frac{\left(\mathbf{E}(u_{\eps,a}-w_{\eps,a,T,X}^{({M_2(\eps)})})({T},x)^{2}\right)^{1/2}}{a(1+\eps^{-\frac{1}{2}({M_{2}(\eps)\delta_{\eps}}+\zeta_{\eps})}|x-X|)}=0.\label{eq:uapproxbyw-1}
	\end{equation}
	By \eqref{M2def}, we have $\frac{1}{2}({M_{2}(\eps)\delta_{\eps}}+\zeta_{\eps})\le 1$.
	Therefore, \eqref{uapproxbyw-1} implies that
		\begin{equation}
		\lim_{\substack{\eps\downarrow0\vphantom{T\in[0,T_{0}]}\\
		\vphantom{a>0,x\in\mathbf{R}^{2}}
		}
		}\sup_{\substack{T\in[0,T_{0}]\\
		a>0,x\in\mathbf{R}^{2}
		}
		}\frac{\left(\mathbf{E}(u_{\eps,a}-w_{\eps,a,T,X}^{({M_2(\eps)})})({T},x)^{2}\right)^{1/2}}{a(1+\eps^{-1}|x-X|)}=0.\label{eq:uapproxbyw-2}
	\end{equation}
	Moreover, by \propref{varbound}, \lemref{wL2bound} {and the fact that $T-t_{M_2(\eps)}'=s_{M_2(\eps)}'<s_{M_2(\eps)}\leq \eps^{2-\zeta_\eps-\delta_\eps}$}, we have
	\begin{align*}
		 & \frac{1}{a}\left(\mathbf{E}(w_{\eps,a,T,X}^{(M_{2}(\eps))}(T,x)-Y_{\eps,a,T}(M_{2}(\eps)))^{2}\right)^{1/2}                                                                                                   \\
		 & \qquad\le\frac{\beta K_{0}}{2a\sqrt{\pi}}\left(\mathbf{E}w_{\eps,a,T,X}^{(M_{2}(\eps))}(t_{M_{2}(\eps)}',{X})^{2}\right)^{1/2}\left(\frac{\log (1+2\eps^{-2}(T-t_{M_2(\eps)}'))}{\log \eps^{-1}}\right)^{1/2} \\
		 & \qquad\le\frac{\beta K_{0}K_{1}}{2\sqrt{\pi}}\sqrt{\frac{\log(1+2\eps^{-\zeta_{\eps}-\delta_\eps})}{\log \eps^{-1}}},
	\end{align*}
	and the quantity on the right side goes to $0$ uniformly in $T,a,x$
	by \eqref{deltagammabd} and \eqref{zetabound}. This, along with \eqref{uapproxbyw-2}, implies
	\eqref{approxubymarkovchain}.
\end{proof}
The following spatial regularity statement for $u_{\eps,a}(T,\cdot)$
is a consequence of \propref{approxubymarkovchain}, so we record
it here for future use.
\begin{cor}
	\label{cor:regularity}We have
	\[
		\lim_{\substack{\eps\downarrow0\vphantom{T\in[0,T_{0}],a>0}\\
		\vphantom{x_{1},x_{2}\in\mathbf{R}^{2}}
		}
		}\sup_{\substack{T\in[0,T_{0}],a>0\\
		x_{1},x_{2}\in\mathbf{R}^{2}
		}
		}\frac{\left(\mathbf{E}(u_{\eps,a}(T,x_{1})-u_{\eps,a}(T,x_{2}))^{2}\right)^{1/2}}{a(1+\eps^{-1}|x_{1}-x_{2}|)}=0.
	\]
\end{cor}

\begin{proof}
	By spatial homogeneity, we can assume that $x_{1}=X$. Then the result
	follows immediately by writing
	\[
		\left(\mathbf{E}(u_{\eps,a}(T,x_{1})-u_{\eps,a}(T,x_{2}))^{2}\right)^{1/2}\le\sum_{i=1}^{2}\left(\mathbf{E}(u_{\eps,a}(T,x_{i})-Y_{\eps,a,T}(M_{2}(\eps)))^{2}\right)^{1/2}
	\]
	and applying \propref{approxubymarkovchain} to both terms.
\end{proof}

\subsection{The martingale differences}

The approximation result in \propref{approxubymarkovchain} motivates
us to study the discrete martingale $\{Y_{\eps,a,T}(m)\}_{m}$. %
Our
ultimate goal will be to show that it approximates a continuous martingale
(coming from a solution to \eqref{dXi-intro}--\eqref{Jdef-intro})
as $\eps\downarrow0$. We will use the martingale problem
approach as explained in \cite[Section 11.2]{SV06}, and en route it
will be important to understand some statistical properties of the increments
$Y_{\eps,a,T}(m)-Y_{\eps,a,T}(m-1)$ conditional on $Y_{\eps,a,T}(m-1)$,
a task to which we now set ourselves. The first observation is
that, due to the independence of $\dif W^{\eps}$ on disjoint time intervals, if we define
\begin{equation}\label{eq:Zdef}
	Z_{\eps,a,m}=\int G_{s_{m}-s'_{m}}({X}-z)u_{\eps,a}({s_{m-1}'}-s_{m},z)\,\dif z, \quad\quad M_1(\eps,T)+1\leq m\leq M_2(\eps),
\end{equation}
(with $s_{k},s_k'$ defined as in \eqref{sdef}) then
\begin{equation}
	\Law[Y_{\eps,a,T}(m)\mid Y_{\eps,a,T}(m-1)=b]=\Law Z_{\eps,b,m}.\label{eq:lawagreeswithZ}
\end{equation}
This can be seen by noting that the evolution equations for $u_{\eps,a}$ and $\tilde{w}^{(m)}$ are the same, and that $\tilde{w}^{(m)}$ is started with constant initial condition equal to $Y_{\eps,a,T}(m-1)$.

\subsubsection{Martingale difference variances}

Our first interest is in the conditional variance $\Var[Y_{\eps,a,T}(m)\mid Y_{\eps,a,T}(m-1)=b]=\Var Z_{\eps,b,m}$,
and we proceed to study this quantity. The first step is to approximate
it by a simpler quantity using the regularity established in \corref{regularity}.
An important role will be played by the function $J_{\eps}:(-\infty,2+\log_{\eps^{-1}}T_{0}]\times\mathbf{R}_{\ge0}\to\mathbf{R}_{\ge0}$
defined by
\begin{equation}
	J_{\eps}(q,a)=\frac{1}{2\sqrt{\pi}}(\mathbf{E}\sigma(u_{\eps,a}(\eps^{2-q},x))^{2})^{1/2}.\label{eq:Jepsdef}
\end{equation}
As $u_{\eps,a}$ is stationary in the spatial variable, the r.h.s. of \eqref{Jepsdef} does not depend on $x$. Here
\[
	q\in (-\infty,2+\log_{\eps^{-1}}T_{0}]\text{ corresponds to }\eps^{2-q}\in(0,T_0],
\]
i.e., we  parameterize $J_{\eps}$ in time on the exponential scale discussed in \subsecref{expscale}.
This section
has two main results. First, we show how to use $J_{\eps}$ to approximate
$\Var Z_{\eps,b,m}$:
\begin{prop}
	\label{prop:varapproxbyJ}We have
	\begin{equation}
		\adjustlimits\lim_{\substack{\eps\downarrow0\vphantom{T\in[0,T_{0}]}\\
		\vphantom{M_{1}(\eps,T)\le m\le M_{2}(\eps)}
		}
		}\sup_{\substack{T\in[0,T_{0}], a>0,\\
		M_{1}(\eps,T)+1\le m\le M_{2}(\eps)
		}
		}a^{-2}\left|\delta_{\eps}^{-1}\Var Z_{\eps,a,m}-J_{\eps}(2-(m-1)\delta_{\eps},a)^{2}\right|=0.\label{eq:Japproximatesvariance}
	\end{equation}
\end{prop}

Also, we will prove the following compactness result for the family
$\{J_{\eps}\}_\eps$, which will help us in our limit procedure:
\begin{prop}
	\label{prop:Jcompact}For any sequence $\eps_{k}\downarrow0$, there
	is a subsequence $\eps_{k_{\ell}}\downarrow0$ and a continuous function
	$J:[0,2]\times\mathbf{R}_{\ge0}\to\mathbf{R}_{\ge0}$ so that
	\begin{equation}
		\lim_{\ell\to\infty}J_{\eps_{k_{\ell}}}|_{[0,2]\times\mathbf{R}_{\ge0}}=J\label{eq:Jsubsequentiallimit}
	\end{equation}
	uniformly on compact subsets of $[0,2]\times\mathbf{R}_{\ge0}$.
\end{prop}

As we assumed that $T_{0}\ge1$, each $J_{\eps}$ is indeed
defined on $[0,2]{\times \mathbf{R}_{\geq0}}$. We will prove \propref{Jcompact} first, since
the intermediate results will be useful in the proof of \propref{varapproxbyJ}.
We need two preparatory lemmas, addressing the regularity of $J_{\eps}$
in $q$ and in $a$. First we address the regularity in $q$.
\begin{lem}
	\label{lem:Jtimelipschitz}For all $\eps,a>0$ and $q_{1},q_{2}\in(-\infty,2+\log_{\eps^{-1}}T_{0}]$,
	we have
	\begin{equation}
		|J_{\eps}(q_{2},a)-J_{\eps}(q_{1},a)|\le{\frac{a\beta^2 K_{0}^{2}}{4\pi}}\left(|q_{2}-q_{1}|^{1/2}+K_{0}(\log\eps^{-1})^{-1/2}\right).\label{eq:Jtimelipschitz}
	\end{equation}
\end{lem}

\begin{proof}
	Assume   $q_{1}\le q_{2}$. Define $I_\eps=[0,\eps^{2-q_2}-\eps^{2-q_1}]$. We can write
	\begin{align*}
		|J_{\eps}(q_{1},a)-J_{\eps}(q_{2},a)| & =\frac{1}{2\sqrt{\pi}}\left|(\mathbf{E}\sigma(u_{\eps,a}^{I_\eps}(\eps^{2-q_{2}},x))^{2})^{1/2}-(\mathbf{E}\sigma(u_{\eps,a}(\eps^{2-q_{2}},x))^{2})^{1/2}\right| \\
		                                      & \le\frac{\beta}{2\sqrt{\pi}}\left(\mathbf{E}\left(u_{\eps,a}^{I_\eps}(\eps^{2-q_{2}},x)-u_{\eps,a}(\eps^{2-q_{2}},x)\right)^{2}\right)^{1/2}.
	\end{align*}
	In the first equality we used the fact that
	\[
		u_{\eps,a}^{I_\eps}(\eps^{2-q_{2}},x)\overset{\mathrm{law}}{=}u_{\eps,a}(\eps^{2-{q_1}},x),
	\]
	where $u_{\eps,a}^{I_\eps}$ is defined
	as in \eqref{uA}--\eqref{uAic}, i.e., the noise is turned off in $I_\eps$. Now we apply \propref{exciseoneinterval}
	with $t=\eps^{2-q_{2}}$, $A=\emptyset$, $\tau_{1}=0$, and $\tau_{2}=\eps^{2-q_{2}}-\eps^{2-q_{1}}$
	to obtain
	\begin{align}
		|J_{\eps}(q_{1},a)-J_{\eps}(q_{2},a)| & \le{\frac{a\beta^2 K_{0}^{2}}{4\pi\sqrt{\log\eps^{-1}}}}\left(\sqrt{\log\frac{\eps^{2-q_{2}}+\eps^{2}/2}{\eps^{2-q_{1}}+\eps^{2}/2}}+K_{0}\right)\label{eq:Jprelimbd} \\
		                                      & \le{\frac{a\beta^2 K_{0}^{2}}{4\pi}}\left(|q_{2}-q_{1}|^{1/2}+K_{0}(\log\eps^{-1})^{-1/2}\right),\label{eq:q1q2}
	\end{align}
	as claimed.
\end{proof}
Now we address the regularity of $J_{\eps}$ in $a$.
Later on, we will also use the following result to prove that \eqref{hLipschitz-1}
is satisfied for the limits of $\{J_{\eps}\}_{\eps}$ as $\eps\downarrow0$.
Thus we need the explicit constant in the middle expression of \eqref{Jspacelipschitz}.
\begin{lem}
	\label{lem:Jspacelipschitz}For all $\eps\in(0,\eps_{0}]$, $q\in(-\infty,2+\log_{\eps^{-1}}T_{0}]$,
	and $a_{1},a_{2}\ge0$, we have
	\begin{equation}
		|J_{\eps}(q,a_{2})-J_{\eps}(q,a_{1})|\le\left(\frac{4\pi}{\beta^{2}}-\frac{\log(1+2\eps^{-q})}{\log\eps^{-1}}\right)^{-1/2}|a_{2}-a_{1}|\le\frac{\beta K_{0}}{2\sqrt{\pi}}|a_{2}-a_{1}|.\label{eq:Jspacelipschitz}
	\end{equation}
	In particular, for all $a>0$, we have
	\begin{equation}
		|J_{\eps}(q,a)|\le\frac{\beta aK_{0}}{2\sqrt{\pi}}.\label{eq:Jbd}
	\end{equation}
\end{lem}

\begin{proof}
	We have
	\begin{align*}
		\left|J_{\eps}(q,a_{1})-J_{\eps}(q,a_{2})\right| & =\frac{1}{2\sqrt{\pi}}\left|(\mathbf{E}\sigma(u_{\eps,a_{1}}(\eps^{2-q},x))^{2})^{1/2}-(\mathbf{E}\sigma(u_{\eps,a_{2}}(\eps^{2-q},x))^{2})^{1/2}\right| \\
		                                                 & \le\frac{1}{2\sqrt{\pi}}\left(\mathbf{E}[\sigma(u_{\eps,a_{1}}(\eps^{2-q},x))-\sigma(u_{\eps,a_{2}}(\eps^{2-q},x))]^{2}\right)^{1/2}                     \\
		                                                 & \le\frac{\beta}{2\sqrt{\pi}}\left(\mathbf{E}[u_{\eps,a_{1}}(\eps^{2-q},x)-u_{\eps,a_{2}}(\eps^{2-q},x)]^{2}\right)^{1/2},
	\end{align*}
	and then the first inequality in \eqref{Jspacelipschitz} follows
	from \eqref{diffudifferentxis} with the explicit constant \eqref{explicitK}.
	The second inequality in \eqref{Jspacelipschitz} is then just \eqref{K0lb}.
	The bound \eqref{Jbd} comes from \eqref{Jspacelipschitz} with $a_{2}=a$
	and $a_{1}=0$.
\end{proof}
Given the regularity results in \lemref[s]{Jtimelipschitz} and~\ref{lem:Jspacelipschitz},
the compactness of the family $(J_{\eps})$ is straightforward.
\begin{proof}[Proof of \propref{Jcompact}.]
	By \lemref[s]{Jtimelipschitz} and~\ref{lem:Jspacelipschitz}, along
	with a simple modification of the Arzelà--Ascoli theorem to account
	for the second term on the r.h.s. of  \eqref{Jtimelipschitz}
	(see e.g.~\cite[Lemma A.4]{DD20}), we can extract a suitable subsequence
	and pass to the limit on any rectangular subset of $[0,2]\times\mathbf{R}_{\ge0}$ of the form $[0,2]\times [0,M]$, with $M>0$. Sending $M\to\infty$ so that the rectangles exhaust $[0,2]\times\mathbf{R}_{\ge0}$
	and using a diagonalization argument, we obtain the desired limit
	and convergence \eqref{Jsubsequentiallimit}.
\end{proof}
Now we turn to the proof of \propref{varapproxbyJ}. We first prove
the following intermediate result.
\begin{lem}
	\label{lem:Vtildegoodapprox}Define
	\begin{equation}
	\begin{aligned}
		V_{\eps,a,m}&=\frac{1}{\log\eps^{-1}}\int_{0}^{s_{m-1}'-s_{m}}\frac{J_{\eps}(2-\log_{\eps}s,a)^{2}}{s_{m-1}'-s_{m}'-s+\eps^{2}/2}\,\dif s\\&=\frac{1}{2\pi\log\eps^{-1}}\int_{0}^{s_{m-1}'-s_{m}}\frac{\mathbf{E}\sigma(u_{\eps,a}(s,X))^{2}}{2(s_{m-1}'-s_{m}'-s)+\eps^{2}}\,\dif s.\end{aligned}\label{eq:Vtildedef}
	\end{equation}
	Then we have, for any fixed $T_{0}<\infty$, that
	\begin{equation}
		\adjustlimits\lim_{\substack{\eps\downarrow0\vphantom{T\in[0,T_{0}]}\\
		\vphantom{M_{1}(\eps,T)\le m\le M_{2}(\eps)}
		}
		}\sup_{\substack{T\in[0,T_{0}]\\
		M_{1}(\eps,T)+1\le m\le M_{2}(\eps)
		}
		}\frac{|V_{\eps,a,m}-\Var Z_{\eps,a,m}|}{a^{2}\delta_{\eps}}=0.\label{eq:Vtildegoodapprox}
	\end{equation}
\end{lem}

\begin{proof}
	We can first write (recalling \eqref{Zdef})
	\begin{align}
		Z_{\eps,a,m} & =a+\frac{1}{\sqrt{\log\eps^{-1}}}\int G_{s_{m}-s_{m}'}(X-z)\int_{0}^{s_{m-1}'-s_{m}}\int G_{s_{m-1}'-s_{m}-s}(z-y)\sigma(u_{\eps,a}(s,y))\,\dif W^{\eps}(s,y)\,\dif z\nonumber \\
		             & =a+\frac{1}{\sqrt{\log\eps^{-1}}}\int_{0}^{s_{m-1}'-s_{m}}\int G_{s_{m-1}'-s_{m}'-s}(X-y)\sigma(u_{\eps,a}(s,y))\,\dif W^{\eps}(s,y).\label{eq:Zstochasticintegral}
	\end{align}
	Therefore, we have
	\begin{equation}
		\begin{aligned}\Var Z_{\eps,a,m} & =\frac{1}{\log\eps^{-1}}\int_{0}^{s_{m-1}'-s_{m}}\iint G_{\eps^{2}}(y_{1}-y_{2})\mathbf{E}\prod_{i=1}^{2}\left(G_{s_{m-1}'-s_{m}'-s}(X-y_{i})\sigma(u_{\eps,a}(s,y_{i}))\right)\,\dif y_{1}\,\dif y_{2}\,\dif s.\end{aligned}
		\label{eq:VarZ}
	\end{equation}
	Now we have, by spatial homogeneity, that
	\[
		\mathbf{E}\sigma(u_{\eps,a}(s,y_{1}))\sigma(u_{\eps,a}(s,y_{2}))=\mathbf{E}\sigma(u_{\eps,a}(s,X))\sigma(u_{\eps,a}(s,X+y_{1}-y_{2})).
	\]
	We also have (using the Cauchy--Schwarz inequality) that
	\begin{align*}
		 & \left|\mathbf{E}\sigma(u_{\eps,a}(s,y_{1}))\sigma(u_{\eps,a}(s,y_{2}))-\mathbf{E}\sigma(u_{\eps,a}(s,y_{1}))^{2}\right|\\&\qquad\le\mathbf{E}\sigma(u_{\eps,a}(s,y_{1}))|\sigma(u_{\eps,a}(s,y_{1}))-\sigma(u_{\eps,a}(s,y_{2}))| \\
		 & \qquad\le\beta^{2}\left(\mathbf{E}u_{\eps,a}(s,y_{1})^{2}\right)^{1/2}\left(\mathbf{E}[u_{\eps,a}(s,y_{1})-u_{\eps,a}(s,y_{2})]^{2}\right)^{1/2},
	\end{align*}
	so by \eqref{L2bound} and \corref{regularity} we have a function
	$f$ satisfying $\lim\limits _{\eps\downarrow0}f(\eps)=0$ and
	\begin{equation}
		\sup_{\substack{s\in[0,{T_0}]\\
		y_{1},y_{2}\in\mathbf{R}^{2}
		}
		}\left|\mathbf{E}\sigma(u_{\eps,a}(s,y_{1}))\sigma(u_{\eps,a}(s,y_{2}))-\mathbf{E}\sigma(u_{\eps,a}(s,y_{1}))^{2}\right|\le a^{2}(1+\eps^{-1}|y_{1}-y_{2}|)f(\eps)\label{eq:diffofcorrs}
	\end{equation}
	for all $y_{1},y_{2}\in\mathbf{R}^{2}$ and all $a\ge0$. Now we note
	that
	\begin{align}
		\frac{1}{\log\eps^{-1}} & \int_{0}^{s_{m-1}'-s_{m}}\iint G_{s_{m-1}'-s_{m}'-s}(X-y)G_{s_{m-1}'-s_{m}'-s}(X-y')G_{\eps^{2}}(y-y')\mathbf{E}\sigma(u_{\eps,a}(s,y))^{2}\,\dif y\,\dif y'\,\dif s                                                                             \notag \\
		                        & =\frac{1}{\log\eps^{-1}}\int_{0}^{s_{m-1}'-s_{m}}\int G_{s_{m-1}'-s_{m}'-s}(X-y)G_{s_{m-1}'-s_{m}'-s+\eps^{2}}(X-y)\mathbf{E}\sigma(u_{\eps,a}(s,y))^{2}\,\dif y\,\dif s                                                                         \notag \\
		                        & =\frac{1}{2\pi\log\eps^{-1}}\int_{0}^{s_{m-1}'-s_{m}}\frac{\mathbf{E}\sigma(u_{\eps,a}(s,X))^{2}}{2(s_{m-1}'-s_{m}'-s)+\eps^{2}}\,\dif s=V_{\eps,a,m}.
		\label{eq:rewriteVepsxim}
	\end{align}
	where in the second-to-last identity we %
	used spatial homogeneity. Subtracting \eqref{VarZ}
	and \eqref{rewriteVepsxim} and applying \eqref{diffofcorrs}, we
	have
	\begin{equation}
		\begin{aligned}| & V_{\eps,a,m}-\Var Z_{\eps,a,m}|                                                                                                                                                                                   \\
			  & \le\frac{f(\eps)a^{2}}{\log\eps^{-1}}\int_{0}^{s_{m-1}'-s_{m}}\iint\left(\prod_{i=1}^{2}G_{s_{m-1}'-s_{m}'-s}(X-y_{i})\right)G_{\eps^{2}}(y_{1}-y_{2})(1+\eps^{-1}|y_{1}-y_{2}|)\,\dif y_{1}\,\dif y_{2}\,\dif s.
		\end{aligned}
		\label{eq:VVzbd}
	\end{equation}
	If we define {$h(r)=(2\pi)^{-1}e^{-\frac{r^2}{2}}(1+r)$ for $r\geq0$}, then the last double integral
	is equal to
	\begin{align*}
		\iint & \left(\prod_{i=1}^{2}G_{s_{m-1}'-s_{m}'-s}(X-y_{i})\right) \eps^{-2}h(\eps^{-1}|y_{1}-y_{2}|)\,\dif y_{1}\,\dif y_{2}                                \\
		      & \le\left(\int\eps^{-2}h(\eps^{-1}|y|)\,\dif y\right)\int G_{s_{m-1}'-s_{m}'-s}(X-y)^{2}\,\dif y=\frac{\int h(|y|)\,\dif y}{4\pi(s_{m-1}'-s_{m}'-s)},
	\end{align*}
	where the inequality is Young's convolution inequality. Substituting
	this back into \eqref{VVzbd}, we have
	\begin{align*}
		|V_{\eps,a,m}-\Var Z_{\eps,a,m}| & \le\frac{f(\eps)a^{2}\int h(|y|)\,\dif y}{4\pi\log\eps^{-1}}\int_{0}^{s_{m-1}'-s_{m}}\frac{1}{s_{m-1}'-s_{m}'-s}\,\dif s                                                                                                                                \\
		                                 & =\frac{f(\eps)a^{2}\int h(|y|)\,\dif y}{4\pi\log\eps^{-1}}\log\frac{s_{m-1}'-s_{m}'}{s_{m}-s_{m}'}=\frac{f(\eps)a^{2}\int h(|y|)\,\dif y}{4\pi\log\eps^{-1}}\log\frac{\eps^{{\gamma_\eps}-\delta_{\eps}}-\eps^{\gamma_{\eps}}}{1-\eps^{\gamma_{\eps}}}.
	\end{align*}
	From this and \eqref{deltagammabd} we obtain \eqref{Vtildegoodapprox}.
\end{proof}
In \lemref{Jtimelipschitz} we derived the regularity of $J_{\eps}$
in time (where time is taken on an exponential scale). Since $\log_{{\eps}}s$
varies slowly on most of the interval $[0,s_{m-1}'-s_{m}]$, it should
be plausible that we could approximate $J_{\eps}(2-\log_{\eps}s,a)^{2}$
by
{\[
	J_{\eps}(2-\log_{\eps}{s_{m-1}'},a)^{2}=J_{\eps}(2-(m-1)\delta_{\eps}-\gamma_\eps,a)^{2}\approx J_{\eps}(2-(m-1)\delta_{\eps},a)^{2}
\]}
in \eqref{Vtildedef}. Indeed we can, and that is how we will prove \propref{varapproxbyJ}.
\begin{proof}[Proof of \propref{varapproxbyJ}.]
	In light of \eqref{Vtildegoodapprox},  \lemref{Jtimelipschitz} and \eqref{Jbd}, it suffices to show that
	\[
		\adjustlimits\lim_{\substack{\eps\downarrow0\vphantom{T\in[0,T_{0}]}\\
		\vphantom{M_{1}(\eps,T)\le m\le M_{2}(\eps)}
		}
		}\sup_{\substack{T\in[0,T_{0}]\\
		M_{1}(\eps,T)+1\le m\le M_{2}(\eps)
		}
		}a^{-2}\left|J_{\eps}(2-(m-1)\delta_{\eps}{-\gamma_\eps},a)^{2}-\delta_{\eps}^{-1}V_{\eps,a,m}\right|=0.
	\]
	We will compare both $J_{\eps}(2-(m-1)\delta_{\eps}{-\gamma_\eps},a)^{2}$ and $\delta_{\eps}^{-1}V_{\eps,a,m}$
	to the intermediate quantity%
	\begin{align*}
		\tilde{V}_{\eps,a,m} & \coloneqq\frac{J_{\eps}(2-(m-1)\delta_{\eps}-\gamma_\eps,a)^{2}}{\log\eps^{-1}}\int_{0}^{s_{m-1}'-s_{m}}\frac{1}{s_{m-1}'-s_{m}'-s+\eps^{2}/2}\,\dif s                                                                 \\
		                     & =\frac{J_{\eps}(2-(m-1)\delta_{\eps}-\gamma_\eps,a)^{2}}{\log\eps^{-1}}\log\frac{\eps^{{\gamma_\eps}-\delta_{\eps}}-\eps^{\gamma_{\eps}}+\eps^{2-m\delta_{\eps}}/2}{1-\eps^{\gamma_{\eps}}+\eps^{2-m\delta_{\eps}}/2}.
	\end{align*}
	First, we have
	\begin{align*}
		 & \left|J_{\eps}(2-(m-1)\delta_{\eps}{-\gamma_\eps},a)^{2}-\delta_{\eps}^{-1}\tilde{V}_{\eps,a,m}\right|                                                                                                                                                      \\
		 & \qquad=J_{\eps}(2-(m-1)\delta_{\eps}{-\gamma_\eps},a)^{2}\left(1-\frac{1}{\delta_{\eps}\log\eps^{-1}}\log\frac{\eps^{{\gamma_\eps}-\delta_{\eps}}-\eps^{\gamma_{\eps}}+\eps^{2-m\delta_{\eps}}/2}{1-\eps^{\gamma_{\eps}}+\eps^{2-m\delta_{\eps}}/2}\right),
	\end{align*}
	and from this, \eqref{Jbd} of \lemref{Jspacelipschitz}, \eqref{deltagammabd},
	and \eqref{M2def} we have
	\begin{equation}
		\adjustlimits\lim_{\substack{\eps\downarrow0\vphantom{T\in[0,T_{0}]}\\
		\vphantom{M_{1}(\eps,T)\le m\le M_{2}(\eps)}
		}
		}\sup_{\substack{T\in[0,T_{0}]\\
		M_{1}(\eps,T)+1\le m\le M_{2}(\eps)
		}
		}a^{-2}\left|J_{\eps}(2-(m-1)\delta_{\eps}{-\gamma_\eps},a)^{2}-\delta_{\eps}^{-1}\tilde{V}_{\eps,a,m}\right|=0.\label{eq:JtoVtilde}
	\end{equation}

	On the other hand, we have by \eqref{Jbd} and \eqref{Jprelimbd}
	that
	\begin{align}
		 & \delta_{\eps}^{-1}\left|\tilde{V}_{\eps,a,m}-V_{\eps,a,m}\right|\nonumber                                                                                                                                                                                                                                                                       \\
		 & \quad\le\frac{1}{\delta_{\eps}\log\eps^{-1}}\int_{0}^{s_{m-1}'-s_{m}}\frac{\left|J_{\eps}(2-(m-1)\delta_{\eps}{-\gamma_\eps},a)^{2}-J_{\eps}(2-\log_{\eps}s,a)^{2}\right|}{s_{m-1}'-s_{m}'-s+\eps^{2}/2}\,\dif s\nonumber                                                                                                                       \\
		 & \quad\le\frac{a^2\beta^3K_0^3}{4\pi^{3/2}\delta_{\eps}(\log\eps^{-1})^{\frac{3}{2}}}\int_{0}^{s_{m-1}'-s_{m}}\frac{\sqrt{\log\frac{s_{m-1}'+\eps^{2}/2}{s+\eps^{2}/2}}+K_{0}}{s_{m-1}'-s_{m}'-s+\eps^{2}/2}\,\dif s\nonumber                                                                                                                    \\
		 & \quad\le\frac{a^2\beta^3K_0^3}{4\pi^{3/2}\delta_{\eps}(\log\eps^{-1})^{\frac{3}{2}}}\bigg[\int_{0}^{s_{m-1}'-s_{m}'}\frac{\log\frac{s_{m-1}'+\eps^{2}/2}{s+\eps^{2}/2}}{s_{m-1}'-s_{m}'-s+\eps^{2}/2}\,\dif s+(1+K_{0})(\log 2 + \delta_\eps \log\eps^{-1})\bigg].\label{eq:deltaVtildeV}
	\end{align}
	In the last inequality we used the elementary inequality $\sqrt{a}\le 1+a$ for all $a\ge 0$ as well as the explicit integral computation
	\begin{align*}
	 \int_{0}^{s_{m-1}'-s_{m}}\frac{\dif s}{s_{m-1}'-s_{m}'-s+\eps^{2}/2} &= \log\frac{s'_{m-1}-s_m'+\eps^2/2}{s_m-s_{m}'+\eps^{2}/2}\le \log\frac{s'_{m-1}-s_m'}{s_m-s_{m}'}=\log\frac{\eps^{(m-1)\delta_\eps+\gamma_\eps}-\eps^{m\delta_\eps+\gamma_\eps}}{\eps^{m\delta_\eps}-\eps^{m\delta_\eps+\gamma_\eps}}\\
	 &=\log\frac{\eps^{{\gamma_\eps}-\delta_{\eps}}-\eps^{\gamma_{\eps}}}{1-\eps^{\gamma_{\eps}}}\le \log 2 + \delta_\eps \log\eps^{-1},
	\end{align*} with the last inequality by \eqref{epsgammaepslt12}.

	For the first term in brackets on the right side of \eqref{deltaVtildeV}, we have by \lemref{logovertminuss}
	(applied with $t=s_{m-1}'-s_{m}'$ and $r=s_{m}'$) that
	\begin{align*}
		\int_{0}^{s_{m-1}'-s_{m}'}\frac{\log\frac{s_{m-1}'+\eps^{2}/2}{s+\eps^{2}/2}}{s_{m-1}'-s_{m}'-s+\eps^{2}/2}\,\dif s %
		 & \le\left(2+\log(1+2\eps^{-2}(s_{m-1}'-s_{m}'))\right)\left(1+\log\frac{s_{m-1}'+\eps^{2}/2}{s_{m-1}'-s_{m}'+\eps^{2}/2}\right) \\
		 & \le\left(2+\log(1+2\eps^{-2+(m-1)\delta_{\eps}+\gamma_\eps})\right)\left(1+\log\frac{1}{1-\eps^{\delta_{\eps}}}\right).
	\end{align*}
	The second bracketed factor goes to $1$ as $\eps\downarrow0$ (recalling
	\eqref{deltagammabd}) while the first factor is bounded by a constant
	times $\log\eps^{-1}$. %
	Using this in \eqref{deltaVtildeV}, we see that there is a constant
	$C<\infty$ so that
	\[
		a^{-2}\delta_{\eps}^{-1}\left|\tilde{V}_{\eps,a,m}-V_{\eps,a,m}\right|\le\frac{C}{\delta_{\eps}(\log\eps^{-1})^{\frac{1}{2}}}(1+\delta_{\eps}),
	\]
	and the right side goes to $0$ as $\eps\downarrow0$ (uniformly in
	$a$ and in $T\in[0,T_{0}]$) by \eqref{deltagammabd}. This and \eqref{JtoVtilde}
	imply \eqref{Japproximatesvariance}.
\end{proof}

\subsubsection{Higher moments}

For tightness purposes, we will also need an upper bound on a higher
moment of $Z_{\eps,b,m}$. Let $p>2$ be as in \propref{momentbound}.
\begin{prop}
	\label{prop:highermomentboundonZ}We have
	\begin{equation}
		\adjustlimits\limsup_{\substack{\eps\downarrow0\\
		\vphantom{m\in[M_{1}(\eps,T),M_{2}(\eps)]}
		}
		}\sup_{\substack{a>0\\
				M_{1}(\eps,T)+1\le m\le M_{2}(\eps)
			}
		}\frac{\mathbf{E}|Z_{\eps,a,m}|^{p}}{a^{p}\delta_{\eps}^{p/2}}<\infty.\label{eq:highermomentboundonZ}
	\end{equation}
\end{prop}
\begin{proof}
	Fix $\eps,m$ and define the martingale
	\[
		Z(r)=a+\frac{1}{\sqrt{\log\eps^{-1}}}\int_{0}^{r}\int G_{s_{m-1}'-s_{m}'-s}(X-y)\sigma(u_{\eps,a}(s,y))\,\dif W^{\eps}(s,y), \quad\quad r\geq0,
	\]
	so by \eqref{Zstochasticintegral} we have $Z_{\eps,a,m}=Z(s_{m-1}'-s_{m})$.
	The quadratic variation process is
	\begin{align*}
		\langle Z\rangle(r) & =\frac{1}{\log\eps^{-1}}\int_{0}^{r}\iint G_{\eps^{2}}(y_{1}-y_{2})\prod_{i=1}^{2}\left(G_{s_{m-1}'-s_{m}'-s}(X-y_{i})\sigma(u_{\eps,a}(s,y_{i}))\right)\,\dif y_{1}\,\dif y_{2}\,\dif s.
	\end{align*}
	By the Burkholder--Davis--Gundy inequality (see e.g.~\cite[Proposition 4.4]{Kho14}),
	we have a constant $C_{p}<\infty$ so that
	\begin{equation}
		\mathbf{E}|Z_{\eps,a,m}|^{p}\le C_{p}\mathbf{E}[\langle Z\rangle(s_{m-1}'-s_{m})]^{p/2}.\label{eq:applyBDG}
	\end{equation}
	By the inequality
	\[
		|\sigma(u_{\eps,a}(s,y_1))\sigma(u_{\eps,a}(s,y_2))| \leq \frac{\beta^2}{2}(u_{\eps,a}(s,y_1)^2+u_{\eps,a}(s,y_2)^2),
	\]
	we can estimate the quadratic variation as
	\begin{align*}
		\langle Z\rangle(r) & \le\frac{\beta^{2}}{\log\eps^{-1}}\int_{0}^{r}\int G_{s_{m-1}'-s_{m}'-s}(X-y)G_{s_{m-1}'-s_{m}'-s+\eps^{2}}(X-y)u_{\eps,a}(s,y)^{2}\,\dif y\,\dif s                                                                           \\
		                    & =\frac{\beta^{2}}{4\pi\log\eps^{-1}}\int_{0}^{r}\frac{1}{s_{m-1}'-s_{m}'-s+\eps^{2}/2}\int G_{\frac{(s_{m-1}'-s_{m}'-s)(s_{m-1}'-s_{m}'-s+\eps^{2})}{2(s_{m-1}'-s_{m}'-s)+\eps^{2}}}(X-y)u_{\eps,a}(s,y)^{2}\,\dif y\,\dif s,
	\end{align*}
	where we used \eqref{Gaussianproduct}  for the above ``=''.
	By Jensen's inequality we have
	\begin{align*}
		\langle Z\rangle(r)^{p/2} & \le\frac{\beta^{p}}{(4\pi\log\eps^{-1})^{p/2}}\left(\int_{0}^{r}\frac{1}{s_{m-1}'-s_{m}'-s+\eps^{2}/2}\,\dif s\right)^{p/2-1}\cdot                                                                  \\
		                          & \qquad\cdot\int_{0}^{r}\int\frac{1}{s_{m-1}'-s_{m}'-s+\eps^{2}/2}G_{\frac{(s_{m-1}'-s_{m}'-s)(s_{m-1}'-s_{m}'-s+\eps^{2})}{2(s_{m-1}'-s_{m}'-s)+\eps^{2}}}(X-y)u_{\eps,a}(s,y)^{p}\,\dif y\,\dif s  \\
		                          & \le\frac{\beta^{p}}{(4\pi\log\eps^{-1})^{p/2}}\left(\log\frac{s_{m-1}'-s_{m}'}{s_{m-1}'-s_{m}'-r}\right)^{p/2-1}\cdot                                                                               \\
		                          & \qquad\cdot\int_{0}^{r}\int\frac{1}{s_{m-1}'-s_{m}'-s+\eps^{2}/2}G_{\frac{(s_{m-1}'-s_{m}'-s)(s_{m-1}'-s_{m}'-s+\eps^{2})}{2(s_{m-1}'-s_{m}'-s)+\eps^{2}}}(X-y)u_{\eps,a}(s,y)^{p}\,\dif y\,\dif s.
	\end{align*}
	Taking expectations and using spatial homogeneity, we have
	\begin{align*}
		\mathbf{E}\langle Z\rangle(r)^{p/2} & \le\frac{\beta^{p}}{(4\pi\log\eps^{-1})^{p/2}}\left(\log\frac{s_{m-1}'-s_{m}'}{s_{m-1}'-s_{m}'-r}\right)^{p/2}\sup_{s\in[0,r]}\mathbf{E}u_{\eps,a}(s,y)^{p}.
	\end{align*}
	Substituting $r=s_{m-1}'-s_{m}$ and recalling \eqref{applyBDG} and
	\propref{momentbound}, we have
	\begin{align*}
		\mathbf{E}|Z_{\eps,a,m}|^{p} & \le\frac{\beta^{p}C_{p}K_{0}^{p}a^{p}}{(4\pi\log\eps^{-1})^{p/2}}\left(\log\frac{s_{m-1}'-s_{m}'}{s_{m}-s_{m}'}\right)^{p/2}=\frac{\beta^{p}C_{p}K_{0}^{p}a^{p}}{(4\pi\log\eps^{-1})^{p/2}}\left(\log\frac{\eps^{{\gamma_\eps}-\delta_{\eps}}-\eps^{\gamma_{\eps}}}{1-\eps^{\gamma_{\eps}}}\right)^{p/2}.
	\end{align*}
	From this and \eqref{deltagammabd} we see \eqref{highermomentboundonZ}.
\end{proof}

\section{Proof of Theorem~\ref{thm:convergence}\label{sec:convergence}}

In this section we complete the proof of \thmref{convergence}. The key remaining step is to show the convergence of the Markov chain defined in \secref{discretemartingale} to a continuous diffusion. The technology for doing this is well-known, through the martingale problem of Stroock and Varadhan. We will essentially use \cite[Theorem~11.2.3]{SV06} as a black box, but we state a special case in a form convenient for us in \appendixref{MCtodiffusion}.
\begin{proof}[Proof of \thmref{convergence}.]
	Suppose that $\eps_{k}\downarrow0$ and $J:[0,2]\times\mathbf{R}_{\ge0}\to\mathbf{R}_{\ge0}$
	are such that
	\begin{equation}
		J_{\eps_{k}}|_{[0,2]\times\mathbf{R}_{\ge0}}\to J\label{eq:Jepskconv}
	\end{equation}
	uniformly on compact subsets of $[0,2]\times\mathbf{R}_{\ge0}$.
	(These are the subsequential limits that are guaranteed to exist by \propref{Jcompact}.)
	By \lemref{Jspacelipschitz}, this implies in particular that $J$
	is uniformly Lipschitz in its second argument. For $Q\in[0,2]$ and
	$a\ge0$, we consider the stochastic differential equation
	\begin{align}
		\dif\tilde{\Xi}_{a,Q}^{J}(q) & =J(2-q,\tilde{\Xi}_{a,Q}^{J}(q))\dif B(q),\qquad q\in(2-Q,2];\label{eq:dXitilde} \\
		\tilde{\Xi}_{a,Q}^{J}(2-Q)   & =a,\label{eq:Xitildeic}
	\end{align}
	where $B$ is a standard Brownian motion. Since $J$ is Lipschitz in the spatial variable, the problem \eqref{dXitilde}--\eqref{Xitildeic}
	has a unique strong solution (given $Q$ and $J$). For the moment, the limit $J$ may depend on the sequence $\{\eps_k\}$, as may the solution to \eqref{Xitildeic}.

	Suppose that $\{Q_{\eps}\in[0,2]\}_{\eps>0}$ is such that \begin{equation}Q\coloneqq\lim_{\eps\downarrow0}Q_{\eps}\label{eq:QktoQ}\end{equation}
	exists. Define $T_{\eps_k}=\eps_k^{2-Q_{\eps_k}}$. We claim that
	\begin{equation}
		u_{\eps_{k},a}(T_{\eps_k},X)\xrightarrow[k\to\infty]{\mathrm{law}}\tilde{\Xi}_{a,Q}^{J}(2).\label{eq:uepskxi-1}
	\end{equation}
	By \propref{approxubymarkovchain}, it suffices to show that
	\begin{equation}
		Y_{\eps_{k},a,T_{\eps_k}}(M_{2}(\eps_{k}))\xrightarrow[k\to\infty]{\mathrm{law}}\tilde{\Xi}_{a,Q}^{J}(2).\label{eq:MCconverges}
	\end{equation}
	We now explain how \eqref{MCconverges} follows from \thmref{MCtodiffusion} with $A_1=2-Q$, $A_2=2$, and $L(q,b)=J(2-q,b)$. From \eqref{M1def}--\eqref{M2def} and \eqref{QktoQ} we have $\delta_{\eps_k}M_1(\eps_k,T_{\eps_k})\to 2-Q$ and $\delta_{\eps_k}M_2(\eps_k )\to 2$  as $k\to\infty$. The condition \eqref{VargoestoL} is verified by \propref{varapproxbyJ}, while the condition \eqref{momentboundforSV} is verified by \propref{highermomentboundonZ}. Thus \thmref{MCtodiffusion} applies and we obtain \eqref{MCconverges} and thus \eqref{uepskxi-1}.

	We note that the family of random variables $\{\sigma(u_{\eps_{k},b}(T_{\eps_k},X))^{2}\}_{k\ge1}$
	is uniformly integrable by the $p>2$ moment bound in \propref{momentbound},
	so from \eqref{uepskxi-1} we can derive
	\begin{equation}
		J(Q,a)=\lim_{k\to\infty}J_{\eps_{k}}(Q,a)=\lim_{k\to\infty}\frac{1}{2\sqrt{\pi}}\left(\mathbf{E}\sigma(u_{\eps_{k},a}(T_{\eps_k},X))^{2}\right)^{1/2}=\frac{1}{2\sqrt{\pi}}\left(\mathbf{E}\sigma(\tilde{\Xi}_{a,Q}^{J}(2))\right)^{1/2}.\label{eq:Jproblem-derived}
	\end{equation}
	The problem \eqref{dXitilde}, \eqref{Xitildeic}, \eqref{Jproblem-derived}
	agrees with the problem \eqref{dXi-intro}--\eqref{Jdef-intro} by
	the change of variables
	\begin{equation}
		\Xi_{a,Q}(q)=\tilde{\Xi}_{a,Q}^J(q+2-Q).\label{eq:trivchgvar}
	\end{equation}
	Note also that
	\[
		J(Q,0)=\lim_{k\to\infty}J_{\eps_{k}}(Q,0)=0,
	\]
	for all $Q\in[0,2]$, and that
	\[
		\Lip J(Q,\cdot)\le\limsup_{k\to\infty}\Lip J_{\eps_{k}}(Q,\cdot)\le\limsup_{k\to\infty}\left(\frac{4\pi}{\beta^{2}}-\frac{\log(1+2\eps_{k}^{-q})}{\log\eps_{k}^{-1}}\right)^{-1/2}=(4\pi/\beta^{2}-q)^{-1/2}
	\]
	by \lemref{Jspacelipschitz}. Therefore, $J$ satisfies both conditions
	of \thmref{wellposedness}, and thus $J$ is uniquely characterized
	by the properties we have established for it. By \propref{Jcompact},
	this means that in fact
	\[
		\lim_{\eps\downarrow0}J_{\eps}|_{[0,2]\times\mathbf{R}_{\ge0}}=J
	\]
	uniformly on compact subsets of $[0,2]\times\mathbf{R}_{\ge0}$, so the limiting procedure above does not depend on the specific choice of $\{\eps_k\}$.
	By the same argument as that leading to \eqref{uepskxi-1}, we
	have
	\begin{equation}
		u_{\eps,a}(\eps^{2-Q_{\eps}},X)\xrightarrow[\eps\to0]{\mathrm{law}}\tilde{\Xi}_{a,Q}^J(2)\overset{\mathrm{law}}{=}\Xi_{a,Q}(Q).\label{eq:convergencegeneralQ}
	\end{equation}
	In particular, for any $T$ independent of $\eps$, taking $Q_{\eps}=2-\log_{\eps}T\to Q=2$, we have
	\[
		u_{\eps,a}(T,X)\xrightarrow[\eps\to0]{\mathrm{law}}\Xi_{a,2}(2),
	\]
	as claimed.
\end{proof}

\begin{rem}\label{rem:ew}
	Now we are able to prove the convergence of the variance of the random variable
	\[
		\mathcal{U}_{\eps,a,T}(g):=\sqrt{\log\eps^{-1}}\int [u_{\eps,a}(T,x) -a]g(x) \dif x,
	\]
	where $T>0$ and a Schwartz function $g$ are fixed. By the mild formulation \eqref{umild}, recalling that $*$ denotes the spatial convolution, we have 
	\begin{align}
		\mathbf{E} & \mathcal{U}_{\eps,a,T}(g)^2\notag                                                                                                                                                 \\&=\mathbf{E}\left|\int_0^T\int G_{T-s}* g(y) \sigma(u_{\eps,a}(s,y))\dif W^{\eps}(s,y)\right|^2\notag\\
		           & =\int_0^T\iint  G_{T-s}* g(y_1)G_{T-s}* g(y_2) G_{\eps^2}(y_1-y_2) \mathbf{E}\sigma(u_{\eps,a}(s,y_1))\sigma(u_{\eps,a}(s,y_2))  \dif y_1 \dif y_2 \dif s\notag                   \\
		           & =\int_0^T\iint  G_{T-s}* g(y_1)G_{T-s}* g(y_1+\eps y_2) G_{1}(y_2) \mathbf{E}\sigma(u_{\eps,a}(s,y_1))\sigma(u_{\eps,a}(s,y_1+\eps y_2))  \dif y_1 \dif y_2 \dif s.\label{eq:101}
	\end{align}
	By \thmref{convergence}, \corref{regularity}, and \propref{momentbound}, we have, for any $s\in(0,T),y_1,y_2\in \mathbf{R}^2$,
	\[
		\mathbf{E}\sigma(u_{\eps,a}(s,y_1))\sigma(u_{\eps,a}(s,y_1+\eps y_2)) \to \mathbf{E}\sigma(\Xi_{a,2}(2))^2, \quad\quad \text{ as }\eps\to0.
	\]
	Then we pass to the limit in \eqref{101} to derive
	\begin{equation}
		\mathbf{E}\mathcal{U}_{\eps,a,T}(g)^2\to\mathbf{E}\sigma(\Xi_{a,2}(2))^2\int_0^T\int |G_{T-s}* g(y)|^2 \dif y \dif s,
	\end{equation}
	so the variance of $\mathcal{U}_{\eps,a,T}(g)$ converges as $\eps\to0$. By adopting the approach in \cite{GL20}, one should be able to further prove the convergence
	\begin{equation}
		\mathcal{U}_{\eps,a,T}(g)\xrightarrow[\eps\downarrow0]{\mathrm{law}} \mathcal{U}_{a,T}(g):=\int U_a(T,x)g(x) \dif x,
	\end{equation}
	with the random distribution $U_a$ solving the Edwards-Wilkinson equation
	\begin{equation}
		\dif U_a=\frac12\Delta U_a \dif t+\sqrt{\mathbf{E}\sigma(\Xi_{a,2}(2))^2}\dif W(t,x), \quad\quad U_a(0,x)=0.
	\end{equation}
	To avoid further lengthening the paper we do not pursue this direction here.
\end{rem}

\section{Multipoint statistics\label{sec:multipoint}}

Now we turn our attention to multipoint statistics and work towards
proving \thmref{multipoint}.

\subsection{Local-in-space dependence of the solution on the noise}

We can interpret \propref{exciseoneinterval} of \secref{shutoffnoise-oneinterval}
as a form of local-in-time dependence of the solution $u_{\eps,a}$
on the noise. In particular, we can turn off the noise in an area
temporally distant from where we evaluate the solution without affecting
the solution much. To discuss multipoint statistics, we will need
a similar property when we turn off the noise in a spatial region
that is distant from our point of interest.

For $B\subset\mathbf{R}^{2}$, let $v_{\eps,a}^{B}$ solve the problem
\begin{align}
	\dif v_{\eps,a}^{B}(t,x) & =\frac{1}{2}\Delta v_{\eps,a}^{B}(t,x)\dif t+(\log\eps^{-1})^{-\frac{1}{2}}\sigma(v_{\eps,a}^{B}(t,x))\dif W^{\eps,B}(t,x);\label{eq:uA-1} \\
	v_{\eps,a}^{B}(0,x)      & =a.\label{eq:uAic-1}
\end{align}
Here, $W^{\eps,B}=G_{\eps^{2}/2}*(W\mathbf{1}_{B})$. Note that $W^{\eps}=W^{\eps,B}+W^{\eps,B^{\mathrm{c}}}$,
and moreover that $W^{\eps,B}$ and $W^{\eps,B^{\mathrm{c}}}$ are independent.
Define
\begin{equation}
	R^{\eps,B}(x,x')=\int_{B}G_{\eps^{2}/2}(x-y)G_{\eps^{2}/2}(x'-y)\,\dif y\label{eq:Rbdef}
\end{equation}
so that, formally,
\begin{align*}
	\mathbf{E}\dif W^{\eps,B}(t,x)\dif W^{\eps,B}(t',x') & =\delta(t-t')R^{\eps,B}(x,x').
\end{align*}
Note that $R^{\eps,B}(x,x')\le G_{\eps^{2}}(x-x')$ for all $x,x'\in\mathbf{R}^{2}$.
We note that $v_{\eps,a}^{B}$ has nothing to do with the $v_{\eps,a}$
considered in \secref{smoothedfield}.

Our first goal will be an estimate on what happens if we turn off
the noise in a half-plane, which we do in \lemref{vBminusvH} below.
We then consider complements of rectangles by taking unions of half-planes
in \propref{approx-square}. First we record a simple moment bound.
\begin{lem}
	\label{lem:vBmomentbound}For any $T\in[0,T_{0}]$ and any $B\subset\mathbf{R}^{2}$,
	we have
	\begin{equation}
		\sup_{x\in\mathbf{R}^{2}}\left(\mathbf{E}v_{\eps,a}^{B}(t,x)^{2}\right)^{1/2}\le K_{0}a.\label{eq:vBmomentbound}
	\end{equation}
\end{lem}

\begin{proof}

	By the mild solution formula and Young's inequality, we have
	\begin{align*}
		\mathbf{E}v_{\eps,a}^{B}( & t,x)^{2}=a^{2}+\frac{1}{\log\eps^{-1}}\int_{0}^{t}\iint G_{t-s}(x-y_{1})G_{t-s}(x-y_{2})R^{\eps,B}(y_{1},y_{2})\cdot                                                                                    \\
		                          & \qquad\qquad\qquad\qquad\qquad\qquad\cdot\mathbf{E}[\sigma(v_{\eps,a}^{B}(t,y_{1}))\sigma(v_{\eps,a}^{B}(t,y_{2}))]\,\dif y_{1}\,\dif y_{2}\,\dif s                                                     \\
		                          & \le a^{2}+\frac{1}{2\log\eps^{-1}}\sum_{i=1}^{2}\int_{0}^{t}\iint G_{t-s}(x-y_{1})G_{t-s}(x-y_{2})R^{\eps,B}(y_{1},y_{2})\mathbf{E}\sigma(v_{\eps,a}^{B}(t,y_{i}))^{2}\,\dif y_{1}\,\dif y_{2}\,\dif s.
	\end{align*}
	This means that
	\begin{align*}
		\sup_{x\in\mathbf{R}^{2}}\mathbf{E}v_{\eps,a}^{B}(t,x)^{2} & \le a^{2}+\frac{1}{2\log\eps^{-1}}\int_{0}^{t}\iint G_{t-s}(x-y_{1})G_{t-s}(x-y_{2})G_{\eps^{2}}(y_{1}-y_{2})\cdot                                   \\
		                                                           & \qquad\qquad\qquad\qquad\qquad\qquad\cdot\sup_{x\in\mathbf{R}^{2}}\mathbf{E}\sigma(v_{\eps,a}^{B}(t,x))^{2}\,\dif y_{1}\,\dif y_{2}\,\dif s          \\
		                                                           & \le a^{2}+\frac{1}{2\pi\log\eps^{-1}}\int_{0}^{t}\frac{\sup_{x\in\mathbf{R}^{2}}\mathbf{E}\sigma(v_{\eps,a}^{B}(t,x))^{2}}{2(t-s)+\eps^{2}}\,\dif s,
	\end{align*}
	and \eqref{vBmomentbound} then follows from \lemref{secondmomentbound-general}
	(and \eqref{K0lb}).
\end{proof}
\begin{lem}
	\label{lem:vBminusvH}Let $B\subset\mathbf{R}^{2}$ and let $H$ be
	a half-plane in $\mathbf{R}^{2}$. Then we have, for all $x\not\in H$,
	that
	\begin{equation}
		\mathbf{E}(v_{\eps,a}^{B}-v_{\eps,a}^{B\setminus H})(t,x)^{2}\le 5a^2K_{0}^{2}\sum_{k=1}^{\infty}\left(\frac{\beta^{2}}{4\pi}\frac{\log(1+2\eps^{-2}t)}{\log\eps^{-1}}\right)^{k}(G_{\frac{1}{2}[t+k\eps^{2}]}*\mathbf{1}_{H})(x).\label{eq:vBvBminusH}
	\end{equation}
\end{lem}

\begin{proof}
	From \eqref{uA-1}--\eqref{uAic-1} we write the mild solution formula
	\[
		v_{\eps,a}^{B}(t,x)=a+\frac{1}{\sqrt{\log\eps^{-1}}}\int_{0}^{t}\int G_{t-s}(x-y)\sigma(v_{\eps,a}^{B}(s,y))\,\dif W^{\eps,B}(s,y).
	\]
	Subtracting the corresponding expression for $v_{\eps,a}^{B\setminus H}$,
	we obtain
	\begin{align*}
		(v_{\eps,a}^{B}&-v_{\eps,a}^{B\setminus H})(t,x) \\& =\frac{1}{\sqrt{\log\eps^{-1}}}\int_{0}^{t}\int G_{t-s}(x-y)[\sigma(v_{\eps,a}^{B}(s,y))-\sigma(v_{\eps,a}^{B\setminus H}(s,y))]\,\dif W^{\eps,B\setminus H}(s,y) \\
		                                                & \qquad+\frac{1}{\sqrt{\log\eps^{-1}}}\int_{0}^{t}\int G_{t-s}(x-y)\sigma(v_{\eps,a}^{B}(s,y))\,\dif W^{\eps,B\cap H}(s,y).
	\end{align*}
	Taking second moments in this expression, using the independence of
	$W^{\eps,B\setminus H}$ and $W^{\eps,B\cap H}$, we have
	\begin{align}
		\mathbf{E} & (v_{\eps,a}^{B}-v_{\eps,a}^{B\setminus H})(t,x)^{2}\nonumber                                                                                                                                                                                     \\
		           & \le\frac{\beta^{2}}{\log\eps^{-1}}\int_{0}^{t}\iint R^{\eps,B\setminus H}(y_{1},y_{2})\prod_{i=1}^{2}\left(G_{t-s}(x-y_i)\left(\mathbf{E}(v_{\eps,a}^{B}-v_{\eps,a}^{B\setminus H})(s,y_{i})^{2}\right)^{1/2}\right)\,\dif y_{1}\,\dif y_{2}\,\dif s\nonumber \\
		           & \qquad+\frac{\beta^{2}}{\log\eps^{-1}}\int_{0}^{t}\iint R^{\eps,B\cap H}(y_{1},y_{2})\prod_{i=1}^{2}\left(G_{t-s}(x-y_{i})\left(\mathbf{E}v_{\eps,a}^{B}(s,y_{i})^{2}\right)^{1/2}\right)\,\dif y_{1}\,\dif y_{2}\,\dif s\nonumber            \\
		           & \eqqcolon I_{1}+I_{2}.\label{eq:EvI1I2}
	\end{align}
	For the first term we can estimate
	\begin{align}
		I_{1} & \le\frac{\beta^{2}}{\log\eps^{-1}}\int_{0}^{t}\iint G_{t-s}(x-y_{1})G_{t-s}(x-y_{2})G_{\eps^{2}}(y_{1}-y_{2})\prod_{i=1}^{2}\left(\mathbf{E}(v_{\eps,a}^{B}-v_{\eps,a}^{B\setminus H})(s,y_{i})^{2}\right)^{1/2}\,\dif y_{1}\,\dif y_{2}\,\dif s\nonumber                                                      \\
		      & \le\frac{\beta^{2}}{\log\eps^{-1}}\int_{0}^{t}\int G_{t-s}(x-y_{1})G_{t-s+\eps^{2}}(x-y_{1})\mathbf{E}(v_{\eps,a}^{B}-v_{\eps,a}^{B\setminus H})(s,y_{1})^{2}\,\dif y_{1}\,\dif s\nonumber                                                                                                                    \\
		      & \le\frac{\beta^{2}}{4\pi\log\eps^{-1}}\int_{0}^{t}\int\frac{G_{\frac{(t-s)(t-s+\eps^{2})}{2(t-s)+\eps^{2}}}(x-y)}{t-s+\eps^{2}/2}\left(\mathbf{1}_{H^{\mathrm{c}}}(y)\mathbf{E}(v_{\eps,a}^{B}-v_{\eps,a}^{B\setminus H})(s,y)^{2}+4K_{0}^{2}a^{2}\mathbf{1}_{H}(y)\right)\,\dif y\,\dif s,\label{eq:I1start}
	\end{align} 
	where in the last inequality we used \eqref{Gaussianproduct} and \lemref{vBmomentbound}.
	For the second term of \eqref{EvI1I2} we can estimate
	\begin{align}
		I_{2} & \le\frac{\beta^{2}a^{2}K_{0}^{2}}{\log\eps^{-1}}\int_{0}^{t}\iint G_{t-s}(x-y_{1})G_{t-s}(x-y_{2})R^{\eps,B\cap H}(y_{1},y_{2})\,\dif y_{1}\,\dif y_{2}\,\dif s\nonumber \\
		      & \le\frac{\beta^{2}a^{2}K_{0}^{2}}{4\pi\log\eps^{-1}}\int_{0}^{t}\frac{(G_{\frac{1}{2}[t-s+\eps^{2}/2]}*\mathbf{1}_{H})(x)}{t-s+\eps^{2}/2}\,\dif s,\label{eq:I2start}
	\end{align}
	where in the second inequality we used \eqref{Gaussianproduct}.
	Using \eqref{I1start} and \eqref{I2start} in \eqref{EvI1I2}, we
	have
	\begin{equation}
		\begin{aligned}\mathbf{E}&(v_{\eps,a}^{B}-v_{\eps,a}^{B\setminus H})(t,x)^{2} \\& \le\frac{\beta^{2}}{4\pi\log\eps^{-1}}\int_{0}^{t}\int\frac{G_{\frac{(t-s)(t-s+\eps^{2})}{2(t-s)+\eps^{2}}}(x-y)}{t-s+\eps^{2}/2}\mathbf{1}_{H^{\mathrm{c}}}(y)\mathbf{E}(v_{\eps,a}^{B}-v_{\eps,a}^{B\setminus H})(s,y)^{2}\,\dif y\,\dif s \\
			                                                              & \qquad+\frac{\beta^{2}a^{2}K_{0}^{2}}{4\pi\log\eps^{-1}}\int_{0}^{t}\frac{\left(\left[4G_{\frac{(t-s)(t-s+\eps^{2})}{2(t-s)+\eps^{2}}}+G_{\frac{1}{2}[t-s+\eps^{2}/2]}\right]*\mathbf{1}_{H}\right)(x)}{t-s+\eps^{2}/2}\,\dif s.
		\end{aligned}
		\label{eq:EvBvBHaftersplitup}
	\end{equation}

	Now we note that for all $x\not\in H$, and all $r>0$, if we let $\omega\ge0$ be the distance between $x$ and $H$, then we have
	\begin{equation}\label{eq:GconvHmonotone}
	\begin{aligned}
	\frac{\dif}{\dif r}(G_r*\mathbf{1}_H)(x)&=\frac{\dif}{\dif r}\int_{\omega}^\infty 
	(2\pi r)^{-1/2}\e^{-\xi^2/(2r)}\,\dif \xi = \int_{\omega}^\infty 
	\frac{\partial^2}{\partial \xi^2}(2\pi r)^{-1/2}\e^{-\xi^2/(2r)}\,\dif \xi
	\\&= -
	\frac{\partial}{\partial \xi}(2\pi r)^{-1/2}\e^{-\xi^2/(2r)}\bigg|_{\xi=\omega}
	= (2\pi r)^{-1/2}\frac{\omega}{r}\e^{-\omega^2/(2r)}\ge 0.
	\end{aligned}
	\end{equation}
	This means that for all $s\in[0,t]$, we have 
	\begin{equation*}
		(G_{\frac{1}{2}[t-s+\eps^{2}/2]}*\mathbf{1}_{H})(x)%
		\le(G_{\frac{1}{2}(t+\eps^{2})}*\mathbf{1}_{H})(x)%
	\end{equation*}
	and similarly
	\[
		(G_{\frac{(t-s)(t-s+\eps^{2})}{2(t-s)+\eps^{2}}}*\mathbf{1}_{H})(x)\le(G_{\frac{t(t+\eps^{2})}{2t+\eps^{2}}}*\mathbf{1}_{H})(x)\le(G_{\frac{1}{2}(t+\eps^{2})}*\mathbf{1}_{H})(x).
	\]
	Using these estimates in \eqref{EvBvBHaftersplitup}, we see that
	if we put $f(t,x)=\mathbf{E}(v_{\eps,a}^{B}-v_{\eps,a}^{B\setminus H})(t,x)^{2}$,
	then for all $x\in H^{\mathrm{c}}$ we have
	\begin{equation}
		\begin{aligned}
			f(t,x) & \le\frac{\beta^{2}}{4\pi\log\eps^{-1}}\int_{0}^{t}\int\frac{G_{\frac{(t-s)(t-s+\eps^{2})}{2(t-s)+\eps^{2}}}(x-y)}{t-s+\eps^{2}/2}\mathbf{1}_{H^{\mathrm{c}}}(y)f(s,y)\,\dif y\,\dif s \\
			       & \qquad+\frac{5}{4\pi}\beta^{2}a^{2}K_{0}^{2}\left(G_{\frac{1}{2}[t+\eps^{2}]}*\mathbf{1}_{H}\right)(x)\frac{\log(1+2\eps^{-2}t)}{\log\eps^{-1}}
		\end{aligned}
		\label{eq:frecursive-final}
	\end{equation}
	Define
	\begin{equation}
		b^{(k)}(t)=\frac{t}{2}+k\frac{\eps^{2}}{2}.\label{eq:bkdef}
	\end{equation}
	We note that
	\begin{equation}
	\begin{aligned}
		\sup_{s\in[0,t]}\left[b^{(k)}(s)+\frac{(t-s)(t-s+\eps^{2}/2)}{2(t-s)+\eps^{2}/2}\right]  &=\sup_{s\in[0,t]}\left[\frac{s}{2}+k\frac{\eps^{2}}{2}+\frac{(t-s)(t-s+\eps^{2}/2)}{2(t-s)+\eps^{2}/2}\right]\\&\le b^{(k+1)}(t)\end{aligned}\label{eq:bkrecursive}
	\end{equation}
	for all $s\in[0,t]$ and all $k\ge1$. Define
	\[
		B_{2}^{(k)}=5a^2K_{0}^{2}\left(\frac{\beta^{2}}{4\pi}\frac{\log(1+2\eps^{-2}t)}{\log\eps^{-1}}\right)^{k}.
	\]
	Suppose that
	\begin{equation}
		f(t,x)\le B_{1}^{(n)}+\sum_{k=1}^{n}B_{2}^{(k)}(G_{b^{(k)}(t)}*\mathbf{1}_{H})(x)\label{eq:fbddbybyBs}
	\end{equation}
	for all $x\in H^{\mathrm{c}}$. This is automatically true for $n=0$ with
	$B_{1}^{(0)}=\|f\|_{L^{\infty}([0,t]\times\mathbf{R}^{2})}$. Then
	we have from \eqref{frecursive-final} that, for all $x\in H^{\mathrm{c}}$,
	\begin{align*}
		f(t,x) & \le\frac{\beta^{2}}{4\pi\log\eps^{-1}}\int_{0}^{t}\int\frac{G_{\frac{(t-s)(t-s+\eps^{2})}{2(t-s)+\eps^{2}}}(x-y)}{t-s+\eps^{2}/2}\mathbf{1}_{H^{\mathrm{c}}}(y)\left[B_{1}^{(n)}+\sum_{k=1}^{n}B_{2}^{(k)}(G_{b^{(k)}(s)}*\mathbf{1}_{H})(y)\right]\,\dif y\,\dif s      \\
		       & \qquad+\frac{5}{4\pi}\beta^{2}a^{2}K_{0}^{2}\left(G_{\frac{1}{2}[t+\eps^{2}]}*\mathbf{1}_{H}\right)(x)\frac{\log(1+2\eps^{-2}t)}{\log\eps^{-1}}                                                                                                                          \\
		       & \le\frac{\beta^{2}B_{1}^{(n)}}{8\pi}\frac{\log(1+2\eps^{-2}t)}{\log\eps^{-1}}+\frac{\beta^{2}}{4\pi\log\eps^{-1}}\sum_{k=1}^{n}B_{2}^{(k)}\int_{0}^{t}\int\frac{1}{t-s+\eps^{2}/2}(G_{b^{(k)}(s)+\frac{(t-s)(t-s+\eps^{2})}{2(t-s)+\eps^{2}}}*\mathbf{1}_{H})(x)\,\dif s \\
		       & \qquad+\frac{5}{4\pi}\beta^{2}a^{2}K_{0}^{2}\left(G_{b^{(1)}(t)}*\mathbf{1}_{H}\right)(x)\frac{\log(1+2\eps^{-2}t)}{\log\eps^{-1}}                                                                                                                                       \\
		       & \le\frac{\beta^{2}B_{1}^{(n)}}{8\pi}\frac{\log(1+2\eps^{-2}t)}{\log\eps^{-1}}+\frac{\beta^{2}\log(1+2\eps^{-2}t)}{4\pi\log\eps^{-1}}\sum_{k=1}^{n}B_{2}^{(k)}(G_{b^{(k+1)}(t)}*\mathbf{1}_{H})(x)+B_{2}^{(1)}\left(G_{b^{(1)}(t)}*\mathbf{1}_{H}\right)(x)               \\
		       & =\frac{\beta^{2}B_{1}^{(n)}}{8\pi}\frac{\log(1+2\eps^{-2}t)}{\log\eps^{-1}}+\sum_{k=1}^{n+1}B_{2}^{(k)}(G_{b^{(k)}(t)}*\mathbf{1}_{H})(x).
	\end{align*}
	In the third inequality we used \eqref{bkrecursive} and \eqref{GconvHmonotone}.
	By induction, this means that \eqref{fbddbybyBs} holds for all $n\ge0$,
	with $B_{1}^{(n)}=\|f\|_{L^{\infty}([0,t]\times\mathbf{R}^{2})}\left(\frac{\beta^{2}}{8\pi}\cdot\frac{\log(1+4\eps^{-2}t)}{\log\eps^{-1}}\right)^{n}\to0$
	as $n\to\infty$. Therefore, we in fact have
	\[
		f(t,x)\le 5a^2K_{0}^{2}\sum_{k=1}^{\infty}\left(\frac{\beta^{2}}{4\pi}\frac{\log(1+2\eps^{-2}t)}{\log\eps^{-1}}\right)^{k}(G_{b^{(k)}(t)}*\mathbf{1}_{H})(x),
	\]
	which (recalling \eqref{bkdef}) is \eqref{vBvBminusH}.
\end{proof}
Now we apply \lemref{vBminusvH} four times to bound the effect of
turning off the noise outside of a square.
\begin{prop}
	\label{prop:approx-square}Suppose that
	\begin{equation}
		\lim_{\eps\downarrow0}\frac{\xi_{\eps}}{\eta_{\eps}}=\lim_{\eps\downarrow0}\frac{t_{\eps}^{1/2}}{\eta_{\eps}}=0.\label{eq:qlimits}
	\end{equation}
	and
	\begin{equation}
		\limsup\limits _{\eps\downarrow0}t_{\eps}<\infty.\label{eq:q3limit}
	\end{equation}
	Let $\square_{\eps}=[-\eta_{\eps},\eta_{\eps}]^{2}$. Then we have
	for all $x\in[-\xi_{\eps},\xi_{\eps}]^{2}$ that
	\begin{equation}
		\lim_{\eps\downarrow0}\mathbf{E}(u_{\eps,a}-v_{\eps,a}^{\square_{\eps}})(t_{\eps},x)^{2}=0.\label{eq:outsideboxgoesto0}
	\end{equation}
\end{prop}

\begin{proof}
	Using \lemref{vBminusvH} four times, we have

	\begin{equation}
		\mathbf{E}(u_{\eps,a}-v_{\eps,a}^{\square_{\eps}})(t_{\eps},x)^{2}\le 5a^2K_{0}^{2}\sum_{i=1}^{4}\sum_{k=1}^{\infty}c_{\eps}^{k}(G_{\frac{1}{2}[t_\eps + k\eps^2]}*\mathbf{1}_{H_{i}})(x),\label{eq:vBvBminusH-1}
	\end{equation}
	where $H_{1},\ldots,H_{4}$ are four half-planes so that $\square_{\eps}=\bigcap_{i=1}^{4}H_{i}$.
	Here we have also defined
	\[
		c_{\eps}=\frac{\beta^{2}}{4\pi}\frac{\log(1+\eps^{-2}t_{\eps})}{\log\eps^{-1}}.
	\]
	We note that \eqref{q3limit} and the subcriticality assumption $\beta<\sqrt{2\pi}$  that
	\begin{equation}
		\limsup_{\eps\downarrow 0} c_\eps <1.\label{eq:limsupceps}
	\end{equation}
	Now if $\limsup\limits _{\eps\downarrow0}\eps^{-2}t_{\eps}<\infty$,
	then $c_{\eps}\to0$ as $\eps\downarrow0$, so using the trivial bound
	$(G_{\frac{1}{2}[t_\eps + k\eps^2]}*\mathbf{1}_{H_{i}})(x)\le1$
	in \eqref{vBvBminusH-1} we get \eqref{outsideboxgoesto0}. Therefore,
	we can assume that
	\begin{equation}
		\limsup\limits _{\eps\downarrow0}\eps^{-2}t_{\eps}=\infty.\label{eq:etothethingtoinfinity}
	\end{equation}
	We break the inner sum in \eqref{vBvBminusH-1} into two pieces. First we estimate
	\[
		\sum_{k=\lfloor\eps^{-2}t_{\eps}\rfloor}^{\infty}c_{\eps}^{k}(G_{\frac{1}{2}[t_\eps + k\eps^2]}*\mathbf{1}_{H_{i}})(x)\le\sum_{k=\lfloor\eps^{-2}t_{\eps}\rfloor}^{\infty}c_{\eps}^{k}=\frac{c_{\eps}^{\lfloor\eps^{-2}t_{\eps}\rfloor}}{1-c_{\eps}}\to0
	\]
	as $\eps\downarrow0$ by \eqref{limsupceps} and \eqref{etothethingtoinfinity}.
	Then we estimate
	\[
		\sum_{k=1}^{\lfloor\eps^{-2}t_{\eps}\rfloor}c_{\eps}^{k}(G_{\frac{1}{2}[t_\eps + k\eps^2]}*\mathbf{1}_{H_{i}})(x)\le(G_{t_{\eps}}*\mathbf{1}_{H_{i}})(x)\sum_{k=1}^{\infty}c_{\eps}^{k}=\frac{c_{\eps}}{1-c_{\eps}}(G_{t_{\eps}}*\mathbf{1}_{H_{i}})(x),
	\]
	using the fact that $t_\eps/2+k\eps^2/2\le t_\eps$ whenever $k\le \eps^{-2}t_\eps$. Now we have, for $x\in [-\xi_\eps,\xi_\eps]^2$, that
	\begin{align*}
		(G_{t_{\eps}}*\mathbf{1}_{H_{i}})(x)\le\frac{1}{\sqrt{2\pi t_{\eps}}}\int_{\eta_{\eps}-\xi_{\eps}}^{\infty}\exp\left\{ -\frac{\alpha^{2}}{2t_{\eps}}\right\} \,\dif\alpha & \le\frac{1}{\sqrt{2\pi t_{\eps}}}\int_{\eta_{\eps}-\xi_{\eps}}^{\infty}\exp\left\{ -\frac{\alpha(\eta_{\eps}-\xi_{\eps})}{2t_{\eps}}\right\} \,\dif\alpha \\
		                                                                                                                                                                          & =\frac{\sqrt{2t_{\eps}/\pi}}{\eta_{\eps}-\xi_{\eps}}\exp\left\{ -\frac{(\eta_{\eps}-\xi_{\eps})^{2}}{2t_{\eps}}\right\} \to0
	\end{align*}
	as $\eps\downarrow0$ by \eqref{qlimits}. Combining the last three
	displays and  \eqref{vBvBminusH-1} gives us \eqref{outsideboxgoesto0}.
\end{proof}

\subsection{Proof of Theorem~\ref{thm:multipoint}}

We now have the tools we need to prove \thmref{multipoint}. Throughout
this section, our setup is as in the statement of \thmref{multipoint}.
We note in particular that \eqref{dijdef} implies (with $d$
as in \eqref{ddef}) that
\[
	d((\tau_{\eps,i},x_{\eps,i}),(\tau_{\eps,j},x_{\eps,j}))=\eps^{1-d_{ij}+o(1)}.
\]
and \eqref{Qdef-1} implies that
\begin{equation}
	\tau_{\eps,i}=\eps^{2-Q+o(1)}\label{eq:Qinterpret}
\end{equation}
as $\eps\downarrow 0$.
Let $\kappa_{\eps}$ be such that $\kappa_{\eps}\to0$ and
\begin{equation}
	10\eps^{1-d_{ij}+\kappa_{\eps}}\le d((\tau_{\eps,i},x_{\eps,i}),(\tau_{\eps,j},x_{\eps,j}))\le\frac{1}{2}\eps^{1-d_{ij}-\kappa_{\eps}}.\label{eq:kappaerrortermdij}
\end{equation}
and
\begin{equation}
	2\eps^{2-Q+2\kappa_{\eps}}\le\tau_{\eps,i}\le\eps^{2-Q-2\kappa_{\eps}}.\label{eq:kappaerrortermQ}
\end{equation}

Our first step will apply \propref{approx-square} to show that the values of the
solution $u_{\eps,a}$ at distant space-time points are asymptotically independent. %
\begin{prop}
	\label{prop:bustupindependent}Let $P_{1},\ldots,P_{R}$ be a partition
	of $[N]$ so that
	\begin{equation}
		d_{ij}\ge Q/2\iff i\in P_{m},j\in P_{n},n\ne m.\label{eq:dijcond}
	\end{equation}
	Then there is an $\eps_{1}\in(0,\eps_0]$ so that if $\eps\in[0,\eps_{1})$ then
	there are independent processes $u_{\eps,a}^{(1)},\ldots,u_{\eps,a}^{(R)}$
	so that $u_{\eps,a}^{(k)}\overset{\mathrm{law}}{=}u_{\eps,a}$ ($k=1,\ldots,R$),
	and for each $j\in P_{k}$ ($k=1,\ldots,R$), we have
	\begin{equation}
		\lim_{\eps\downarrow0}\mathbf{E}(u_{\eps,a}^{(k)}(\tau_{\eps,j},x_{\eps,j})-u_{\eps,a}(\tau_{\eps,j},x_{\eps,j}))^{2}=0.\label{eq:approx-independent}
	\end{equation}
\end{prop}

\begin{proof}
	For each $k=1,\ldots,R$, let $i_{k}$ be an arbitrary element of
	$P_{k}$. Define
	\begin{equation}
		D_{k}\coloneqq\max_{i,j\in P_{k}}d_{ij}<Q/2,\label{eq:Dkdef}
	\end{equation}
	with the inequality by \eqref{dijcond}.
	Define the sets $S_{\eps,k}\subset\mathbf{R}\times\mathbf{R}^{2}$
	by
	\[
		S_{\eps,k}=\left(\tau_{\eps,i_{k}}+[-\eps^{2-Q+2\kappa_{\eps}+2\zeta_{\eps}},\eps^{2-Q+2\kappa_{\eps}+2\zeta_{\eps}}]\right)\times\left(x_{\eps,i_{k}}+[-\eps^{1-Q/2+\kappa_{\eps}},\eps^{1-Q/2+\kappa_{\eps}}]^2\right).
	\]
	Here $\kappa_{\eps}$ is as in \eqref{kappaerrortermdij}--\eqref{kappaerrortermQ}
	and $\zeta_{\eps}$ is as in \eqref{zetabound}.

	If $k_{1}\ne k_{2}$, then we have by \eqref{dijcond} that $d_{i_{k_{1}}i_{k_{2}}}\ge Q/2$,
	so by \eqref{ddef} and \eqref{kappaerrortermdij} we have
	\[
		\max\{|\tau_{\eps,i_{k_{1}}}-\tau_{\eps,i_{k_{2}}}|^{1/2},|x_{\eps,i_{k_{1}}}-x_{\eps,i_{k_{2}}}|\}\ge10\eps^{1-Q/2+\kappa_{\eps}}.
	\]
	This means that $\{S_{\eps,1},\ldots,S_{\eps,R}\}$ forms a pairwise-disjoint
	family of sets.

	Let $A_{k}=[0,\tau_{\eps,i_{k_{1}}}-\eps^{2-Q+\kappa_{\eps}+\zeta_{\eps}}]$.
	Define $u_{\eps,a}^{A_{k}}$ as in \eqref{uA}--\eqref{uAic}. By
	\propref{exciseoneinterval}, we have, for all $j\in P_{k}$, that
	\begin{equation}
	\begin{aligned}
		&\left(\mathbf{E}(u_{\eps,a}-u_{\eps,a}^{A_{k}})(\tau_{\eps,j},x_{\eps,j})^{2}\right)^{1/2}\\&\qquad\le\frac{\beta aK_{0}^2}{2\sqrt{\pi\log\eps^{-1}}}\left(K_{0}+\sqrt{\log\frac{\tau_{\eps,j}+\eps^{2}}{\tau_{\eps,j}-\tau_{\eps,i_{k}}+\eps^{2-Q+2\kappa_{\eps}+2\zeta_{\eps}}+\eps^{2}}}\right).\label{eq:cutoffnoisefirst}
		\end{aligned}
	\end{equation}
	We note  (still assuming $j\in P_{k}$) that
	\begin{equation}
		|\tau_{\eps,j}-\tau_{\eps,i_{k}}|\le\frac{1}{4}\eps^{2-2D_{k}-2\kappa_{\eps}}\ll \eps^{2-Q+2\kappa_{\eps}+2\zeta_{\eps}} \qquad\text{and}\qquad\tau_{\eps,j}\le\eps^{2-Q-2\kappa_{\eps}}\label{eq:tauepsjclose}
	\end{equation}
	by \eqref{kappaerrortermdij}, \eqref{kappaerrortermQ}, and \eqref{Dkdef}. %
    Thus from
	\eqref{cutoffnoisefirst} we obtain a constant $C$ so that
	\begin{align}
		\left(\mathbf{E}(u_{\eps,a}-u_{\eps,a}^{A_{k}})(\tau_{\eps,j},x_{\eps,j})^{2}\right)^{1/2} & \le\frac{C\beta aK_{0}^2}{2\sqrt{\pi\log\eps^{-1}}}\left(K_{0}+\sqrt{\log\eps^{-4\kappa_{\eps}-2\zeta_{\eps}}}\right)\to0\label{eq:uuAk}
	\end{align}
	as $\eps\downarrow0$ since $\kappa_{\eps},\zeta_{\eps}\to0$.

	Define $\pi_1:\R\times\R^2\to\R$ be given by $\pi_1(t,x)=t$ and $\pi_2:\R\times\R^2\to\R^2$ be given by $\pi_2(t,x)=x$. %
	Let $\tilde{u}_{\eps,a}^{(k)}$ solve the problem
	\begin{align}
		\dif\tilde{u}_{\eps,a}^{(k)}(t,x) & =\frac{1}{2}\Delta\tilde{u}_{\eps,a}^{(k)}(t,x)\dif t+(\log\eps^{-1})^{-\frac{1}{2}}\mathbf{1}_{\pi_1(S_{\eps,k})}(t)\sigma(\tilde{u}_{\eps,a}^{(k)}(t,x))\dif W^{\eps,\pi_2(S_{\eps,k})}(t,x);\label{eq:uA-1-1} \\
		\tilde{u}_{\eps,a}^{(k)}(0,x)     & =a.\label{eq:uAic-1-1}
	\end{align}
	This turns off some temporal part of the noise as in \eqref{uA}--\eqref{uAic}
	but also a spatial part of the noise as in \eqref{uA-1}--\eqref{uAic-1}.
	Since $\{S_{\eps,1},\ldots,S_{\eps,R}\}$ is pairwise-disjoint, the
	processes $u_{\eps,a}^{(1)},\ldots,u_{\eps,a}^{(R)}$ are independent.
	We now want to apply (a translated version of) \propref{approx-square}
	with
	\[
		\xi_{\eps}=\eps^{1-D_{k}-\kappa_{\eps}},\qquad\eta_{\eps}=\eps^{1-Q/2+\kappa_{\eps}},\qquad t_{\eps}=\tau_{\eps,j}-\tau_{\eps,i_{k_{1}}}+\eps^{2-Q+2\kappa_{\eps}+2\zeta_{\eps}}.
	\]
	Note that
	\[
		\lim_{\eps\downarrow0}\frac{\xi_{\eps}}{\eta_{\eps}}=\lim_{\eps\downarrow0}\frac{\eps^{1-D_{k}-\kappa_{\eps}}}{\eps^{1-Q/2+\kappa_{\eps}}}=\lim_{\eps\downarrow0}\eps^{Q/2-D_{k}-2\kappa_{\eps}}=0
	\]
	since $D_k<Q/2$ and $\kappa_\eps\to 0$, 
	and also that (using these facts along with \eqref{zetabound} and \eqref{tauepsjclose}) that
	\[
		\lim_{\eps\downarrow0}\frac{t_{\eps}^{1/2}}{\eta_{\eps}}\le\lim_{\eps\downarrow0}\frac{(\tau_{\eps,j}-\tau_{\eps,i_{k_{1}}})^{1/2}}{\eps^{1-Q/2+\kappa_{\eps}}}+\lim_{\eps\downarrow0}\frac{\eps^{1-Q/2+\kappa_{\eps}+\zeta_{\eps}}}{\eps^{1-Q/2+\kappa_{\eps}}}\le\lim_{\eps\downarrow0}\frac{\eps^{1-D_{k}-\kappa_{\eps}}}{\eps^{1-Q/2+\kappa_{\eps}}}+\lim_{\eps\downarrow0}\eps^{\zeta_{\eps}}=0.
	\]
	Therefore, \eqref{qlimits}
	is verified, so \propref{approx-square} applies, and we have (combining
	the result with \eqref{uuAk}) that
	\begin{equation}
		\lim_{\eps\downarrow0}\mathbf{E}(u_{\eps,a}-\tilde{u}_{\eps,a}^{(k)})(\tau_{\eps,i_{j}},x_{\eps,i_{j}})^{2}=0\label{eq:uutilde}
	\end{equation}
	for all $j\in P_{k}$. Now let $u_{\eps,a}^{(k)}$ solve the problem
	\begin{align}
		\label{eq:duk} \dif u_{\eps,a}^{(k)}(t,x) & =\frac{1}{2}\Delta u_{\eps,a}^{(k)}(t,x)\dif t\\&\qquad+(\log\eps^{-1})^{-\frac{1}{2}}\mathbf{1}_{\pi_{1}(S_{\eps,k})}(t)\sigma(u_{\eps,a}^{(k)}(t,x))\dif[W^{\eps,\pi_{2}(S_{\eps,k})}(t,x)+\tilde{W}^{\eps,\pi_{2}(S_{\eps,k})^{\mathrm{c}}}]\notag \\&\notag\qquad+(\log\eps^{-1})^{-\frac{1}{2}}\mathbf{1}_{\mathbf{R}\setminus\pi_{1}(S_{\eps,k})}(t)\sigma(u_{\eps,a}^{(k)}(t,x))\dif\tilde{W}^{\eps}(t,x)\\%
		u_{\eps,a}(0,x)            & =a,\label{eq:uk0}
	\end{align}
	where $\tilde{W}$ is an independent copy of
	$W$ (different and independent across different choices of $k$). Note that $u_{\eps,a}^{(1)},\ldots,u_{\eps,a}^{(R)}$
	are independent since the family $\{S_{\eps,1},\ldots,S_{\eps,R}\}$
	is disjoint. %
	The pairs
	$(u_{\eps,a},\tilde{u}_{\eps,a}^{(k)})$ and $(u_{\eps,a}^{(k)},\tilde{u}_{\eps,a}^{(k)})$
	have the same joint laws because to go from $u_{\eps,a}$ to $u_{\eps,a}^{(k)}$ we simply replaced a part of the noise (on $S_{\eps,k}^{\mathrm{c}}$) that is independent of $\tilde{u}_{\eps,a}^{(k)}$ (for which the noise on $S_{\eps,k}^{\mathrm{c}}$ is turned off). Therefore, \eqref{uutilde} also means
	that
	\begin{equation}
		\lim_{\eps\downarrow0}\mathbf{E}(u_{\eps,a}^{(k)}-\tilde{u}_{\eps,a}^{(k)})(\tau_{\eps,i_{j}},x_{\eps,i_{j}})^{2}=0,\label{eq:uutilde-1}
	\end{equation}
	and combining this with \eqref{uutilde} yields \eqref{approx-independent}.
\end{proof}
Now we can prove \thmref{multipoint}.

\begin{proof}[Proof of \thmref{multipoint}.]
	We use induction on $N$. The base case, \eqref{multipointconvergence}
	with $N=1$, is simply an application of \eqref{convergencegeneralQ}.
	Now suppose that $N\ge2$ and that \eqref{multipointconvergence}
	holds for all strictly smaller $N$. Let
	\begin{equation}
		q_{0}=2-2\max_{i,j\in[N]}d_{i,j}.\label{eq:q0def}
	\end{equation}
	Then we have
	\begin{equation}
		q<Q-2+q_{0}\implies i_{(Q-q)/2}([N])=\{1\}\label{eq:iqis1}
	\end{equation}
	by the definition \eqref{iqdef}. Define
	\begin{equation}
		m_{\eps}(q_{0})=\max\{M_{1}(\eps,\tau_{\eps,1}),\lfloor(q_{0}-2\kappa_{\eps}-2\gamma_{\eps})\delta_{\eps}^{-1}\rfloor\},\label{eq:mepsq0}
	\end{equation}
	recalling the definition \eqref{M1def}, and also recall the definition
	\eqref{dwm}--\eqref{wmic} of $w_{\eps,a,T,X}^{(m)}$. In the case
	$m_{\eps}(q_{0})=M_{0}(\eps,\tau_{\eps,1})$, we have
	\begin{equation}
		u_{\eps,a}(\tau_{\eps,j},x_{\eps,j})=w_{\eps,a,\tau_{\eps,1},x_{\eps,1}}^{(m_{\eps}(q_{0}))}\label{eq:utildeisw}
	\end{equation}
	by the definition \eqref{M1ic}.
Otherwise, we note using \eqref{kappaerrortermdij} that %
	\begin{equation}
		\tau_{\eps,j}  \ge\tau_{\eps,1}-\frac{1}{2}\eps^{2-2d_{1j}-2\kappa_{\eps}}\ge\tau_{\eps,1}-\frac{1}{2}\eps^{q_{0}-2\kappa_{\eps}}\ge\tau_{\eps,1}-\frac{1}{2}\eps^{m_{\eps}(q_{0})\delta_{\eps}+\gamma_{\eps}}.\label{eq:taujbigger}
	\end{equation}
	Thus we can apply \propref{approxubyw} with $C_{\eps}=\eps^{-\gamma_{\eps}/2}$ (recalling \eqref{deltagammabd}) and $c=1/2$,
	and by \eqref{taujbigger} take $T=\tau_{\eps,1}$, $k=m_{\eps}(q_{0})$, and $t=\tau_{\eps,j}$ in the
	supremum in \eqref{uapproxbyw}, to obtain
	\begin{equation}
		\lim_{\eps\downarrow0}\frac{\left(\mathbf{E}(u_{\eps,a}-w_{\eps,a,\tau_{\eps,1},x_{\eps,1}}^{(m_{\eps}(q_{0}))})(\tau_{\eps,j},x_{\eps,j})^{2}\right)^{1/2}}{a(1+\eps^{-m_{\eps}(q_{0})\delta_{\eps}/2-\gamma_{\eps}/2}|x_{\eps,j}-x_{\eps,1}|)}=0.\label{eq:applymiddleofthepath}
	\end{equation} Note that \eqref{utildeisw} implies \eqref{applymiddleofthepath} as well, so in fact \eqref{applymiddleofthepath} holds unconditionally.
	On the other hand, we also have, using \eqref{mepsq0}, \eqref{kappaerrortermdij}, \eqref{q0def},
	and \eqref{deltagammabd}, that
	\begin{align*}
		\lim_{\eps\downarrow0}\eps^{-m_{\eps}(q_{0})\delta_{\eps}/2-\gamma_{\eps}/2}|x_{\eps,j}-x_{\eps,1}| & \le\frac{1}{2}\lim_{\eps\downarrow0}\eps^{\gamma_{\eps}/2-q_{0}/2+1-d_{1j}}\le\frac{1}{2}\lim_{\eps\downarrow0}\eps^{\gamma_{\eps}/2}=0.
	\end{align*}
	Combined with \eqref{applymiddleofthepath}, this means that
	\begin{equation}
		\lim_{\eps\downarrow0}a^{-1}\left(\mathbf{E}(u_{\eps,a}-w_{\eps,a,\tau_{\eps,1},x_{\eps,1}}^{(m_{\eps}(q_{0}))})(\tau_{\eps,j},x_{\eps,j})^{2}\right)^{1/2}=0.\label{eq:utow}
	\end{equation}

	Now define
	\begin{equation}
		\ell_{\eps}=\tau_{\eps,1}-\eps^{m_{\eps}(q_{0})\delta_{\eps}+\gamma_{\eps}}\label{eq:seps}
	\end{equation}
	and
	\[
		w(t,x)=w_{\eps,a,\tau_{\eps,1},x_{\eps,1}}^{(m_{\eps}(q_{0}))}(t+\ell_{\eps},x_{\eps,1}).
	\]
	Note that if $T=\tau_{\eps,1}$ then $t_{m_{\eps}(q_{0})}'=\tau_{\eps,1}-\ell_{\eps}$, so $w(0,\cdot)$ is constant in space and $w(0,x)\overset{\mathrm{law}}{=}Y_{\eps,a,\tau_{\eps,1}}(m_{\eps}(q_{0}))$. Thus, by applying \thmref{MCtodiffusion} as in the proof of \eqref{MCconverges}
	(recalling \eqref{trivchgvar} and \eqref{iqis1}), we see
	that
	\begin{equation}
		w(0,x) %
		\xrightarrow[\eps\downarrow0]{\mathrm{law}}\Gamma_{a,Q,1}(Q-(2-q_0)).\label{eq:wtogamma1}
	\end{equation}
	Moreover, $w$ is equal in law to $u_{\eps,b}$, where $b=w(0,x)$
	is taken to be independent of the noise driving $u_{\eps,b}$.

	Recall the definition \eqref{iqdef} and let
	\[
		P_{k}=i_{1-q_{0}/2}^{-1}(k)=\{j\in[N]\ :\ i_{1-q_{0}/2}(j)=k\}.
	\]
	Note that $P_{1},\ldots,P_{N}$ form a partition of $[N]$, and by
	\eqref{q0def} this partition is nontrivial. If $i_{1-q_{0}/2}(j_{1})=i_{1-q_{0}/2}(j_{2})$
	then $d_{j_{1},j_{2}}<1-q_{0}/2$ by the strong triangle inequality
	\eqref{strongtriangleinequality}. On the other hand, if $i_{1-q_{0}/2}(j_{1})\ne i_{1-q_{0}/2}(j_{2})$
	and $d_{j_{1},j_{2}}<1-q_{0}/2$, then we have by \eqref{strongtriangleinequality}
	and \eqref{q0def} that
	\[
		d_{i_{1-q_{0}/2}(j_{1}),i_{1-q_{0}/2}(j_{2})}\le\max\{d_{i_{q_{0}}(j_{1}),j_{1}},d_{j_{1},j_{2}},d_{j_{2},i_{q_{0}}(j_{2})}\}<1-q_{0}/2,%
	\]
	contradicting the definition \eqref{iqdef}. Therefore, we have \begin{equation}\label{eq:checksplithyps} i_{1-q_{0}/2}(j_{1})= i_{1-q_{0}/2}(j_{2})\iff d_{j_{1},j_{2}} < 1-q_{0}/2.\end{equation} Furthermore, we note that, for
	all $j\in P_{k}$, we have $2 d_{j,k}<2-q_{0}$, which means that (recalling
	\eqref{seps}, \eqref{kappaerrortermdij}, and \eqref{q0def}) we
	have
	\begin{equation}
		2-\lim_{\eps\downarrow0}\log_{\eps}\left(\tau_{\eps,j}-\ell_{\eps}\right)=2-q_{0}.\label{eq:newQ}
	\end{equation}
	Comparing this with \eqref{Qdef-1}, we see that the collection $\{(\tau_{\eps,j}-\ell_{\eps},x_{\eps,j})\}_{j\in [N]}$
	of space-time points satisfies the hypotheses of the theorem with
	the same $d_{ij}$s %
	but with $Q$
	replaced by $2-q_{0}$. Thus by \eqref{checksplithyps}, \propref{bustupindependent} applies
	and we obtain independent processes $w^{(1)},\dots,w^{(N)}$, each
	distributed identically to $w$, so that, whenever $j\in P_{k}$,
	we have
	\begin{equation}
		\lim_{\eps\downarrow0}\mathbf{E}(w^{(k)}-w)(\tau_{\eps,j}-\ell_{\eps},x_{\eps,j})^{2}=0.\label{eq:ukapproxsu}
	\end{equation}
	By the nontriviality of the partition $\{P_{1},\ldots,P_{N}\}$ we
	have $|P_{k}|<N$ for each $k$. Therefore, by the inductive hypothesis,
	we have
	\[
		(w^{(k)}(\tau_{\eps,j}-\ell_{\eps},x_{\eps,j}))_{j\in P_{k}}\xrightarrow[\eps\downarrow0]{\mathrm{law}}(\Gamma_{b,2-q_{0},j}(2-q_{0}))_{j\in P_{k}},
	\]
	with $b=w(0,x)$ independent of the randomness in the processes on
	the right side. Here we also used that $i_{(2-q_0-q)/2}(j)$ does not change when the minimum in \eqref{iqdef} is restricted to elements of $P_k$, since $P_k$ was defined so that this minimum will be an element of $P_k$ anyway. But since the family $(w^{(k)})_{k=1}^{N}$ is independent,
	as is the family $((\Gamma_{b,Q-q_{0},j}(Q-q_{0}))_{j\in P_{k}})_{k=1}^{N}$,
	this means that in fact
	\begin{equation}
		(w^{(k)}(\tau_{\eps,j}-\ell_{\eps},x_{\eps,j}))_{j\in1}^{N}\xrightarrow[\eps\downarrow0]{\mathrm{law}}(\Gamma_{b,2-q_{0},j}(2-q_{0}))_{j=1}^{N},\label{eq:inductiveafterpart}
	\end{equation}
	again with $b=w(0,x)$ independent of the randomness in the processes
	on the right side. Combining \eqref{utow}, \eqref{wtogamma1}, \eqref{ukapproxsu},
	\eqref{inductiveafterpart}, and the continuity of the SDE \eqref{dGamma-intro}--\eqref{Gammaic-intro}
	with respect to the initial condition, we obtain \eqref{multipointconvergence}.
\end{proof}

\appendix

\section{Convergence of discrete Markov martingales to continuous diffusions\label{appendix:MCtodiffusion}}
For the convenience of readers, we recall in this section a classical result on the convergence of Markov
chains to diffusions that is used in the paper. We use the formulation
and results given in \cite[Section 11.2]{SV06}.
\begin{thm}\label{thm:MCtodiffusion}
	Suppose that we have a sequence of numbers $\delta_{k}\downarrow0$,
	a sequence of discrete Markov martingales $(\{Y_{k}(m)\}_{m=A_1(k),\ldots,A_2(k)})_{k=1}^{\infty}$,
	and a continuous function $L:[A_1,A_2]\times\mathbf{R} \to\mathbf{R} $
	satisfying the following conditions:
	\begin{enumerate}
		\item \label{enu:icconv}The sequence of random variables $(Y_{k}(A_1(k)))$
		      converges in law to a random variable $X$ as $k\to\infty$.
		\item \label{enu:lipschitz}For each $q\in[A_1,A_2]$, the function $L(q,\cdot)$
		      is Lipschitz with the Lipschitz constant bounded above independent of $q$.
		\item We have $\delta_k m\in [A_1,A_2]$ for all $k\geq1$ and $m=A_1(k),\ldots,A_2(k)$,  and %
		      \[
			      \lim_{k\to\infty}\delta_{k}A_1(k)=A_1\qquad\text{and}\qquad\lim_{k\to\infty}\delta_{k}A_2(k)=A_2.
		      \]
		\item For each $R<\infty$, we have
		      \begin{equation}
			      \adjustlimits\lim_{\substack{k\to\infty\vphantom{|x|\le R}\\
					      \vphantom{A_1(k)\le m<A_2(k)}
				      }
			      }\sup_{\substack{|x|\le R\\
					      A_1(k)\le m<A_2(k)
				      }
			      }\left|\delta_{k}^{-1}\Var[Y_{k}(m+1)\mid Y_{k}(m)=x]-L(\delta_{k}m,x)\right|=0.\label{eq:VargoestoL}
		      \end{equation}

		\item There is a $p>2$ so that, for each $R<\infty$, we have
		      \begin{equation}
			      \sup_{\substack{k<\infty,|x|\le R\\
					      A_1(k)\le m<A_2(k)
				      }
			      }\delta_{k}^{-p/2}\mathbf{E}[(Y_{k}(m+1)-Y_{k}(m))^{p}\mid Y_{k}(m)=x]<\infty.\label{eq:momentboundforSV}
		      \end{equation}
	\end{enumerate}
	Let $(Y(q))_{q\in[A_1,A_2]}$ solve the stochastic differential equation
	\begin{align}
		\dif Y(q) & =L(q,Y(q))\,\dif B(q), \quad\quad q>A_1;\label{eq:dYq} \\
		Y(A_1)    & =X,\label{eq:YAa}
	\end{align}
	where $B(q)$ is a standard Brownian motion. Then we have
	\begin{equation}
		Y_{k}(A_2(k))\xrightarrow[k\to\infty]{\mathrm{law}}Y(A_2).\label{eq:YkBkconverges}
	\end{equation}
\end{thm}

\begin{proof}
	This is essentially an application of \cite[Theorem 11.2.3]{SV06}.
	Since that theorem is stated in a general form, we provide some details
	on how to check the conditions. First we note that although \cite[Theorem 11.2.3]{SV06}
	is stated for time-independent diffusions, it is trivial to add the
	time-dependence simply by considering the space-time processes of
	the form $\{(Y_{k}(m),\delta_km)\}_{m=A_{1}(k),\ldots,A_2(k)}$.
	Applying \cite[Theorem 11.2.3]{SV06} requires also knowing that the
	limiting martingale problem corresponding for \eqref{dYq}--\eqref{YAa}
	is well-posed. The SDE \eqref{dYq}--\eqref{YAa} has pathwise unique
	solutions by the standard theory and condition \enuref{lipschitz} in
	the statement of theorem. This implies that there are unique solutions
	for the martingale problem by results \cite{WY71I,WY71II} of Watanabe
	and Yamada; see \cite[Corollary 8.1.6]{SV06}. Finally, \cite[Theorem 11.2.3]{SV06}
	is stated for diffusions starting at time $0$ and lasting for all
	time; this can be adapted to our setting (a finite time interval with arbitrary starting time) by shifting time and extending
	the Markov chains to later times in some arbitrary way.

	The quantitative conditions for \cite[Theorem 11.2.3]{SV06} are \cite[(11.2.4)--(11.2.6)]{SV06}.
	In our setting, \cite[(11.2.4)]{SV06} is a consequence of \eqref{VargoestoL}
	(and the fact that there is no diffusion for the time process).
	The fact that we have assumed that each $Y_{k}(\cdot)$ is a martingale means
	that there is no drift for the space process, and of course the drift
	condition is satisfied trivially for the time process, so \cite[(11.2.5)]{SV06}
	is trivial in our setting. Finally, \cite[(11.2.6]{SV06} holds because,
	by \eqref{momentboundforSV} and Markov's inequality, we have for
	any fixed $\kappa>0$ that
	\begin{align*}
		\frac{1}{\delta_{k}}\mathbf{P}\left(\left|Y_{k}(m+1)-Y_{k}(m)\right|\ge\kappa\mid Y_{k}(m)=x\right) & \le\frac{\mathbf{E}\left[\left|Y_{k}(m+1)-Y_{k}(m)\right|^{p}\mid Y_{k}(m)=x\right]}{\delta_{k}\kappa^{p}}\\&\le C\delta_{k}^{p/2-1}\kappa^{-p}
	\end{align*}
	for a constant $C<\infty$, and the last quantity goes to $0$ as
	$k\to\infty$ since $p>2$ and $\delta_{k}\downarrow0$.

	Now condition~\enuref{icconv} and the proof of \cite[Theorem 11.2.3]{SV06}
	show that, if we define
	\[
		\overline{Y}_{k}(A_1+\delta_{k}[m-A_1(k)])=Y_{k}(m),\quad\quad m=A_1(k),\ldots,A_2(k),
	\]
	and extend $\overline{Y}_{k}$ to $[A_1,A_2]$ by linear interpolation
	(possibly extending it by a constant on the small interval $[A_1+\delta_k(A_2(k)-A_1(k)),A_2]$),
	then $\overline{Y}_{k}$ converges to $Y$ in distribution with respect
	to the uniform topology on continuous functions on $[A_1,A_2]$. Then \eqref{YkBkconverges}
	follows.
\end{proof}

\emergencystretch=2em

\end{document}